\documentclass[11pt,reqno]{amsart}
\usepackage[margin=1in,letterpaper]{geometry}
\usepackage{graphicx}
\usepackage{amssymb}
\usepackage{amsthm}
\usepackage{mathrsfs}
\usepackage{stmaryrd}
\usepackage{accents} 
\usepackage{enumitem} 
\usepackage{subcaption}

\usepackage[bookmarksopen,bookmarksdepth=2]{hyperref} 
\allowdisplaybreaks

\makeatletter\let\over\@@over\makeatother

\numberwithin{equation}{section}
\theoremstyle{plain} 
\newtheorem{theorem}{Theorem}[section] 
\newtheorem{proposition}[theorem]{Proposition} 
\newtheorem{corollary}[theorem]{Corollary}
\newtheorem{lemma}[theorem]{Lemma}
\theoremstyle{remark}
\newtheorem{remark}[theorem]{Remark}
\theoremstyle{definition}
\newtheorem{definition}[theorem]{Definition}

\newcommand{\be}{\begin{equation}}
\newcommand{\ee}{\end{equation}}%
\newcommand{\bse}{\begin{subequations}}
\newcommand{\ese}{\end{subequations}}

\newcommand{\dist}{\operatorname{dist}}
\newcommand{\realpart}{\operatorname{Re}}

\newcommand{\signum}[1]{\operatorname{sgn}{#1}}

\newcommand{\sech}{\operatorname{sech}} 
\newcommand{\jump}[1]{\left\llbracket{#1}\right\rrbracket}
\newcommand{\range}{\operatorname{rng}}
\newcommand{\kernel}{\operatorname{ker}}
\newcommand{\linspan}{\operatorname{span}}

\newcommand{\ind}{\operatorname{ind}}

\newcommand{\FSproj}{\mathcal{Q}}

\newcommand{\R}{\mathbb{R}} 
\newcommand{\placeholder}{\,\cdot\,}
\newcommand{\maps}{\colon}

\newcommand{\n}[2][]{#1\lVert #2 #1\rVert}
\newcommand{\abs}[2][]{#1\lvert #2 #1\rvert}
\newcommand{\dell}{\partial}
\newcommand{\ina}{\textup{~in~}}
\newcommand{\ona}{\textup{~on~}}

\newcommand{\asa}{\textup{~as~}}

\newcommand{\bdd}{\mathrm{b}}       
\newcommand{\loc}{{\mathrm{loc}} }     

\newcommand\A{\mathcal A}    
\newcommand\B{\mathcal B}    
\providecommand{\G}{}        
\renewcommand\G{\mathcal G}  
\newcommand\F{\mathscr F}    

\newcommand\Xspace{\mathscr X}
\newcommand\Yspace{\mathscr Y}

\newcommand{\genU}{\mathcal U}

\newcommand{\cm}{{\mathscr C}}  
\newcommand{\iftcm}{{\mathscr C}}       
\newcommand{\borderedcm}{{\mathscr K}}            

\newcommand\flowforce{\mathscr{H}}
\newcommand\limL{\mathscr{L}}
\newcommand\prineigenvaluepm{\sigma_{0}^\pm}
\newcommand\prineigenvalue{\sigma_{0}}
\newcommand\essspec{\sigma_{\mathrm{ess}}}

\newcommand\fluidD{\mathscr{D}}
\newcommand\fluidS{\mathscr{S}}

\newcommand\Lip{\operatorname{Lip}}
\newcommand\lambdastar{\lambda_+} 
\newcommand\hplus{h_+} 
\newcommand\fbu{w} 

\newcommand\invec{\mu}

\begin{document}

\title[Global bifurcation of fronts]{Global bifurcation for monotone fronts of elliptic equations}
\date{\today}

\author[R. M. Chen]{Robin Ming Chen}
\address{Department of Mathematics, University of Pittsburgh, Pittsburgh, PA 15260} 
\email{mingchen@pitt.edu}  

\author[S. Walsh]{Samuel Walsh}
\address{Department of Mathematics, University of Missouri, Columbia, MO 65211} 
\email{walshsa@missouri.edu}

\author[M. H. Wheeler]{Miles H. Wheeler}
\address{Department of Mathematical Sciences, University of Bath, Bath BA2 7AY, United Kingdom}
\email{mw2319@bath.ac.uk}

\begin{abstract} In this paper, we present two results on global continuation of monotone front-type solutions to elliptic PDEs posed on infinite cylinders.  This is done under quite general assumptions, and in particular applies even to fully nonlinear equations as well as quasilinear problems with transmission boundary conditions.  Our approach is rooted in the analytic global bifurcation theory of Dancer \cite{dancer1973globalsolution,dancer1973globalstructure} and Buffoni--Toland \cite{buffoni2003analytic}, but extending it to unbounded domains requires contending with new potential limiting behavior relating to loss of compactness.     We obtain an exhaustive set of alternatives for the global behavior of the solution curve that is sharp, with each possibility having a direct analogue in the bifurcation theory of second-order ODEs.

As a major application of the general theory, we construct global families of internal hydrodynamic bores.  These are traveling front solutions of the full two-phase Euler equation in two dimensions. The fluids are confined to a channel that is bounded above and below by rigid walls, with incompressible and irrotational flow in each layer.  Small-amplitude fronts for this system have been obtained by several authors.   We give the first large-amplitude result in the form of continuous curves of elevation and depression bores.  Following the elevation curve to its extreme, we find waves whose interfaces either overturn (develop a vertical tangent) or become exceptionally singular in that the flow in both layers degenerates at a single point on the boundary.   For the curve of depression waves, we prove that either the interface overturns or it comes into contact with the upper wall. 
\end{abstract}

\maketitle

\setcounter{tocdepth}{1}
\tableofcontents

\section{Introduction}\label{sec_intro}

Let $\Omega = \mathbb{R} \times \Omega^\prime$ be an unbounded cylinder whose base $\Omega^\prime \subset \mathbb{R}^{n-1}$ is bounded.  For simplicity, assume that $\Omega$ is connected with a $C^{k+2+\alpha}$ boundary $\partial \Omega = \Gamma_0 \cup \Gamma_1$, for a fixed $k \geq 0$ and $\alpha \in (0,1)$ and such that $\Gamma_0 \cap \Gamma_1 = \emptyset$.   Note that $\Gamma_0 = \mathbb{R} \times \Gamma_0^\prime$ and $\Gamma_1 = \mathbb{R} \times \Gamma_1^\prime$, for some $\Gamma_0^\prime, \Gamma_1^\prime \subset \partial \Omega^\prime$.  
Points in $\Omega$ will be denoted $(x,y)$, where $x \in \mathbb{R}$ and $y \in \Omega^\prime$.  

We will study nonlinear PDEs set on $\Omega$ taking the quite general form
\begin{equation}
 \label{fully nonlinear elliptic pde}
  \left\{ \begin{aligned}
    \mathcal{F}(y, u, \nabla u, D^2 u,  \Lambda)  & = 0  \qquad \textrm{in } \Omega, \\
    \G(y, u, \nabla u, \Lambda) & = 0 \qquad  \textrm{on } \Gamma_1, \\
    u & = 0 \qquad \textrm{on } \Gamma_0,
  \end{aligned} \right.
\end{equation}
where $\Lambda \in \R^m$ is a collection of parameters, and $\mathcal{F}$ and $\mathcal{G}$ are regular enough that  
\begin{equation}
  \label{gen regularity elliptic coef} 
  \begin{aligned}
    (z,\xi,r,\Lambda) &\longmapsto \left(   \mathcal F(\placeholder, z, \xi, r, \Lambda), \, \mathcal G(\placeholder, z, \xi, \Lambda) \right)  \\
    \mathcal{V}    & \longrightarrow     C^{k+\alpha}(\overline{\Omega^\prime}) \times C^{k+1+\alpha}(\Gamma_1^\prime)
  \end{aligned}
  \qquad  \textrm{is real analytic,}
\end{equation}
for some open set $\mathcal{V} \subset \mathbb{R} \times \mathbb{R}^n \times \mathbb{S}^{n\times n} \times \mathbb{R}^m$. We assume that \eqref{fully nonlinear elliptic pde} is uniformly elliptic with a uniformly oblique boundary condition on $\Gamma_1$: there exist $c_1, c_2 > 0$ such that
\begin{align}
  \label{ellipticity assumption}
  \mathcal F_{r^{ij}}(\placeholder ,z,\xi,r,\Lambda) \eta_i \eta_j \ge c_1 \abs\eta^2 \quad \textrm{on } \overline{\Omega}, \qquad
    \mathcal G_{\xi^i}(\placeholder,z,\xi,\Lambda) \nu_i > c_2 \quad \textrm{on } \Gamma_1^\prime 
\end{align}
for all $(z, \xi, r, \Lambda) \in \mathcal{V}$ and $\eta \in \mathbb{R}^n$.  Here  $\nu = \nu(y)$ is the outward pointing normal to $\Omega$ at $(x,y) \in \Gamma_1$.

We call a solution $(u, \Lambda)$ of \eqref{fully nonlinear elliptic pde} a \emph{front} if $u \in C^{k+2+\alpha}_\bdd(\overline{\Omega})$ has well-defined point-wise limits as $x \to -\infty$ and $x \to +\infty$. Anticipating applications to water waves, we call these limits the upstream and downstream states, respectively.  From the structure of the equation, one can further infer that $\partial^\beta u$ will have uniform limits as $x \to\pm\infty$ for any $|\beta| \leq k+2$, and hence that the upstream and downstream states are $x$-independent solutions of \eqref{fully nonlinear elliptic pde}. We say a front is \emph{monotone} if $\partial_x u \leq 0$ (or $\partial_x u \geq 0$) in $\Omega$, and \emph{strictly monotone} if $\partial_x u < 0$ (or $\partial_x u > 0$) in $\Omega \cup \Gamma_1$.   
Fronts are studied in a vast array of physical settings, including the spread of invasive species or alleles in biology, and, more broadly, phase transitions in reaction-diffusion equations.  

Our purpose in this work is two-fold.  First, we develop a systematic approach to constructing large monotone fronts through analytic global bifurcation theory.  For this, substantial new analysis is needed to overcome a host of issues stemming from the unboundedness of the domain.  Ultimately, we obtain a list of alternatives for the limiting behavior of the bifurcation curve that is sharp, with each alternative having an analogue for second-order ODEs.  At the same time, we impose only minimal conditions on the structure of the equation, making the resulting machinery quite robust.

The paper's second part concerns a longstanding open problem in water waves.  Using the general theory, we construct many curves of large-amplitude hydrodynamic bores of elevation and depression.  These are, respectively, strictly monotone increasing and decreasing front solutions of the full two-phase free boundary Euler equations.  Of special significance is that,  in the limit along the curve of elevation bores, the waves overturn (the free surface develops a vertical tangent) or else develop a highly degenerate singularity.  Overhanging steady gravity water waves were observed numerically nearly 40 years ago, but a proof of their existence continues to be one of the most sought after results in the field.   Following the curve of depression bores, we find that the free boundary either overturns or contacts the upper wall.  In the latter case, the flow is expected to approach a \emph{gravity current} --- a type of traveling wave that has been investigated extensively in fluid mechanics, both experimentally and computationally \cite{simpson1982gravity}.  While formal analytical studies date back at least to the famous work of von K\'arm\'an \cite{vonkarman1940engineer} in 1940, gravity currents have never been constructed rigorously.
 
We begin in the next section by describing the global continuation theorems in the general setting.  The application to water waves is then discussed in Section~\ref{intro bores section}.  

\subsection{Statement of abstract  results} \label{statement abstract results section}

To simplify the notation, we introduce the spaces 
\[ 
\Xspace := \left\{ u \in C^{k+2+\alpha}(\overline{\Omega}) : u|_{\Gamma_0} = 0 \right\}, \qquad 
\Yspace = \Yspace_1 \times \Yspace_2 := C^{k+\alpha}(\overline{\Omega}) \times C^{k+1+\alpha}(\Gamma_1).
\]
Note that elements of $\Xspace$ and $\Yspace$ are \emph{locally} H\"older continuous.  By convention, we say $u_n \to u$ in $C^{\ell+\beta}_\loc$ provided $u_n \to u$ in $C^{\ell+\beta}$ on any compact subset of $\overline{\Omega}$.  On the other hand, we denote by $\Xspace_\bdd$ and $\Yspace_\bdd$  the Banach spaces of functions whose corresponding H\"older norms are finite.  Finally, let $\Xspace^\prime$ be the subspace of $\Xspace$ consisting of functions that are independent of $x$, and likewise for $\Yspace^\prime$.    

The PDE \eqref{fully nonlinear elliptic pde} can then be rewritten as the abstract operator equation
\[ \mathscr{F}(u, \Lambda) = 0,\]
and from \eqref{gen regularity elliptic coef} it follows that $\mathscr{F}$ is a real-analytic mapping $\genU \subset \Xspace_\bdd \times \mathbb{R}^m \to \Yspace_\bdd$, for some open set $\genU$ determined by $\mathcal{V}$.

For a monotone front $(u,\Lambda)$, we define the \emph{transversal linearized operator at $x = \pm\infty$} to be the bounded linear mapping 
\begin{equation}
  \limL_\pm^\prime(u,\Lambda) \colon \Xspace^\prime \to \Yspace^\prime \qquad w \longmapsto \lim_{x \to \pm\infty} \mathscr{F}_u(u,\Lambda)w.  \label{def limiting transversal lin op} 
\end{equation} 
Note that these limits exist by the discussion above, and, in particular, $\limL_\pm^\prime(u,\Lambda)$ is a linear elliptic operator with mixed boundary conditions posed on $\Omega^\prime$.   One can then show that $\limL_\pm^\prime(u,\Lambda)$ has a principal eigenvalue that we will denote by $\prineigenvaluepm(u,\Lambda)$; see Appendix~\ref{principal eigenvalue appendix}.  

Suppose first that we are given a single strictly monotone front $(u,\Lambda)$.  As the system \eqref{fully nonlinear elliptic pde} is invariant under translation in $x$, a simple elliptic regularity argument shows that $\partial_x u$ lies in $\kernel{\F_u(u,\Lambda)}$.   Let us assume that the kernel is exactly one dimensional:
\begin{equation}
  \ker \F_u(u,\Lambda) = \linspan \{ \partial_x u \}.  \label{kernel assumption}\tag{H1} 
\end{equation}
It is well known that the Fredholm properties of $\F_u(u, \Lambda)$ are determined by the limiting linearized operators upstream and downstream (see, for example, \cite{volpert2011book1,volpert2003degree}).  In this work, we focus on the situation where 
\begin{equation}
  \prineigenvalue^-(u,\Lambda),~\prineigenvalue^+(u,\Lambda) < 0.  \label{spectral assumption}\tag{H2} 
\end{equation}
Loosely speaking, this corresponds to the upstream and downstream states being spectrally stable in a sense to be discussed shortly.    

We will show that  \eqref{spectral assumption} implies in particular that $\F_u(u, \Lambda)$ is Fredholm index $0$. Generically, then, one expects the zero-set of $\mathscr{F}$ to be locally a curve provided that the parameter space is two dimensional ($m = 2$), though the solutions on it need not be monotone nor even fronts.  The next theorem says something much stronger: there exists a \emph{global} curve of strictly monotone fronts.  This curve is  maximal in a certain sense among all (locally) analytic curves containing $(u,\Lambda)$, and its limiting behavior is characterized by a set of four alternatives, all of which are realizable.

\begin{theorem}[Global implicit function theorem] \label{global ift}   Consider the elliptic PDE \eqref{fully nonlinear elliptic pde} with two parameters $\Lambda = (\lambda, \mu)$.  Suppose that $(u_0,\lambda_0,\mu_0) \in \genU$ is a strictly monotone front solution to \eqref{fully nonlinear elliptic pde} satisfying the nondegeneracy condition \eqref{kernel assumption}, spectral condition \eqref{spectral assumption}, and transversality condition
  \begin{equation}
    \mathscr{F}_\mu(u_0, \lambda_0, \mu_0) \not\in \range{\mathscr{F}_u(u_0, \lambda_0, \mu_0)}.\label{transversality condition} \tag{H3} 
  \end{equation}
Then there exists a global curve $\iftcm \subset \genU$ of strictly monotone front solutions with the parameterization
\[ \iftcm := \left\{ ( u(s),  \lambda(s),\mu(s) ) : s \in \mathbb{R}  \right\} \subset \mathscr{F}^{-1}(0)\]
for some continuous $\mathbb{R} \ni s \longmapsto (  u(s), \lambda(s), \mu(s)) \in \genU$ with $(u(0), \lambda(0),\mu(0)) = (u_0, \lambda_0, \mu_0)$.
\begin{enumerate}[label=\rm(\alph*)]
\item \label{gen ift alternatives} \textup{(Alternatives)} As $s \to +\infty$, one of four alternatives must occur:       
\begin{enumerate}[label=\rm(A\arabic*)]
    \item \textup{(Blowup)} \label{gen blowup alternative}
      The quantity 
      \begin{align}
        \label{gen global blowup}
        N(s):= \n{u(s)}_{\Xspace} + |\Lambda(s)| + \frac 1{\dist((u(s),\Lambda(s)), \, \dell \genU)} \longrightarrow \infty.
      \end{align}
    \item \label{gen hetero degeneracy} \textup{(Heteroclinic degeneracy)} There exist a sequence $s_n \to +\infty$ and a sequence $x_n \to \pm \infty$ with
      \[ \left(u(s_n)(\placeholder + x_n,\placeholder),\,  \Lambda(s_n) \right) \longrightarrow (u_*,\Lambda_*) \textrm{ in } C^{k+2}_\loc(\overline\Omega) \times \mathbb{R}^2\] 
      for some monotone front solution $(u_*,\Lambda_*) \in \genU$, but the three limiting states
     \begin{equation*}
       \lim_{x \to \mp\infty} u_*(x, \placeholder),
       \quad 
       \lim_{n \to \infty} \lim_{x \to +\infty} u(s_n)(x, \placeholder),
       \quad 
       \lim_{n \to \infty} \lim_{x \to -\infty} u(s_n)(x, \placeholder),
     \end{equation*}    
     are all distinct.
      \item \label{gen ripples} \textup{(Spectral degeneracy)}    There exists a sequence $s_n \to +\infty$ with $\sup_{n} N(s_n) < \infty$ so that 
      \[
          \prineigenvalue^-(u(s_n), \Lambda(s_n)) \to 0 \qquad \textrm{or} \qquad \prineigenvalue^+(u(s_n), \Lambda(s_n)) \to 0.  
      \]
\item \label{gen loop} \textup{(Loop)}  $\iftcm$ is a closed loop in that  $s\mapsto (u(s), \Lambda(s))$ is $T$-periodic for some $T > 0$.
        \end{enumerate}
\item \label{gen ift other direction alternatives} One of the above alternatives must also occur as $s \to -\infty$.   If in either of these limits the loop alternative \ref{gen loop} happens, then clearly it happens in both limits.  

\item \label{gen ift analyticity} \textup{(Analyticity)} At each parameter value $s \in \mathbb{R}$, $\iftcm$ admits a local real-analytic reparameterization.

\item \label{gen ift maximality}  \textup{(Maximality)} If $\mathscr{J} \subset \genU$ is any locally real-analytic curve of monotone front solutions to \eqref{fully nonlinear elliptic pde} that contains $(u_0,\Lambda_0)$ and along which \eqref{spectral assumption} holds, then $\mathscr{J} \subset \iftcm$. 
\end{enumerate}
\end{theorem}

Global implicit function theorems have been used by many authors, including for instance very recent work~\cite{garcia2022pairs} on pairs of rotating vortex patches. Since a preprint of this article first appeared, Theorem~\ref{global ift} above has been adapted to show the existence of global curves of solutions to the Boussinesq $abcd$ system~\cite{chen2022abcd}. We remark that, at first glance, one might be concerned that the loop alternative \ref{gen loop} occurs with $\cm$ being just a family of translates of $u_0$.  In fact, this degenerate scenario is prevented in our construction through the use of a functional \eqref{def C functional} that breaks the translation symmetry.   

The next theorem addresses the related problem of continuing a given ``local'' curve $\cm_\loc$ of strictly monotone front solutions.  It is most natural in this setting to consider the one-parameter case ($m = 1$), so we write $\Lambda = \lambda \in \mathbb{R}$.   Usually, one obtains $\cm_\loc$ through a preliminary local bifurcation argument.  A common scenario on unbounded domains is that $\cm_\loc$ originates from an $x$-independent solution to \eqref{fully nonlinear elliptic pde} that is singular in the sense that the linearized operator there fails to be Fredholm.  With that in mind, suppose that $\cm_\loc$ admits the $C^0$ parameterization
\[ 
\mathscr{C}_\loc = \left\{ \left( u(\varepsilon), \lambda(\varepsilon) \right) : 0 < \varepsilon < \varepsilon_0 \right\}  \subset \genU,
\]
where 
\begin{equation}
  \begin{gathered}
    (u(\varepsilon), \lambda(\varepsilon)) \to (u_0,\lambda_0) \in \genU \quad \text{as } \varepsilon \to {0+} \\ 
    \textrm{ and } \prineigenvalue^+(u_0,\lambda_0)=0 \text{ or } \prineigenvalue^-(u_0,\lambda_0) = 0.
  \end{gathered}
  \label{local singular assumption} 
  \tag{H4} 
\end{equation}
Our main global bifurcation result is then the following.

\begin{theorem}[Global bifurcation] \label{general global bifurcation theorem}
Consider the elliptic PDE \eqref{fully nonlinear elliptic pde} with a single parameter $\Lambda = \lambda$.  Let $\cm_\loc$ be a curve of strictly monotone front solutions bifurcating from a singular point as in \eqref{local singular assumption}.  Assume that at each $(u,\lambda) \in \cm_\loc$, the nondegeneracy \eqref{kernel assumption} and spectral \eqref{spectral assumption} conditions hold.  

Then, possibly after translation, $\mathscr{C}_\loc$ is contained in a global curve of strictly monotone front solutions $\mathscr{C} \subset \genU$, parameterized as 
\[ \mathscr{C} := \left\{ \left(u(s), \lambda(s) \right) : 0 < s < \infty  \right\} \subset \mathscr{F}^{-1}(0)\]
for some continuous $\mathbb{R}_+ \ni s \longmapsto (u(s), \lambda(s)) \in \genU$ with the  properties enumerated below.
  \begin{enumerate}[label=\rm(\alph*)]
  \item \label{gen alternatives} As $s \to \infty$ either the blowup \ref{gen blowup alternative}, heteroclinic degeneracy \ref{gen hetero degeneracy}, or spectral degeneracy \ref{gen ripples} alternative occurs.  
  \item \label{gen reparam} The curve $\cm$ is locally real analytic and maximal in the sense of Theorem~\ref{global ift}\ref{gen ift maximality}.
  \item \label{gen reconnect} For all $s$ sufficiently large, $(u(s),\lambda(s)) \not\in \cm_\loc$.  In particular, $\cm$ is not a closed loop.
  \end{enumerate}
\end{theorem}

Some extended discussion of the assumptions and conclusions of Theorem~\ref{global ift} and Theorem~\ref{general global bifurcation theorem} is given below.  

\subsubsection*{On the hypotheses}

In developing this theory we have endeavored to make no structural hypotheses on the system beyond analyticity \eqref{gen regularity elliptic coef} and ellipticity \eqref{ellipticity assumption}.    Moreover, we do not impose compactness requirements on $\F$ except at the given front or local curve.  This contrasts dramatically with the analytic global bifurcation theory in \cite{buffoni2003analytic}, for example, where the zero-set $\F^{-1}(0)$ must be locally compact and the restriction of $\F$ to it must be Fredholm index $0$.  Degree theoretic global bifurcation typically makes the even stronger assumption that $\F$ is Fredholm index $0$ and locally proper throughout its domain of definition (see, for example, \cite{rabinowitz1971some}).  These hypotheses are reasonable for elliptic PDEs set on a compact region; for classical solutions, they usually follow from Schauder theory.  

On unbounded domains, however, it becomes a major analytical challenge to prove that a nonlinear elliptic operator is locally proper.  Indeed, as we see in alternative \ref{gen hetero degeneracy}, the global curve may not be locally pre-compact.  Nor can Fredholmness be taken for granted: as hinted at by \eqref{local singular assumption}, we often wish to continue curves that bifurcate from singular points where $0$ is in the essential spectrum of $\F_u$.    Moreover, the Fredholm index is determined by the spectral properties of the linearized operators at infinity, which will in principle change in physically meaningful ways as we traverse $\cm$.    These considerations argue strongly that the loss of compactness be treated as an \emph{alternative}.  This decision shifts the difficulty from verifying properness and Fredholmness to classifying qualitatively how they might fail.   

To motivate the spectral condition \eqref{spectral assumption}, it is instructive to look at two specific classes of equation.  Consider first the semilinear  Robin problem 
\begin{equation}
 \label{intro toy problem} 
\left\{ 
  \begin{aligned}
    \Delta u & = 0 & \qquad & \textrm{in } \Omega \\
    u_y - u + g(u,\Lambda) & = 0 & \qquad & \textrm{on } \Gamma_1 \\
    u &=  0 & \qquad &\textrm{on } \Gamma_0,
  \end{aligned} \right.
\end{equation}
set on the infinite cylinder $\Omega := \mathbb{R} \times (0,1)$ with upper boundary $\Gamma_1 := \{ y = 1\}$ and lower boundary $\Gamma_0 := \{ y =0 \}$.  Here $g = g(z,\Lambda)$ is a smooth nonlinearity with one or two parameters.  This equation is \emph{reversible} (invariant under reflection in $x$) and \emph{variational} (its solutions are formally critical points of a certain functional).  It can be rewritten as the infinite-dimensional spatial Hamiltonian system 
\begin{equation}
  \partial_x \begin{pmatrix} u \\ v \end{pmatrix}  
    = J \delta \flowforce, \qquad J := \begin{pmatrix} 0 & -1 \\ 1 & 0 \end{pmatrix} \label{example hamiltonian eq} 
\end{equation}
where $v := \partial_x u$, and the Hamiltonian $\flowforce$ is given by 
\[ \flowforce(u,v;x) := \int_0^1 \Big( \frac {u_y^2-v^2}2 - uu_y + g(u,\Lambda)u_y \Big)\, dy. \]

Suppose now that $(u,\Lambda)$ is a monotone front solution to \eqref{intro toy problem} with limiting states $U_\pm$.  For the Hamiltonian formulation \eqref{example hamiltonian eq}, this corresponds to a heteroclinic connection between the rest points $(U_-, 0)$ and $(U_+,0)$.  It is easily seen that $\sigma \in \mathbb{C}$ is an eigenvalue of $J \delta^2 \flowforce(U_\pm)$ if and only if $\sigma^2$ is an eigenvalue of $-\limL_\pm^\prime(u,\Lambda)$.  Thus the spectral condition \eqref{spectral assumption} is equivalent to the hyperbolicity of the equilibria for the spatial dynamical system. 
 
  As a second example, consider the time-dependent reaction-diffusion equation 
 \begin{equation}
   \partial_t u - \Delta u + b(y,\Lambda) \partial_x u = f(u, \Lambda), \label{dynamic reaction diffusion eq} 
 \end{equation}
 where $f = f(z,\Lambda)$ is a smooth, parameter-dependent semilinear term.   For discussion purposes, let $\Gamma_1 = \emptyset$ so that the boundary conditions are simply homogeneous Dirichlet.   A traveling front with wave speed $c \in \mathbb{R}$ will then satisfy the elliptic PDE
 \begin{equation}
   \Delta u + \left(c-b(y,\Lambda)\right) \partial_x u + f(u,\Lambda) = 0. \label{reaction diffusion eq} 
 \end{equation}
Berestycki and Nirenberg \cite{berestycki1992travelling} give a rather comprehensive treatment of this problem in the case that $b$ and $f$ are independent of $\Lambda$ and $c$ serves as the parameter.  This assumption, however, considerably simplifies the global behavior as the set of $x$-independent solutions and the corresponding limiting transversal linearized operators will be fixed.  

Let $(u,\Lambda)$ be a monotone front solution to \eqref{reaction diffusion eq} with $U_-$ and $U_+$ its upstream and downstream states.  As they are necessarily independent of $x$, $U_\pm$ are both stationary solutions of the time-dependent problem \eqref{dynamic reaction diffusion eq}.   The limiting linearized operator at $x = \pm\infty$ is  
\[ 
\limL_\pm(u,\Lambda) =  \Delta + \left(c - b(y, \Lambda) \right) \partial_x + f_z(U_\pm, \Lambda).
\]
In Proposition~\ref{essential spec proposition}, we prove that the spectral assumption \eqref{spectral assumption} is equivalent to the essential spectrum of $\limL_\pm$ being properly contained in the left complex half-plane $\mathbb{C}_-$.  Thus both the upstream and downstream states are spectrally stable as steady state solutions of \eqref{dynamic reaction diffusion eq}.  

As this reasoning shows, for reaction-diffusion equations of the form \eqref{reaction diffusion eq}, our theory is tailored to so-called \emph{bistable} or \emph{Allen--Cahn-type nonlinearities}.  They are referred to as ``Type C'' by Berestycki and Nirenberg \cite{berestycki1992travelling}, who also impose \eqref{spectral assumption} in several of their results regarding this case.  By exploiting more fully the structure of the equation --- especially the damping effect of the first-order term --- they are able to treat other classes of nonlinearity as well.  However, these arguments will not hold in the general context of \eqref{fully nonlinear elliptic pde}.  Indeed, one cannot expect monotonicity to persist without \eqref{spectral assumption}.  We note that the first-order term in \eqref{reaction diffusion eq} plays a similarly important role in the existence and multiplicity results of Bakker, van den Berg, and Vandervorst~\cite{bakker2018homology}, which are based on topological methods from dynamical systems.

\subsubsection*{On the alternatives}

Let us now discuss in somewhat more detail the limiting behavior along the curve.   Blowup \ref{gen blowup alternative} is commonly thought of as the most desirable alternative as it indicates that $\cm$ includes arbitrarily large solutions or else limits to some pathological behavior characterized by the boundary of $\genU$.  

\begin{figure}
  \centering
  \includegraphics[scale=1,page=1]{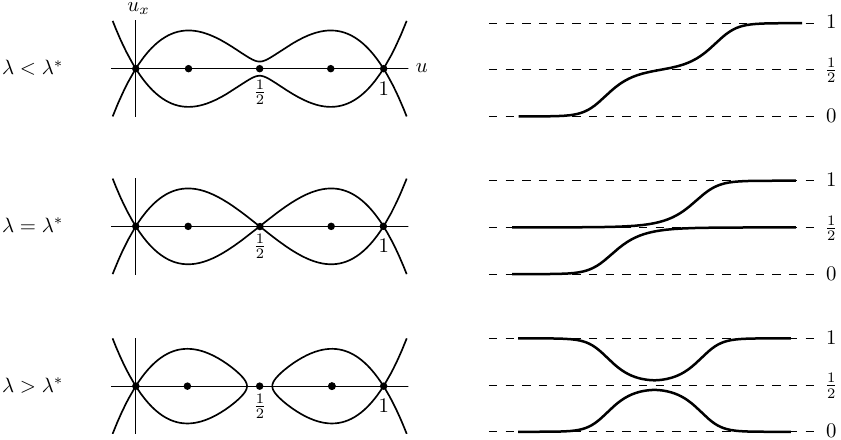}

  \caption{An ODE $u_{xx} = f(u,\lambda)$ which experiences a heteroclinic degeneracy \ref{gen hetero degeneracy}. The left shows phase portraits for fixed values of $\lambda$, and the graphs on the right are selected heteroclinic/homoclinic solutions as functions of the independent variable $x$. For $\lambda < \lambda^*$, there is a monotone increasing heteroclinic orbit connecting the constant solutions $u=0$ and $u=1$. At $\lambda = \lambda^*$, however this orbit degenerates into two, one connecting $u=0$ with $u=1/2$, and another connecting $u=1/2$ to $u=1$. For $\lambda > \lambda^*$, there are instead only homoclinic orbits connecting $u=0$ to itself and similarly for $u=1$.}
  \label{hetero degen phase portrait figure} 
\end{figure}

\begin{figure}
  \centering
  \includegraphics[scale=1,page=2]{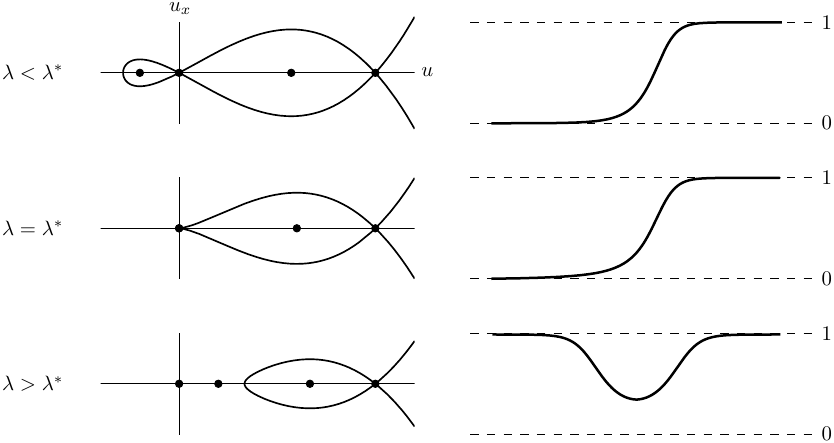}

  \caption{An ODE $u_{xx} = f(u,\lambda)$ which experiences a spectral degeneracy \ref{gen ripples}. The left shows phase portraits for fixed values of $\lambda$, and the graphs on the right are selected heteroclinic/homoclinic solutions as functions of the independent variable $x$. For $\lambda < \lambda^*$, there is a monotone increasing heteroclinic orbit connecting the constant solutions $u=0$ and $u=1$. This orbit persists for $\lambda = \lambda^*$, but the Jacobian matrix at $u=0$ ceases to be invertible and the orbit no longer decays exponentially as $x \to -\infty$.
  For $\lambda > \lambda^*$, the heteroclinic orbit degenerates into a homoclinic orbit to $u=1$, while $u=0$ becomes a center.}
  \label{ripples phase portrait figure} 
\end{figure}

The intuition for the heteroclinic degeneracy alternative \ref{gen hetero degeneracy} is best explained in terms of bifurcation of ODEs.  Consider a scenario where a heteroclinic orbit between two equilibria breaks down and a new heteroclinic is born that connects one of them to an intermediate rest point as in Figure~\ref{hetero degen phase portrait figure}.  This sort of breakdown can also occur for the PDE \eqref{fully nonlinear elliptic pde}, where the upstream and downstream states play the role of the equilibria in the original heteroclinic orbit, and the intermediate equilibrium is a distinct $x$-independent solution. The sequence of translations in \ref{gen hetero degeneracy} shifts this intermediate state off to $\pm\infty$, so that locally we have convergence to a new front.  There is a related phenomenon for solitary waves (that is, homoclinic orbits) wherein the solution broadens into an infinitely long ``table top.'' This has been observed numerically in \cite{turner1988broadening}, for example.  

Next, consider the spectral degeneracy alternative \ref{gen ripples}.  As mentioned above, \eqref{spectral assumption} is equivalent to the essential spectrum of the limiting linearized operators being properly contained in $\mathbb{C}_-$.  By standard elliptic theory, the principal eigenvalues are real and lie strictly to the right of the rest of the spectrum of $\limL_\pm^\prime$.   Thus spectral degeneracy indicates resonance:   the essential spectrum of the linearized problem upstream or downstream moves through the origin.  This results in a loss of semi-Fredholmness and potentially relative compactness of the zero-set.   For reaction-diffusion equations, it corresponds to the onset of ``essential instability'' \cite{sandstede1999essential,sandstede2001essential}.  To continue any further would require a detailed study of this limit and the resulting linearized problem, perhaps using center manifold reduction techniques.  There is little hope of successfully carrying out such an argument without making additional structural hypotheses.  Even for ODEs, heteroclinic orbits may cease to exist beyond spectral degeneracy; see Figure~\ref{ripples phase portrait figure}.   

While the statements of alternatives \ref{gen hetero degeneracy} and \ref{gen ripples} appear somewhat complicated, often they can be drastically simplified.  This is especially true if there are conserved quantities for the problem, such as the Hamiltonian $\flowforce$ in \eqref{intro toy problem},  and if the set of $x$-independent solutions can be completely characterized.  One then obtains a finite set of conditions on $\Lambda$ that are necessary for a heteroclinic condition to exist; these we call the   \emph{conjugate flow equations} in reference to Benjamin's seminal work \cite{benjamin1971unified}.  They play a central role in both of our applications.   

Finally, let us note that there is an important (though subtle) distinction to be made regarding  \ref{gen loop}.  We consider $\iftcm$ to be a closed loop if it has a global $C^0$ parameterization that is periodic and locally real analytic.  Clearly, this cannot be true for $\cm$ in Theorem~\ref{general global bifurcation theorem} due to \eqref{local singular assumption}.  However, it can happen that $\cm$ \emph{reconnects} to the singular point $(u_0,\lambda_0)$ as $s \to +\infty$; this possibility is captured by the spectral degeneracy alternative \ref{gen ripples}.  One can imagine this occurring, for instance,  if there are multiple local bifurcation curves branching from the same singular point --- a scenario that can be ruled out with a complete account of the monotone solutions nearby.  For the applications presented in this paper, we accomplish such a characterization via center manifold reduction.  

\subsection{Large fronts for semilinear Robin problems} \label{intro robin section}

We present two applications of the general theory.  The first is a detailed study of the semilinear Robin problem \eqref{intro toy problem} for a class of nonlinearities $g$ related to double well potentials.   A similar family of equations was studied by Rabinowitz \cite{rabinowitz1994solutions} through global variational techniques.  Using a center manifold reduction, we construct a local curve of perturbative solutions.  The hypotheses of Theorem~\ref{general global bifurcation theorem} are then verified, furnishing a global curve of strictly monotone fronts.  By proving uniform bounds and exploiting the Hamiltonian structure of the problem, we are able to link the qualitative properties of $g$ to each of the alternatives \ref{gen blowup alternative}, \ref{gen hetero degeneracy}, and \ref{gen ripples}.  See Theorem~\ref{toy global bifurcation theorem} and Proposition~\ref{realization proposition} for precise statements of these results.

\subsection{Large-amplitude bores} \label{intro bores section}

\begin{figure}[tb]
  \centering
  \includegraphics[scale=1]{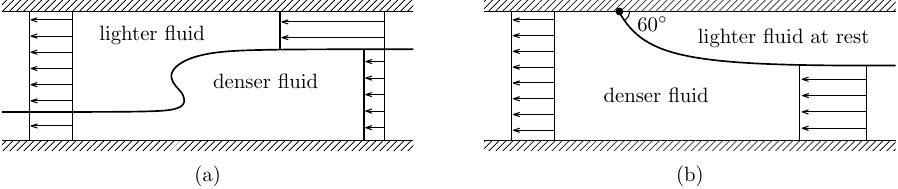}
  \caption{Limiting configurations along the families of monotone bores from Theorem~\ref{global bore theorem}. The strictly increasing bores may overturn, meaning that the interface between the lighter and denser fluids develops a vertical tangent.  Depicted in (a) is the expected state sometime after this has occurred.  On the other hand, the strictly decreasing bores either overturn or else the distance between the interface and upper wall limits to $0$.  The latter case is illustrated in (b).  Both formal analysis \cite{vonkarman1940engineer} and numerical computations \cite{dias2003internal} suggest that the free boundary will meet the wall at an angle of $60^\circ$.}
  \label{bores alternatives figure}
\end{figure}

 Waves in salt water bodies are frequently stratified: they  exhibit large nearly homogeneous density regions that are separated by much thinner regions, called \emph{pycnoclines}, where the density varies rapidly.  It is then reasonable to treat the water as two immiscible fluids, each with constant density and governed by the incompressible Euler equations.  The pycnocline accordingly becomes a surface of discontinuity for the density and a free boundary dividing the layers.  Since the 1960s, oceanographers have known that massive \emph{internal waves} can propagate along this interface, and even remain coherent over long distances.    While this phenomenon has been studied extensively in geophysics, the results are mostly restricted to linear or weakly nonlinear model equations that do not adequately capture the features of very large internal waves \cite{helfrich2006review}.   Front-type solutions in this context are called (smooth) \emph{hydrodynamic bores}.  They are quintessentially a product of stratification in the sense that traveling fronts in homogeneous density water do not exist \cite{wheeler2013solitary}.   

 To focus on the motion of the internal interface, we suppose that the water is bounded above and below by flat rigid boundaries.  We also take the velocity field in each layer to be irrotational and assume there are no (horizontal) stagnation points.  Under these assumptions, the system can be transformed into a quasilinear elliptic PDE with a transmission boundary condition.  The first mathematical construction of heteroclinic solutions for this problem is due to Amick and Turner \cite{amick1989small}, who found small-amplitude bores in a neighborhood of the trivial state where the free boundary is flat.   Later, Mielke \cite{mielke1995homoclinic},  Makarenko \cite{makarenko1992bore}, and Chen, Walsh, and Wheeler \cite{chen2022center} obtained similar results using alternative methods.     
 
 In this paper, we give the first existence theory for genuinely large-amplitude bores; see Theorem~\ref{global bore theorem}.  By means of Theorem~\ref{general global bifurcation theorem}, the local curve of solutions constructed in \cite{chen2022center} is continued globally.  Careful consideration of the conjugate flow equations enables us to rule out heteroclinic degeneracy \ref{gen hetero degeneracy} and spectral degeneracy \ref{gen ripples}, leaving only blowup \ref{gen blowup alternative}.  Then, through a priori bounds, we prove that \ref{gen blowup alternative} leads inexorably to stagnation:  as we follow the global curve, we encounter solutions where at some point the horizontal velocity comes arbitrarily close to the speed of the wave itself.   This type of limiting behavior is well known in the water waves literature.  Most famously, the family of Stokes waves terminates at an ``extreme wave'' with a stagnation point at its crest \cite{amick1982stokes}.  Unlike Stokes waves, the limiting form of the interface along our solution curves is not a corner.  Instead, using a novel free boundary regularity argument, we prove that the elevation bores either overturn (develop a vertical tangent) or become singular in a certain sense to be made clear shortly.  On the other hand, the depression bores will limit either to an overturning front or one whose interface meets the upper wall.  See Figure~\ref{bores alternatives figure} and the precise statements in Theorem~\ref{limit eta theorem} as well as Remarks~\ref{elevation overturning remark} and \ref{depression overturning remark}.
 
To put this into context, recall that a steady water wave is said to be \emph{overhanging} if its free boundary elevation is multi-valued.  Note that these are traveling wave solutions, so the interface will persist in this unusual configuration for all time.  With surface tension but absent gravity (pure capillary waves) explicit solutions of this type were discovered by Crapper \cite{crapper1957exact}.  For different physical regimes, Akers, Ambrose, and Wright \cite{akers2014gravity} and C\'ordoba, Enciso, and Grubic \cite{cordoba2016existence} constructed overhanging periodic capillary-gravity waves using the (local) implicit function theorem at a Crapper wave to introduce a small amount of gravity.  By contrast, Ambrose, Strauss, and Wright \cite{ambrose2016global} obtain periodic capillary-gravity internal waves through global bifurcation; numerical computations by the same authors found that some of these solutions are overhanging.  
 
All of this theory relies crucially on capillarity.  However, a series of remarkable numerical results in the 1980s predicted the existence of overhanging gravity waves \cite{pullin1988finite,turner1988broadening}.  Rather than surface tension, they are able to maintain their shape due to vorticity, either in the form of a background current or a vortex sheet, as is the case for internal waves.  Since that time, constructing overhanging gravity waves has been among the largest open problems in water waves.    

Constantin, Varvaruca, and Strauss \cite{constantin2016global}  made  significant progress in this direction, proving a global bifurcation result for two-dimensional periodic gravity water waves with constant vorticity in a setting that permits overhanging.  Subsequent computations by Dyachenko and Hur \cite{dyachenko2019folds,dyachenko2019stokes} offer overwhelming numerical evidence that this family does indeed contain waves that overhang.  Unfortunately, no rigorous proof is available as they are unable to exclude the possibility that a corner forms prior to overturning.  Indeed, this scenario is also seen numerically for some parameter regimes.  The central issue is that it is exceedingly difficult to track the qualitative features of the waves as one follows the global curve.  This is particularly true when working with the Babenko-like formulation adopted in \cite{constantin2016global}.  Similarly, an earlier work by Sun \cite{sun2001existence} treats periodic internal gravity waves in a two-fluid system where each layer is infinitely deep.   He derives an integral equation formulation in the spirit of Nekrasov that allows overhanging waves, but he does not guarantee that they are present on the global bifurcation curve.  Recent numerics \cite{maklakov2020note} suggest that these wave do not in fact overturn.  Other global bifurcation results for waves which may be overhanging have been established by Haziot~\cite{haziot2021stratified} for linear stratification and constant vorticity, by Haziot and Wheeler~\cite{haziot2023solitary} for solitary waves with constant vorticity, and by Wahl\'en and Weber~\cite{wahlen2022critical} for waves with general vorticity.

Very recently, Hur and Wheeler~\cite{hur2020exact} (see also \cite{hur2020crapper}) found a family of explicit overturning waves with constant vorticity, infinite depth, zero gravity, and zero surface tension. Perhaps surprisingly, the surfaces of these waves are exactly the same as for Crapper's pure capillary waves, although of course the flow pattern beneath the wave is completely different. Hur and Wheeler then perturbed these waves~\cite{hur2022overhanging} to allow for a small amount of gravity, in the same spirit as \cite{akers2014gravity,cordoba2016existence} for capillary-gravity waves. While this rigorously confirms the existence of overturning gravity waves, many related questions remain unresolved, not least about the waves on the global bifurcation curves in~\cite{constantin2016global}.

Our global bifurcation argument is conducted using the Dubreil-Jacotin formulation, which forbids us from having stagnation points.  However, as compensation, it is much simpler to detect overturning in these variables.   For the elevation bores, we are able to eliminate the possibility of the interface meeting the walls.  The only remaining  obstruction to a definitive overturning result is excluding a degenerate scenario where both phases limit to stagnation at exactly the same point on the free boundary. This is the ``singularity'' mentioned above.  In fact, double stagnation would occur at the point on the interface whose elevation coincides with the amplitude of the wave of greatest height in the one-fluid case.  As far as we are aware, none of the numerical studies of this system report seeing this occur.  Indeed, the paper of Dias and Vanden-Broeck \cite{dias2003internal} treats exactly our problem, and they find that the elevation bores invariably  overturn.  On that basis, we conjecture that the singularity alternative can be eliminated with further analysis; see Remark~\ref{elevation overturning remark} for additional comments.   If one can confirm overturning occurs, then in principle one can reformulate the problem in the style of Constantin et al.~to continue into the overhanging regime.  At the same time, we wish to emphasize that the global bifurcation theory needed to construct bores is significantly more subtle than that for periodic waves for all the reasons discussed in Section~\ref{statement abstract results section}.

\subsection{Idea of the proof}

We conclude the section by outlining the proof of Theorems~\ref{global ift} and \ref{general global bifurcation theorem} and its main novelties.  The majority of the argument is contained in Section~\ref{general theory section}.  We begin by fixing an appropriate functional analytic framework for studying fronts.   This is more subtle than it first appears, since we must avoid overreliance on the structural features of the equation.  For example, it is a common practice in reaction-diffusion equations to establish convergence rates for the upstream and downstream limits (this is an integral component of Berestycki and Nirenberg's treatment \cite[Section 3]{berestycki1992travelling}).  However, this would be very technically challenging to carry out in the general case of \eqref{fully nonlinear elliptic pde}.    Another reasonable sounding idea is to subtract a background front, so that the resulting functions vanish at infinity.  Unfortunately, that introduces $x$-dependence into the underlying equation, which is clearly undesirable.    
Our approach is to instead work in a space $\Xspace_\infty$ of H\"older continuous functions with well-defined limits at infinity.  While unconventional for PDEs, this turns out to be a very natural setting for analyzing fronts in the general case. 

In Section~\ref{monotonicity section}, a maximum principle argument is used to confirm that, as long as spectral degeneracy \ref{gen ripples} does not occur, strict monotonicity is both an open and closed property in a relative topology on $\F^{-1}(0)$.  As a consequence, strict monotonicity holds on any connected subset of the zero-set containing $(u,\Lambda)$ or $\cm_\loc$.  On the other hand, in Section~\ref{compactness or hetero degen section}, we show that the zero-set is relatively pre-compact as long as the solutions it contains are strictly monotonic and heteroclinic degeneracy \ref{gen hetero degeneracy} does not occur.

By exploiting monotonicity and the definition of $\Xspace_\infty$, we are able to eliminate in a rather elegant way the kernel direction generated by the translation invariance.   This is accomplished by means of a functional $\mathcal{C}$ that compares the value of $u$ at a near-field point to the average of its values upstream and downstream at the same height.  The kernel of $\mathcal{C}$ thus contains precisely one translate of any monotone front.  Following that approach, we introduce a ``bordered problem'' whose linearization at a monotone front solution is Fredholm index $0$.    Through the global bifurcation theory in \cite{chen2018existence}, we can then continue the zero-set of the augmented nonlinear operator to obtain a global curve that satisfies a larger set of alternatives.    Tying these threads together, we complete the proofs of Theorem~\ref{global ift} and Theorem~\ref{general global bifurcation theorem} in Section~\ref{general theorem proof section}.

In Section~\ref{variational section} we look at the special case where \eqref{fully nonlinear elliptic pde} has a variational structure.  We prove that for such problems, the translation invariance in $x$ generates a conserved quantity, which can then be used to analyze the conjugate flows.   We also discuss how Theorem~\ref{global ift} and Theorem~\ref{general global bifurcation theorem} can be extended to handle problems with transmission boundary conditions.

We then turn to the applications.  Global bifurcation for the semilinear Robin problem is undertaken in  Section~\ref{robin section}, while Section~\ref{bores section} is devoted to the hydrodynamic bore problem.   

Lastly, two appendices are included.  In Appendix~\ref{principal eigenvalue spec condition appendix}, we discuss the principal eigenvalues of elliptic operators on bounded domains.  Some facts from the literature are recalled, and we prove a new result for transmission problems that is needed for the internal waves application.  We then establish the relationship between \eqref{spectral assumption} and the essential spectrum of the linearized operators at infinity.  Appendix~\ref{quoted results appendix} contains a version of the global bifurcation theory in \cite{chen2018existence} reframed in the way most convenient for the present paper.

\section{Abstract global continuation} \label{general theory section}

\subsection{Spaces}\label{spaces section}

The first task is to set down the functional analytic framework.  We want to work in a closed subspace of $\Xspace_\bdd$ that enforces front-type behavior as $x \to \pm\infty$, and then choose an appropriate codomain so that restricting to these spaces does not alter the Fredholm properties of the mapping.  Our strategy is to consider spaces of H\"older continuous functions that have well-defined upstream and downstream limits; this is essentially the largest subspace of $\Xspace_\bdd$ containing all fronts. We then define a linear functional $\mathcal C$ on these spaces which enables us to distinguish between different translates of a strictly monotone front.

We begin by defining
\begin{equation}
  \begin{aligned}
    \Xspace_\infty := \Big\{ u \in \Xspace_\bdd  : &  \lim_{x \to -\infty} \partial^\beta u \;\textrm{ and } \lim_{x \to +\infty} \partial^\beta u\textrm{ exist} && \textrm{ for all $|\beta| \leq k+2$} \Big\}, 
  \end{aligned}
  \label{def X infty space} 
\end{equation}
where all of the above limits are uniform in $y$. Notice that as a consequence of boundedness, we have moreover that if $u \in \Xspace_\infty$, then $\partial^\beta \partial_x u \to 0$ uniformly in $y$ as $x \to \pm\infty$ for all multi-indices $|\beta| \leq k+1$.  For the target space, we take
\begin{equation}
  \begin{aligned}
    \Yspace_\infty := \Big\{ (f_1,f_2) \in \Yspace_\bdd  : &  \lim_{x \to -\infty} \partial^\beta f_1 \;\textrm{ and } \lim_{x \to -\infty} \partial^\beta f_1 \textrm{ exist} && \textrm{ for all $|\beta| \leq k$}, \\
    & \lim_{x \to -\infty} \partial_\tau^\beta f_2 \;\textrm{ and } \lim_{x \to +\infty} \partial_\tau^\beta f_2 \textrm{ exist} &&\textrm{ for all $|\beta| \leq k+1$} \Big\},
  \end{aligned}
  \label{def Y infty space} 
\end{equation}
where $\partial_\tau^\beta$ denotes a generic tangential derivative of order $\beta$ along $\Gamma_1$.  Finally, let 
\[ \genU_\infty := \left\{ (u,\Lambda) \in \genU : u \in \Xspace_\infty \right\}.\]
The following two lemmas are immediate consequences of these definitions.

\begin{lemma} 
  It holds that $\Xspace_\infty$ and $\Yspace_\infty$ are closed subspaces of $\Xspace_\bdd$ and $\Yspace_\bdd$, respectively, and hence Banach spaces.
\end{lemma}

\begin{lemma}  \label{F maps Xinfty to Yinfty lemma} 
  $\F(\genU_\infty) \subset \Yspace_\infty$ and $\F$ is real analytic as a mapping $\genU_\infty \to \Yspace_\infty$. 
\end{lemma}

The choice of $\Xspace_\infty$ is justified by the next lemma.

\begin{lemma} 
  \label{fronts are Xinfty lemma} If $(u,\Lambda) \in \genU$ is a front solution to $\F(u,\Lambda)=0$, then there exist $U_\pm \in C^{k+2+\alpha}(\overline{\Omega^\prime})$ such that 
  \begin{equation}
    \lim_{x \to \pm\infty} \partial^\beta u(x, \placeholder) = \partial^\beta U_\pm \qquad \textrm{for all } |\beta| \leq k+2 
  \label{fronts have limits}
  \end{equation}
  uniformly in $y$, and consequently $u \in \Xspace_\infty$ and $\F(U_\pm, \Lambda) = 0$.
\end{lemma}
\begin{proof}
  Since $(u,\Lambda)$ is a front, there exist $U_{\pm} = U_{\pm}(y)$ so that $u(x, y) \to U_\pm(y)$ as $x \to \pm\infty$ for each $y \in \Omega^\prime$.  Consider the doubly infinite sequence of translates $\{u_\tau = u(\placeholder + \tau, \placeholder)\}_{\tau \in \mathbb{Z}}$.  Then $(u_\tau, \Lambda)$ is a front solution for each $\tau$, and $\{ u_\tau \}$ is uniformly bounded in $C_\bdd^{k+2+\alpha}$.  It follows that, possibly passing to a subsequence, there exists $u_{\pm} \in \Xspace_\bdd$ such that $u_\tau \to u_\pm$ in $C_\loc^{k+2+\alpha/2}$ as $\tau \to \pm\infty$. Clearly, we must have $u_\pm = U_\pm$, and hence $(U_\pm,\Lambda)$ is an $x$-independent solution of \eqref{fully nonlinear elliptic pde}.  

  Now, let $v_\tau := u_\tau - U_+$ and consider the limit $\tau \to \infty$ on the bounded set $\Omega_2 := (-2,2) \times \Omega'$. For $\tau$ sufficiently large, the translation invariance of \eqref{fully nonlinear elliptic pde} and the fact that $v_\tau$ is the difference of two solutions implies that it satisfies a linear elliptic PDE 
  \begin{equation} \label{vtau equation}
   \left\{     
    \begin{alignedat}{2} 
      a^{ij}_\tau(x,y) \partial_{i}\partial_{j} v_\tau + b^i_\tau(x,y) \partial_i v_\tau + c_\tau(x,y) v_\tau &= 0 &\qquad& \text{in } \Omega_2\\
      \beta_\tau^i(x,y) \partial_i v_\tau + \gamma_\tau(x,y) v_\tau &= 0 && \text{on }  \Gamma_1 \cap \partial\Omega_2\\
      v_\tau &= 0 && \text{on }  \Gamma_0 \cap \partial\Omega_2,
    \end{alignedat} 
    \right. 
  \end{equation}
  where we are using summation convention with the shorthand $\nabla_{(x,y)} =: (\partial_1, \ldots, \partial_n)$. The coefficients are defined in terms of the convex combinations $u_\tau^{(s)} := su_\tau + (1-s)U_+$ by the integrals
  \begin{align*}
    a^{ij}_\tau
    &:= 
    \int_0^1 \mathcal F_{r^{ij}}(y, u_\tau^{(s)}, \nabla u_\tau^{(s)}, D^2 u_\tau^{(s)},\Lambda)\, ds,
    &
    b^i_\tau
    &:= 
    \int_0^1 \mathcal F_{\xi^{i}}(y, u_\tau^{(s)}, \nabla u_\tau^{(s)}, D^2 u_\tau^{(s)},\Lambda)\, ds,
    \\
    c_\tau
    &:= 
    \int_0^1 \mathcal F_z(y, u_\tau^{(s)}, \nabla u_\tau^{(s)}, D^2 u_\tau^{(s)},\Lambda)\, ds,
    &
    \beta^i_\tau
    &:= 
    \int_0^1 \mathcal G_{\xi^{i}}(y, u_\tau^{(s)}, \nabla u_\tau^{(s)},\Lambda)\, ds,
    \\
    \gamma_\tau
    &:= 
    \int_0^1 \mathcal G_z(y, u_\tau^{(s)}, \nabla u_\tau^{(s)},\Lambda)\, ds.
  \end{align*}
  These integrands are well-defined for $\tau$ sufficiently large as
  \begin{align*}
    (u_\tau^{(s)}, \nabla u_\tau^{(s)}, D^2 u_\tau^{(s)})
    \longrightarrow (U_+, \nabla U_+, D^2 U_+) \quad \asa \tau \to \infty,
  \end{align*}
  uniformly in $s \in [0,1]$ and $(x,y) \in \Omega_2$. It is easy to see that the uniform bounds on $u_\tau$ and $U_+$ in $C^{k+2+\alpha}_\bdd(\overline\Omega)$ imply that the $C^{k+\alpha}$ norms of $a^{ij}_\tau,b^i_\tau,c_\tau$ as well as the $C^{k+1+\alpha}$ norms of $\beta^i_\tau,\gamma_\tau$ are bounded uniformly. As \eqref{vtau equation} is uniformly elliptic with a uniformly oblique boundary condition, we therefore have a (linear) Schauder estimate
  \begin{align*}
    \n{v_\tau}_{C^{k+2+\alpha}(\Omega_1)} \le C \n{v_\tau}_{C^0(\Omega_2)} 
  \end{align*}
  where the constant $C > 0$ is independent of $\tau$ and $\Omega_1 := (-1,1) \times \Omega^\prime$. Since $v_\tau = u_\tau - U_+ \to 0$ in $C^0_\bdd(\overline\Omega_2)$, we conclude that $v_\tau \to 0$ in $C_\bdd^{k+2+\alpha}(\overline{\Omega_1})$.  Recalling its definition, this shows that 
  \[    \partial^\beta u(x, \placeholder) \to \partial^\beta U_+ \qquad \textrm{for all } |\beta| \leq k+2,\]
  as $x \to \infty$, uniformly in $y$. Redefining $v_\tau := u_\tau - U_-$ and running the same argument to study the limit $\tau \to -\infty$, we obtain  \eqref{fronts have limits}. The fact that $u \in \Xspace_\infty$ is then clear from its definition in \eqref{def X infty space}. 
\end{proof}

Clearly, translation in $x$ is an isometry on $\Xspace_\infty$. When necessary, we will kill this symmetry by introducing the bounded linear functional
\begin{equation}
  \label{def C functional} 
  \begin{aligned}
    \mathcal C \maps \Xspace_\infty \to \R
    \qquad \quad
    \mathcal{C} u &:=   u(0,y_0) - \frac{1}{2} \Big( \lim_{x \to \infty} u(x,y_0) + \lim_{x \to -\infty} u(x,y_0) \Big),
  \end{aligned}
\end{equation}
where $y_0$ is an arbitrary but fixed point in $\Omega' \cup \Gamma_1'$. If $(u,\Lambda)$ is a strictly monotone front, then exactly one translate of $u$ lies in the kernel of $\mathcal C$, and moreover $\mathcal C u_x = u_x(0,y_0) \ne 0$.

\subsection{Monotonicity} \label{monotonicity section}

An important component of the continuation argument is that the solutions along the global curve inherit the strict monotonicity of the solution at the bifurcation point, in the case of Theorem~\ref{global ift}, or along the local curve in the setting of Theorem~\ref{general global bifurcation theorem}.  We accomplish this by showing that monotonicity is both an open and closed property in an appropriate relative topology on $\F^{-1}(0)$ so long as the spectral condition \eqref{spectral assumption} holds.   The main tools will be the maximum principle and Hopf boundary-point lemma; see Lemma~\ref{Hopf lemma}.    

In particular, we will use many times the idea of ``quasilinearizing'' the equation.  Because the PDE \eqref{fully nonlinear elliptic pde} is translation invariant in the $x$-direction, it follows that, for any solution $(u,\Lambda) \in \genU$,  $\partial_x u$ is in the kernel of the linearized operator at $(u,\lambda)$.  This operator $\mathscr{L} := \F_u(u,\Lambda)$ takes the form 
\begin{equation}
  \begin{aligned}
    \mathscr{L}_1 v & := a^{ij}(x,y) \partial_i \partial_j v + b^i(x,y) \partial_i v + c(x,y) v, \\
    \mathscr{L}_2 v & := \left(  \beta^i(x,y) \partial_i v + \gamma(x,y) v\right)|_{\Gamma_1},
  \end{aligned} \label{def linearized elliptic operator} 
\end{equation}
where the coefficients $a^{ij}, b^i, c  \in C_\bdd^{k+\alpha}(\overline{\Omega})$, $\beta^i, \gamma \in C_\bdd^{k+1+\alpha}(\Gamma_1)$ and, in light of \eqref{ellipticity assumption}, $\mathscr{L}_1$ is a uniformly elliptic operator while $\mathscr{L}_2$ is a uniformly oblique boundary operator.    Elliptic regularity ensures that $\partial_x u \in C^{k+2+\alpha}(\overline{\Omega})$.  If $u \in \Xspace_\infty$, then the coefficients have well-defined limits as $x \to \pm\infty$, and $\mathscr{L}$ limits to a linearized operator at $x = \pm\infty$, which we denote by 
\begin{equation}
  \begin{aligned}
    \limL_{1\pm} v & := a_\pm^{ij}(y) \partial_i \partial_j v + b_\pm^i(y) \partial_i v + c_\pm(y) v, \\
    \limL_{2\pm} v & := \left(  \beta_\pm^i(y) \partial_i v + \gamma_\pm(y) v\right)|_{\Gamma_1}.
  \end{aligned} \label{def linearized lim elliptic operator} 
\end{equation}
 
 The next lemma gives a crucial implication of the spectral condition \eqref{spectral assumption}: the existence of a comparison function $\varphi_\pm$ that permits us to apply the maximum principle in a neighborhood of $x= \pm\infty$. Without loss of generality, here and in the remainder of the section, we present the argument for the decreasing case.
 
\begin{lemma}[Comparison function] \label{supersolution lemma}
Let $(u,\Lambda) \in \genU_\infty$ satisfying \eqref{spectral assumption} be given.  For all $\delta > 0$ sufficiently small, there exists $\varphi_\pm \in C^{k+2+\alpha}(\overline{\Omega'})$ with $\varphi_\pm > 0$ on $\overline{\Omega^\prime}$ and such that
\begin{equation}
 \left\{ 
  \begin{aligned}
    \left( \limL_{1\pm}^\prime  + \delta \right) \varphi_\pm & = 0 &\qquad& \textup{in } \Omega^\prime, \\
    \limL_{2\pm}^\prime  \varphi_\pm &=  1 &\qquad& \textup{on } \Gamma_1^\prime, \\
    \varphi_\pm & = 1 &\qquad& \textup{on } \Gamma_0^\prime.
  \end{aligned} \right. \label{def comparison} 
\end{equation}
\end{lemma} 
\begin{proof} 
  By continuity, it follows that the operator $(\limL_{1\pm}^\prime + \delta, \limL_{2\pm}^\prime)$ has a strictly negative principal eigenvalue for $\delta > 0$ sufficiently small; see Lemma~\ref{continuity sigma0 lemma}.  We may therefore define $\varphi_\pm \in C^{k+2+\alpha}(\overline{\Omega^\prime})$ to be the (unique) solution to \eqref{def comparison}, which in particular ensures $\varphi_\pm > 0$ on $\Gamma_0^\prime$ and $\limL_{2\pm}^\prime \varphi_\pm \geq 0$.  The unique solvability of \eqref{def comparison} follows from \cite[Theorem~4.1]{lopezgomez2003classifying} and a standard homotopy argument similar to the proof of Lemma~\ref{invertibility lemma}. The negativity of the principal eigenvalue then enables us to appeal to Theorem~\ref{oblique prob eigenvalue theorem}\ref{oblique classification part} to conclude further that $\varphi_\pm > 0$ on all of $\overline{\Omega^\prime}$.  
\end{proof}

\begin{lemma}[Asymptotic monotonicity] \label{asymptotic monotonicity lemma}
  For $M > 0$, denote 
  \[ \Omega_+^M := (M,\infty) \times \Omega^\prime, \qquad \Omega_-^M := (-\infty, -M) \times \Omega^\prime, \]
  and let $\Gamma_{0\pm}^M$ and $\Gamma_{1\pm}^M$ be the corresponding boundary components of $\Omega_\pm^M$. For any $(\bar u,\bar \Lambda) \in \mathscr{F}^{-1}(0) \cap \genU_\infty$ satisfying \eqref{spectral assumption}, there exist $\varepsilon > 0$ and $M > 0$ such that any other $(u,\Lambda) \in \mathscr{F}^{-1}(0) \cap \genU_\infty$ with
  \[ \| u - \bar u  \|_{C^{2}(\Omega_\pm^M)} + |\Lambda-\bar \Lambda| < \varepsilon, \quad \textrm{and} \quad  \partial_x u \leq 0~ \textrm{on } \{ x = \pm M \},\]
  satisfies
  \[ \partial_x u = 0 \quad \textrm{in } \overline{\Omega_\pm^M} \qquad \textrm{or} \qquad \partial_x u < 0 \quad \textrm{in } \Omega_\pm ^M \cup \Gamma_{1\pm}^M.\]
\end{lemma} 
\begin{proof}
  Let $(\bar u,\bar \Lambda)$ be given as above.  By hypothesis \eqref{spectral assumption} and Lemma~\ref{supersolution lemma}, there exists a comparison function $\bar\varphi_\pm$ as in \eqref{def comparison}. After perhaps shrinking $\varepsilon > 0$, $(u,\Lambda)$ also satisfies \eqref{spectral assumption}, and the associated comparison function $\varphi_\pm$ (with the same $\delta$) has $\varphi_\pm \to \bar\varphi_\pm$ in $C^2(\overline{\Omega'})$ as $\varepsilon \to 0$. Taking $u_x =: v \varphi_\pm$, we have that $v \in C_\bdd^{k+2+\alpha}(\overline{\Omega_\pm^M})$, vanishes in the limit $x \to \pm \infty$, and solves 
\begin{equation}
  \label{division tricked}
  \left\{ \begin{aligned}
    a^{ij}(x,y) \partial_i \partial_j v + \left( \frac{a^{ij}(x,y) \partial_j \varphi_\pm}{\varphi_\pm} +  b^i(x,y)  \right) \partial_i v +
    \frac{\limL_1 \varphi_\pm}{\varphi_\pm} v & = 0 && \qquad \textrm{in } \Omega_\pm^M \\
    \beta^i(x,y) \partial_i v + \frac{\limL_{2} \varphi_\pm}{\varphi_\pm} v & = 0 && \qquad \textrm{on } \Gamma_{1\pm}^M \\
    v & = 0 && \qquad \textrm{on } \Gamma_{0\pm}^M.
  \end{aligned} \right.
\end{equation}
By choosing $\varepsilon > 0$ sufficiently small and $M > 0$ sufficiently large, we can make 
\begin{align*}
  \left\| 
  \frac{\limL_1 \varphi_\pm}{\varphi_\pm}
  -
  \frac{\bar\limL_{1\pm}' \bar\varphi_\pm}{\bar\varphi_\pm}
  \right\|_{C^0(\Omega_\pm^M)},
  \qquad 
  \left\| 
  \frac{\limL_2 \varphi_\pm}{\varphi_\pm}
  -
  \frac{\bar\limL_{2\pm}' \bar\varphi_\pm}{\bar\varphi_\pm}
  \right\|_{C^0(\Gamma_{1\pm}^M)}
\end{align*}
arbitrarily small. In particular, since
\begin{align*}
  \frac{\bar\limL_{1\pm}' \bar\varphi_\pm}{\bar\varphi_\pm} = -\delta < 0,
  \qquad 
  \frac{\bar\limL_{2\pm}' \bar\varphi_\pm}{\bar\varphi_\pm} = \frac 1{\bar\varphi_\pm} > 0,
\end{align*}
we can pick $\varepsilon$ and $M$ so that the zeroth order coefficients in \eqref{division tricked} have the strict signs
\begin{align}
  \label{strict signs}
  \frac{\limL_1 \varphi_\pm}{\varphi_\pm} < 0 \ona \Omega^M_\pm,
  \qquad 
  \frac{\limL_2 \varphi_\pm}{\varphi_\pm} > 0 \ona \Gamma^M_{1\pm}.
\end{align}
We can now conclude using the maximum principle and Hopf lemma as follows. 
  By assumption, $v$ vanishes on $\Gamma_{0\pm}^M$ and in the limit $x \to \pm\infty$.  Moreover, from the positivity of $\varphi_\pm$, we also have that $v \leq 0$ on $\{ x = \pm M\}$. So assume for the sake of contradiction that $v$ achieves a nonnegative maximum at some point $(x_0, y_0) \in \Omega_\pm^M \cup \Gamma_{1\pm}^M$. Thanks to the first inequality in \eqref{strict signs}, if $(x_0,y_0) \in \Omega_\pm^M$ then $v$ must vanish identically. Similarly, if $(x_0,y_0) \in \Gamma_{1\pm}^M$, then the Hopf boundary-point lemma implies that either $v$ vanishes identically or else
\[ 0 < \beta^i \partial_i v = -\frac{\limL_2 \varphi_\pm}{\varphi_\pm} v \qquad \textrm{at } (x_0, y_0),\] 
which contradicts $v(x_0,y_0) \ge 0$ and the second inequality in \eqref{strict signs}. The statement now follows by recalling the relation between $v$ and $\partial_x u$.
\end{proof}
 
\begin{lemma}[Open property] \label{open property lemma}
Let $(\bar u,\bar \Lambda) \in \genU_\infty$ be a strictly monotone front solution to \eqref{fully nonlinear elliptic pde} satisfying \eqref{spectral assumption}.  There exists $\delta = \delta(\bar u,\bar \Lambda) > 0$ such that any solution $(u, \lambda) \in \genU_\infty$ of \eqref{fully nonlinear elliptic pde}  with 
\[ \| u - \bar u \|_{C^{2}(\Omega)} + |\Lambda - \bar \Lambda| < \delta, \]
is also strictly monotone.   
\end{lemma}
\begin{proof}
  Choose $\varepsilon > 0$ and $M > 0$ as in Lemma~\ref{asymptotic monotonicity lemma}. First, we consider the situation on the finite cylinder $(-2M, 2M) \times \Omega^\prime$.  By assumption $\partial_x \bar u < 0$ in $\Omega \cup \Gamma_1$. Moreover, since $\partial_x \bar u$ vanishes on $\Gamma_0$, the Hopf boundary-point lemma applied to the elliptic equation $\bar{\mathscr L}_1 \partial_x u = 0$ yields
\[ \nu \cdot \nabla \partial_x \bar u \geq c_2 > 0 \qquad \textrm{on } [-2M, 2M] \times \Gamma_0^\prime\]
for some $c_2 = c_2(M)$.  Choosing $\delta > 0$ sufficiently small  we can therefore ensure that 
\[ \partial_x u < 0 \qquad \textrm{on } [-2M, 2M] \times ( \Omega^\prime \cup \Gamma_1^\prime) .\]
In particular, this also implies that $\partial_x u < 0$ on $\{ x = \pm M\}$. Shrinking $\delta$ further if necessary so that $\delta < \varepsilon$, Lemma~\ref{asymptotic monotonicity lemma} then furnishes the strict monotonicity of $u$ on the tail regions $\Omega_{\pm}^M \cup \Gamma_{1\pm}^M$.        
\end{proof}

We also need an analogue of Lemma~\ref{open property lemma} for limiting states.
\begin{lemma}[Open property for limiting states]\label{limiting open property lemma}   
  Let $(\bar u,\bar \Lambda) \in \genU_\infty$ be a front solution to \eqref{fully nonlinear elliptic pde} satisfying \eqref{spectral assumption}, and let $\bar U$ be one of its limiting states. Then all solutions $(u,\Lambda)$ in an open neighborhood of $(\bar U,\bar \Lambda)$ in $\F^{-1}(0) \cap \genU$ depend only the transverse variable $y$.
  \begin{proof}
    By Lemma~\ref{invertibility lemma}, the spectral condition \eqref{spectral assumption} implies that the linearized operator $\F_u(\bar U,\bar \Lambda)$ is invertible $\Xspace_\bdd \to \Yspace_\bdd$. Applying the analytic implicit function theorem, there is therefore an open neighborhood of $(\bar U, \bar \Lambda)$ in $\F^{-1}(0) \cap \genU$ which can be expressed as a graph $\mathscr M = \{ (u(\Lambda), \Lambda) : \Lambda \in \mathcal{I}\}$, where $\mathcal I$ is a neighborhood of $\Lambda$ in $\R^m$ and $u \maps \mathcal I \to \Xspace_\bdd$ is analytic.

    Now let $\F' \maps \genU' \to \Yspace'$ be the restriction of $\F$ to functions of $y \in \Omega'$ alone, written $\F'=\F'(U,\Lambda)$. As in the proof of Lemma~\ref{supersolution lemma}, \eqref{spectral assumption} also implies that the linearized operator $\F'_U(\bar U,\bar \Lambda)$ is invertible. Thus there is an open neighborhood of $(\bar U, \bar \Lambda)$ in $\F^{-1}(0) \cap \genU'$ which is a graph $\mathscr M' = \{ (U(\Lambda), \Lambda) : \Lambda \in \mathcal{I}'\}$. Perhaps after shrinking $\mathcal{I}$, uniqueness forces $\mathscr M \subset \mathscr M'$, and the proof is complete.
  \end{proof}
\end{lemma}

\begin{lemma}[Closed property] \label{closed property lemma}
  Suppose that $\{ (u_n, \Lambda_n) \} \subset \genU_\infty$ is a sequence of strictly monotone front solutions of \eqref{fully nonlinear elliptic pde}.  If $(u_n, \Lambda_n) \to (u,\Lambda)$ in $\Xspace_\bdd \times \mathbb{R}$ for some $(u,\Lambda) \in \genU_\infty$ satisfying \eqref{spectral assumption}, then $(u,\Lambda)$ is strictly monotone.  
\end{lemma}
\begin{proof}
  First, observe that $\Xspace_\infty$ is closed under the $C^{k+2+\alpha}$ norm, and hence $(u,\Lambda) \in \Xspace_\infty \times \mathbb{R}$.  Moreover, the continuity of $\mathscr{F}$ ensures that $(u,\Lambda)$ is also a solution. 

  Suppose that $u_x$ does not vanish identically. It is immediately clear that
  \[ u_x \leq 0 \qquad \textrm{in } \overline{\Omega},\]
  so suppose for the sake of contradiction that $u_x$ vanishes at some point $(x_0,y_0) \in \Omega \cup \Gamma_1$. Since $\mathscr L_1 u_x = 0$, the strong maximum principle implies that $(x_0,y_0) \notin \Omega$. Unlike in the proof of Lemma~\ref{asymptotic monotonicity lemma} above, we can ignore the sign of the zeroth-order coefficient $c$ because the maximum in question is $0$. On the other hand, if $(x_0,y_0) \in \Gamma_1$, then the Hopf boundary-point lemma and the uniform obliqueness of $\mathscr L_2$ yield $\beta^i \partial_i u_{x} > 0$ there. But this is a contradiction since $\mathscr L_2 u_x =0$ forces
  $\beta^i \partial_i u_{x} =  -\gamma u_x = 0$ at $(x_0, y_0)$.

  It remains to consider the case where $u_x$ vanishes identically. Then $u$ depends only on the transverse variable $y$, and so we can apply Lemma~\ref{limiting open property lemma} with $\bar u = \bar U = u$ to see that, for $n$ sufficiently large, $u_n$ also only depends on $y$, contradicting the assumed strict monotonicity.
\end{proof}

Combining the above lemmas, we arrive at the main result of this subsection.%
\begin{theorem}[Monotonicity] \label{monotonicity theorem} Suppose that $\mathcal{K} \subset \F^{-1}(0) \cap \genU_\infty$ is a connected set that contains a strictly monotone decreasing front (or increasing front).  If \eqref{spectral assumption} holds on $\mathcal{K}$, then every element of $\mathcal{K}$ is a strictly monotone decreasing front (or increasing front).  
\end{theorem}

\subsection{Compactness or heteroclinic degeneracy} \label{compactness or hetero degen section}

With the monotonicity properties in hand, we next characterize the loss of compactness scenario.  Previously, this was done for homoclinic solutions (solitary waves) in  \cite[Lemma 6.3]{chen2018existence}.  The adaptation of these ideas to the heteroclinic setting is somewhat subtle, and involves the linear functional $\mathcal C$ introduced in Section~\ref{spaces section}.

\begin{lemma}[Compactness or heteroclinic degeneracy] \label{compactness fronts lemma} Suppose that there exists a uniformly bounded sequence $\{ (u_n, \Lambda_n) \} \subset \genU_\infty$ of strictly monotone front solutions to \eqref{fully nonlinear elliptic pde} with $\mathcal C u_n = 0$.   Then, either \begin{enumerate}[label=\rm(\roman*)]
  \item \label{compact front alt} we can extract a subsequence so that $(u_n,\Lambda_n) \to (u,\Lambda)$ in $C^{k+2+\alpha}_\bdd(\overline\Omega) \times \mathbb{R}^m$, where $(u,\Lambda)$ is a monotone front solution to \eqref{fully nonlinear elliptic pde}; or
  \item \label{triple conjugacy alt} there exists a sequence $x_n \to \pm\infty$ such that, after extracting a subsequence,
  \[ \left( u_n(\placeholder + x_n, \placeholder), \, \Lambda_n\right) \to (u_*,\Lambda_*) \qquad \textrm{in } C^{k+2}_\loc(\overline\Omega) \times \mathbb{R}^m\]
   for some $(u_*,\Lambda_*) \in C^{k+2+\alpha}_\bdd(\overline\Omega) \times \mathbb{R}$ that is a monotone front solution of \eqref{fully nonlinear elliptic pde}, but the three limiting states
   \begin{equation*}
     \lim_{x \to \mp\infty} u_*(x, \placeholder),
     \quad 
     \lim_{n \to \infty} \lim_{x \to +\infty} u_n(x, \placeholder),
     \quad 
     \lim_{n \to \infty} \lim_{x \to -\infty} u_n(x, \placeholder),
   \end{equation*}    
   are all distinct.
 \end{enumerate}
\end{lemma}
Indeed, we will see from the proof that the limiting states in \ref{triple conjugacy alt} are strictly ordered.
\begin{proof}
  Extracting a subsequence, we can assume that the fronts are all monotone increasing or monotone decreasing. We give the proof in the decreasing case, the other case being completely analogous. 

  Throughout the proof, let $U_{n\pm} \in C^{k+2+\alpha}(\overline{\Omega^\prime})$ denote the limiting states of $u_n$. Passing to a subsequence, we can assume that $U_{n\pm} \to U_\pm$ in $C^{k+2}(\overline{\Omega^\prime})$ and $\Lambda_n \to \Lambda_* \in \mathbb{R}^m$, for some $(U_\pm, \Lambda_*) \in C^{k+2+\alpha}(\overline{\Omega^\prime}) \times \mathbb{R}^m$.  Arguing in almost exactly the same way as in the proof of \cite[Lemma 6.3]{chen2018existence}, we can show that if the two limits
  \[ \lim_{x \to \pm \infty} \sup_n \|  u_n(x,\placeholder) - U_{n\pm} \|_{C^0(\Omega^\prime)} = 0,\]
  then the compactness alternative \ref{compact front alt} holds.  So assume to the contrary that there exists $\varepsilon > 0$ and a sequence of points $(x_n,y_n) \in \Omega$ with $x_n \to \pm\infty$ and such that 
  \begin{equation}
    \label{nonequidecay}
    |u_n(x_n,y_n) - U_{n\pm}(y_n)| > \varepsilon.
  \end{equation}

  As $\overline{\Omega^\prime}$ is compact, passing to a subsequence we can assume that $y_n \to y_* \in \overline{\Omega^\prime}$.    Suppose that $x_n \to \infty$; the case $x_n \to -\infty$ follows from an almost identical argument.  The uniform boundedness of $\{ u_n\}$ in $C^{k+2+\alpha}(\overline{\Omega})$ allows us to extract a subsequence with
  \[ \tilde u_n := u_n(\placeholder + x_n, \placeholder) \to  u_* \quad \textrm{in } C^{k+2}_\loc(\overline \Omega)\]
  for some $u_* \in C^{k+2+\alpha}_\bdd(\overline \Omega)$. Thus $(u_*,\Lambda_*)$ is solution of \eqref{fully nonlinear elliptic pde} which is monotone in that $\partial_x u_* \le 0$. This monotonicity implies the existence of pointwise limits $U_{*\pm}$ as $x \to \pm\infty$, which by Lemma~\ref{fronts are Xinfty lemma} satisfy $U_{*\pm} \in C^{k+2+\alpha}(\overline{\Omega'})$ and $\F(U_{*\pm},\Lambda_*)=0$. Since
  $U_{n-} \ge \tilde u_n \ge U_{n+}$, we conclude that these four limiting states have the non-strict ordering
  \begin{align}
    \label{limiting states}
    U_- \ge U_{*-} \ge U_{*+} \ge U_+. 
  \end{align}
  It remains to show that the three states $U_{*-},U_-,U_+$ are distinct. 
 
  Let $U_1,U_2$ be any two of the limiting states in \eqref{limiting states} with $U_1 \ge U_2$. Since $\F(U_1,\Lambda)=\F(U_2,\Lambda)=0$ by continuity, arguing as in the proof of Lemma~\ref{fronts are Xinfty lemma} we find that the difference $v := U_2 - U_1 \le 0$ satisfies an elliptic equation of the form \eqref{vtau equation}, except that the coefficients depend only on the transverse variable $y$. Applying the strong maximum principle and Hopf boundary-point lemma as in the proof of Lemma~\ref{closed property lemma}, we conclude that either $U_1 \equiv U_2$ or $U_1 > U_2$ on $\Omega' \cup \Gamma_1'$. 
 
  First consider the pair $U_{*-},U_+$. The assumption \eqref{nonequidecay} implies
    $\tilde u_n(0,y_n) > U_{n+}(y_n) + \varepsilon$. Taking limits yields
  \begin{align*}
    U_{*-}(y_*) \ge u_*(0,y_*) > U_{+}(y_*),
  \end{align*}
  and hence by the above argument that $U_{*-} > U_+$ on $\Omega' \cup \Gamma_1'$. Plugging into \eqref{limiting states}, this then implies that $U_- > U_+$ on $\Omega' \cup \Gamma_1'$.

  Finally, consider the pair $U_-,U_{*-}$. Here we will need the assumption $u_n \in \ker \mathcal C$, i.e.~that at the fixed point $y_0 \in \Omega' \cup \Gamma_1'$ we have
  \begin{align*}
    u_n(0,y_0) = \tfrac 12 \big(U_{n+}(y_0) + U_{n-}(y_0)\big).
  \end{align*}
  Fix $x \in \R$. For $n$ sufficiently large we have $-x_n < x$ and hence by monotonicity that
  \begin{align*}
    \tilde u_n(x,y_0) < \tilde u_n(-x_n,y_0) = u_n(0,y_0)= \tfrac 12 \big(U_{n+}(y_0) + U_{n-}(y_0)\big).
  \end{align*}
  Sending $n \to \infty$ yields
  \begin{align}
    \label{send to minf}
    u_*(x,y_0) \le \tfrac 12 \big(U_+(y_0)+U_-(y_0)),
  \end{align}
  so that upon sending $x \to -\infty$ we recover
  \begin{align}
    \label{use this}
    U_{*-}(y_0) \le \tfrac 12 \big(U_+(y_0)+U_-(y_0)\big) < U_-(y_0)
  \end{align}
  where the last inequality follows from $U_- > U_+$ in $\Omega' \cup \Gamma_1'$. Thus we must have
  \begin{align*}
    U_- > U_{*-} > U_+ \ona \Omega' \cup \Gamma_1'.
  \end{align*}
  In particular, these three functions are distinct and the proof is complete.
\end{proof}

\subsection{Fredholm properties} \label{fredholm theory section}

In this section, we establish some necessary facts about the Fredholm index of $\F_u(u, \Lambda)$.  First, we consider it as a mapping between the larger spaces $\Xspace_\bdd \to \Yspace_\bdd$.  
\begin{lemma}[Fredholm on $\Xspace_\bdd$] \label{fredholm Xb lemma}
Suppose that $(u, \Lambda) \in \genU_\infty$ satisfies the spectral hypothesis \eqref{spectral assumption}.  Then $\F_u(u,\Lambda)$ is Fredholm index $0$ as a mapping $\Xspace_\bdd \to \Yspace_\bdd$. 
\end{lemma}
As the proof of this result relies on some general facts about principal eigenvalues, we postpone it to Appendix~\ref{essential spectrum appendix}.

It remains now to understand how this translates to properties of $\F_u(u,\Lambda)$ as a mapping $\Xspace_\infty \to \Yspace_\infty$.  Since $\Xspace_\infty$ is a closed subspace, we know that if $(u,\Lambda)$ satisfies \eqref{spectral assumption}, then $\F_u(u, \Lambda)\colon \Xspace_\infty \to \Yspace_\infty$ is locally proper, i.e.~semi-Fredholm with a finite-dimensional kernel. To characterize the range, we will need the lemma below.  It is quite similar to \cite[Lemma A.10]{wheeler2015pressure}.

\begin{lemma}\label{lem ext} Let $(u,\Lambda) \in \genU_\infty$ satisfying \eqref{spectral assumption} be given.  If  $\F_u(u, \Lambda) w = (f_1, f_2) \in \Yspace_\infty$, for some $w\in \Xspace_\bdd$, then in fact $w \in \Xspace_\infty$.
\end{lemma}
\begin{proof}
Suppose on the contrary that there exists some $w \in \Xspace_\bdd \backslash \Xspace_\infty$ such that $\F_u(u,\Lambda) w =: (f_1, f_2) \in \Yspace_\infty$.  
To obtain a contradiction, it suffices to  show that all of the limits as $x\to\pm\infty$ behave as in the definition of $\Xspace_\infty$ \eqref{def X infty space}.  We will only present the argument for $|\beta| \leq 2$; the rest follow by differentiating the equation in the usual way. 

Assume that $\lim_{x \to \infty} \partial^\beta w$ does not exist, for some $|\beta| \leq 2$.  Then there is a $\delta > 0$,  $y_* \in \overline{\Omega^\prime}$, and two sequences $x_{1n}, x_{2n} \to +\infty$ with 
\[
 \left| \partial^\beta  w(x_{1n},y_*) - \partial^\beta w(x_{2n},y_*)  \right| \ge \delta.
\]
Consider now the shifted functions
\begin{align*}
v_n & := w(\placeholder+x_{1n}, \placeholder ) - w(\placeholder+x_{2n}, \placeholder), \\
g_{in} & := f_i(\placeholder+x_{1n}, \placeholder) - f_i(\placeholder+ x_{2n}, \placeholder), \quad \textrm{for } i = 1,2,
\end{align*}
and shifted linear operators
\[ \mathscr{L}_n := \F_u(u(\placeholder+x_{1n}, \placeholder), \Lambda). \]
The uniform bound of $v_n$ in $C^{k+2+\alpha}$ implies that we can extract a subsequence so that $v_n \to v$ in $C^{k+2}_\loc(\overline{\Omega})$, for some $v \in \Xspace_\bdd$. Since $(f_1, f_2) \in \Yspace_\infty$, we see that $g_{1n} \to 0$ in $C_\loc^k(\overline{\Omega})$, and $g_{2n} \to 0$ in $C_\loc^{k+1}(\Gamma_1)$. 
Thus 
\[
\mathscr{L}_n v_n = (g_{1n}, g_{2n}) \to 0 \quad \text{in } C_\loc^{k}(\overline{\Omega})  \times C^{k+1}_\loc(\Gamma_1),
\]
from which it follows that $\limL_+ v = 0$.  As $\prineigenvalue^+(u,\Lambda) < 0$, Lemma~\ref{invertibility lemma} forces $v = 0$.  On the other hand, we find that 
\[
|\partial^\beta v(0,y_*)| = \lim_{n\to \infty}  \left| \partial^\beta  w(x_{1n},y_*) - \partial^\beta w(x_{2n},y_*)  \right| \ge \delta,
\]
which is a contradiction. Arguing similarly for $x \to -\infty$, we conclude that $\partial^\beta w$ has well-defined limits as $x\to\pm\infty$ for all $|\beta| \leq k+2$.  Finally, because $w \in \Xspace_\bdd$ this also implies that $\partial^\beta \partial_x w$ vanishes as $x \to \infty$ for all $|\beta| \leq k+1$.
\end{proof}

\begin{corollary} \label{infinity spaces fredholm corollary}
  Let $(u,\Lambda) \in \genU_\infty$ be given such that \eqref{spectral assumption} holds.  Then $\F_u(u, \Lambda)$ is  Fredholm index $0$ as a mapping $\Xspace_\infty \to \Yspace_\infty$.
\end{corollary}
\begin{proof}
  We use the same pair of homotopies as in the proof of Lemma~\ref{fredholm Xb lemma} in Appendix~\ref{essential spectrum appendix}, connecting the operator $\limL:=\F_u(u,\Lambda)$ to the operator $\mathscr T:= (\Delta-\sigma,\partial_\nu)$, where $\sigma > 0$ is a constant and $\partial_\nu$ is the outward pointing normal derivative. The remainder of the argument is very similar to arguments in \cite[Appendix A]{wheeler2015pressure}. By the proof of Lemma~\ref{fredholm Xb lemma}, the operators along these homotopies are locally proper $\Xspace_\bdd \to \Yspace_\bdd$. Since $\Xspace_\infty \subset \Xspace_\bdd$ and $\Yspace_\infty \subset \Yspace_\bdd$ are closed subspaces, these operators are then also locally proper, and hence semi-Fredholm, as mappings $\Xspace_\infty \to \Yspace_\infty$. By the continuity of the index, it therefore suffices to show that $\mathscr T$ is invertible, and hence Fredholm index 0, as a mapping $\Xspace_\infty \to \Yspace_\infty$. From the proof of Lemma~\ref{fredholm Xb lemma}, we know that $\mathscr T$ is invertible $\Xspace_\bdd \to \Yspace_\bdd$, which immediately implies that it is injective $\Xspace_\infty \to \Yspace_\infty$. Moreover, it is easy to see that Lemma~\ref{lem ext} continues to hold with the same proof if $\F_u(u,\Lambda)$ replaced by $\mathscr T$. This and the surjectivity of $\mathscr T \maps \Xspace_\bdd \to \Yspace_\bdd$ implies surjectivity of $\mathscr T \maps \Xspace_\infty \to \Yspace_\infty$, which completes the proof.
\end{proof}

\subsection{Proof of the main theorems} \label{general theorem proof section}

We begin with Theorem~\ref{global ift}, the global implicit function theorem.  Let $(u_0, \Lambda_0)$ be a strict monotone front solution satisfying \eqref{kernel assumption}, \eqref{spectral assumption}, and \eqref{transversality condition}.   By Lemma~\ref{fronts are Xinfty lemma}, we have $u_0 \in \Xspace_\infty$.  Then, in light of \eqref{spectral assumption} and Corollary~\ref{infinity spaces fredholm corollary}, the linearized operator $\F_u(u_0, \Lambda_0)$  is Fredholm index $0$ as a mapping $\Xspace_\infty \to \Yspace_\infty$.  However, it is not an isomorphism due to the kernel direction generated by the translation invariance in $x$, and so we cannot apply Theorem~\ref{homoclinic global ift} directly.

Instead, we study a certain bordered problem: set $w := (u, \mu)$ and consider the new nonlinear operator 
\[ \mathscr{G}\colon \mathcal{W} \subset ( \Xspace_\infty \times \mathbb{R}) \times \mathbb{R} \longrightarrow \Yspace_\infty \times \mathbb{R}\]
given by
\begin{equation}
  \mathscr{G}(w,\lambda) := \begin{pmatrix} \mathscr{F}(u,\lambda,\mu) \\ \mathcal{C} u \end{pmatrix},\label{definition bordered op} 
\end{equation}
where $\mathcal{W}$ is an open set derived from $\genU_\infty$ in the obvious way, and $\mathcal C \maps \Xspace \to \R$ is the bounded linear functional defined in \eqref{def C functional}. Note that here we are adopting the convention that whenever we have $(w,\lambda) \in \mathcal{W}$, it is understood that $u$ denotes the component of $w$ in $\Xspace_\infty$.  
Without loss of generality, assume that $\mathcal{C} u_0 = 0$.   Appealing to Theorem~\ref{homoclinic global ift}, we then obtain a preliminary global bifurcation curve.  

\begin{lemma}[Global IFT for $\mathscr{G}$] \label{global ift K lemma}
There exists a global curve $\borderedcm \subset \mathcal{W}$, parameterized as 
\[ \borderedcm := \left\{ \left(w(s), \lambda(s) \right) : s \in \mathbb{R}  \right\} \subset \mathscr{G}^{-1}(0)\]
for some continuous $\mathbb{R} \ni s \longmapsto (w(s), \lambda(s)) \in \mathcal{W}$ with $(w(0), \lambda(0)) = (w_0, \lambda_0)$.   Moreover, 
\begin{enumerate}[label=\rm(\alph*)]
\item \label{bordered ift alternatives} In each of the limits $s \to \infty$ and $s \to -\infty$, one of the alternatives \ref{K blowup alternative}, \ref{K loss of compactness alternative}, \ref{K loss of fredholmness alternative}, or \ref{K loop} of Theorem~\ref{homoclinic global ift} is realized.  If in either direction \ref{K loop} happens, then clearly it happens in both.  
\item \label{bordered ift analyticity} At each point,  $\borderedcm$ admits a local real-analytic reparameterization.  
\item \label{bordered ift maximality} The curve $\borderedcm$ is maximal in the sense that, if $\mathscr{J} \subset \mathscr{G}^{-1}(0)$ is another locally real-analytic curve of monotone front solutions along which \eqref{spectral assumption} holds, and $(w_0, \Lambda_0) \in \mathscr{J}$, then $\mathscr{J} \subset \borderedcm$.  
\end{enumerate}
\end{lemma}
\begin{proof}  Because $\mathcal{C}$ is a bounded functional on $\Xspace_\infty$, it follows from Corollary~\ref{infinity spaces fredholm corollary} and Fredholm bordering (see Lemma~\ref{bordering lemma}) that 
\[ \mathscr{G}_w(w_0,\lambda_0) =
 \begin{pmatrix} 
 \F_u(u_0, \lambda_0, \mu_0) & \F_\mu(u_0, \lambda_0, \mu_0) \\  
 \mathcal{C} & 0 \end{pmatrix} \]
 is Fredholm index $0$ as a mapping $\Xspace_\infty \times \mathbb{R} \to \Yspace_\infty \times \mathbb{R}$.  We claim, moreover, that it is injective and hence an isomorphism.   Let $\dot w = (\dot u, \dot \mu)$ be an element of $\ker{\mathscr{G}_w(w_0,\lambda_0)}$.  Then from \eqref{definition bordered op}, we have
\[ \mathscr{F}_u(u_0, \Lambda_0) \dot u + \F_\mu(u_0, \Lambda_0) \dot \mu = 0,\]
which by the transversality assumption \eqref{transversality condition} implies that $\dot \mu = 0$ and $\dot u \in \ker\mathscr{F}_u(u_0,\Lambda_0)$.   Hypothesis \eqref{kernel assumption} requires that the kernel of $\mathscr{F}_u(u_0,\lambda_0,\mu_0)$ be generated by $\partial_x u_0$.  However, 
\[ \mathcal{C} \partial_x u_0 = \partial_x u_0(0,y_0) \neq 0,\]
since $u$ is a strictly monotone front.  Therefore, we must have $\dot u = 0$, which confirms that $\ker \mathscr{G}_w(w_0,\lambda_0)$ is trivial.  The invertibility of $\mathscr{G}_w(w_0,\lambda_0)$ follows.

Taking $\mathscr{W} := \Xspace_\infty \times \mathbb{R}$ and $\mathscr{Z} := \Yspace_\infty \times \mathbb{R}$, the above argument implies that $\mathscr{G}$ satisfies the hypotheses \eqref{homoclinic global ift assumptions} of Theorem~\ref{homoclinic global ift}.   As an immediate consequence, we have that there exists a global curve $\borderedcm \subset \mathscr{G}^{-1}(0)$ of solutions exhibiting the properties claimed in parts~\ref{bordered ift alternatives}, \ref{bordered ift analyticity}, and \ref{bordered ift maximality}.  
\end{proof}

Finally, we complete the proof of Theorem~\ref{global ift} by translating the global theory for the bordered problem back to the original equation \eqref{fully nonlinear elliptic pde}.  In doing so, we will use the results of Section~\ref{monotonicity section} and Section~\ref{compactness or hetero degen section} to confirm that strict monotonicity is preserved along the curve, which will then allow us to characterize the loss of compactness~\ref{K loss of compactness alternative} and Fredholmness alternatives~\ref{K loss of fredholmness alternative} in terms of heteroclinic degeneracy \ref{gen hetero degeneracy} and spectral degeneracy \ref{gen ripples}.  

\begin{proof}[Proof of Theorem~\ref{global ift}]
In Lemma~\ref{global ift K lemma}, we constructed a global curve $\borderedcm$ of solutions to the Fredholm bordered problem.  
Naturally, this implies the existence of a corresponding $C^0$ curve 
\[ \iftcm = \left\{ (u(s), \Lambda(s)) : s \in \mathbb{R} \right\} \subset \F^{-1}(0) \cap \genU_\infty \]
of solutions to \eqref{fully nonlinear elliptic pde}.

Observe that the uniformity of the limits in the definition of the space $\Xspace_\infty$ and the continuity of the principal eigenvalue given by Lemma~\ref{continuity sigma0 lemma} together ensure that $s \mapsto \prineigenvaluepm(u(s), \Lambda(s))$ is continuous. Perhaps making a $C^0$ reparameterization, we can therefore assume that \eqref{spectral assumption} holds along $\cm$.  
Since $\cm$ is connected and contains $(u_0, \Lambda_0)$, Theorem~\ref{monotonicity theorem} guarantees that every $(u(s), \Lambda(s))$ is strictly monotone.   

Consider now the behavior as $s \to \infty$.  Assume that the heteroclinic degeneracy alternative~\ref{gen hetero degeneracy} does not occur.
  Since the solutions on $\cm$ are strictly monotone, we conclude from Lemma~\ref{compactness fronts lemma} that $\iftcm$ is locally pre-compact in $\Xspace_\infty \times \mathbb{R}^2$.  The curve $\borderedcm$ is likewise locally pre-compact, and hence the loss of compactness alternative~\ref{K loss of compactness alternative} cannot occur for it.  

  Suppose further that the spectral degeneracy alternative \ref{gen ripples} does not occur. Then if $s_n \to \infty$ is a sequence with $\sup_n N(s_n) < \infty$ and $(u(s_n),\Lambda(s_n))$ converging to $(u_*,\Lambda_*) \in \genU_\infty$, then the spectral condition \eqref{spectral assumption} must be satisfied at this limit. By Corollary~\ref{infinity spaces fredholm corollary}, the linearized operator $\mathscr{F}_u(u_*,\Lambda_*) \colon \Xspace_\infty \to \Yspace_\infty$ is therefore Fredholm index $0$.  Thus, along the curve $\borderedcm$, the loss of Fredholmness alternative~\ref{K loss of fredholmness alternative} cannot happen.   
 
  Recalling Lemma~\ref{global ift K lemma}, this winnows the possible limiting behavior of $\borderedcm$ to either \ref{K blowup alternative} or \ref{K loop}. For $\iftcm$, these lead to the blowup alternative~\ref{gen blowup alternative} or closed loop~\ref{gen loop}, respectively.  This proves part~\ref{gen ift alternatives}, and applying the same reasoning to the limit $s \to -\infty$ yields part~\ref{gen ift other direction alternatives}.
  Part~\ref{gen ift analyticity} is given directly by Lemma~\ref{global ift K lemma}\ref{bordered ift analyticity}.  Finally, if $\mathscr{J}$ is a curve of monotone front solutions as in part~\ref{gen ift maximality}, then by Lemma~\ref{fronts are Xinfty lemma} it lies in $\genU_\infty$.  The result then follows from Lemma~\ref{global ift K lemma}\ref{bordered ift maximality} and the proof of the theorem is complete.
\end{proof}

Next we turn to Theorem~\ref{general global bifurcation theorem} on global continuation of a local curve of strictly monotone fronts.  

\begin{proof}[Proof of Theorem~\ref{general global bifurcation theorem}]
  Let  $\cm_\loc$ be given as in the statement of the theorem.  Recalling Lemma~\ref{fronts are Xinfty lemma}, we have that $u(\varepsilon) \in \Xspace_\infty$ for each $\varepsilon \in (0,\varepsilon_0)$.  Fix any $\varepsilon_1 \in (0,\varepsilon_0)$ and set $(u_1, \lambda_1) := (u(\varepsilon_1), \lambda(\varepsilon_1))$.  Then there is a unique $\tau_1 \in \mathbb{R}$ such that such that $\mathcal{C} u_1(\placeholder-\tau_1, \placeholder) = 0$, where $\mathcal{C}$ is the functional in \eqref{def C functional}. Let
\[ c(\varepsilon, \tau) := \mathcal{C}\left[ u(\varepsilon)(\placeholder - \tau, \placeholder) \right] \qquad \textrm{for } (\varepsilon,\tau) \in  (0,\varepsilon_0) \times \mathbb{R}.\]
By construction, $c(\varepsilon_1,\tau_1) = 0$, and we compute that 
\[ \partial_\tau c(\varepsilon_1,\tau_1) = -\partial_x u(\varepsilon)(-\tau_1, y_0) \neq 0.\]
The implicit function theorem furnishes a $C^0$ curve $\tau = \tau(\varepsilon)$, defined for $|\varepsilon - \varepsilon_1| \ll 1$, so that $c(\varepsilon, \tau(\varepsilon)) = 0$.  Repeating this argument at each parameter value ensures that there exist an extension of $\tau(\varepsilon)$ to all of $(0,\varepsilon_0)$.  We will identify $\cm_\loc$ with its image under this family of translations, so that $\mathcal{C}$ vanishes along it.  Naturally, the hypotheses \eqref{kernel assumption}, \eqref{spectral assumption}, and \eqref{local singular assumption} are unaffected.

We claim, moreover, that this curve is locally real analytic.  To see this, fix again any $\varepsilon_1 \in (0,\varepsilon_0)$ and put $(u_1, \lambda_1) := (u(\varepsilon_1), \lambda(\varepsilon_1))$.  As we have argued before, hypothesis~\eqref{spectral assumption} and Corollary~\ref{infinity spaces fredholm corollary} imply that $\F_u(u_1, \lambda_1)$ is Fredholm index $0$ as a map $\Xspace_\infty \to \Yspace_\infty$.  In view of \eqref{kernel assumption}, there exists $\chi_1 \in \Yspace_\infty \setminus \{ 0\}$ such that 
\[ \Yspace_\infty = \range{\F_u(u_1, \lambda_1)} \oplus \mathbb{R} \chi_1.\]

Consider the two-parameter bordered problem 
\[ \mathscr{G}(w,\lambda) = 0\]
where, as in the proof of Theorem~\ref{global ift}, we are writing $w = (u,\mu) \in \Xspace_\infty \times \mathbb{R}$, but now define  
\[ \mathscr{G} \colon \mathcal{W} \subset ( \Xspace_\infty \times \mathbb{R}) \times \mathbb{R} \longrightarrow \Yspace_\infty \times \mathbb{R} \qquad (u, \lambda, \mu) \mapsto  \begin{pmatrix} \F(u,\lambda) + \mu \chi_1 \\ \mathcal{C} u \end{pmatrix}, \]
with the open set $\mathcal{W}$ derived from $\genU_\infty$ in the same way.  The local curve $\cm_\loc$ trivially lifts to a local $C^0$ curve of solutions to the bordered problem 
\[ \borderedcm_\loc := \left\{ (w(\varepsilon), \lambda(\varepsilon)) : \varepsilon \in (0,\varepsilon_0)  \right\} \subset \mathscr{G}^{-1}(0)\] 
where $w(\varepsilon) := (u(\varepsilon), 0)$.
Clearly $\mathscr{G}$ is real analytic, and letting $w_1 = (u_1,0)$, we compute that
\[ \mathscr{G}_w(w_1, \lambda_1) = \begin{pmatrix} \F_u(u_1, \lambda_1)  & \chi_1 \\ \mathcal{C} & 0 \end{pmatrix}. \]
By Fredholm bordering, this is a Fredholm index $0$ mapping $\Xspace_\infty \times \mathbb{R} \to \Yspace_\infty \times \mathbb{R}$.  Moreover, if $(\dot w, \dot \mu)$ is in its kernel, then 
\[ \F_u(u_1, \lambda_1) \dot u + \dot\mu \chi_1 = 0,\]
and hence $\dot\mu = 0$,  $\dot u \in \linspan{\{\partial_x u_1\}}$, and $\mathcal{C} \dot u = 0$.   But $u_1$ is a strictly monotone front, and so this forces $\dot u = 0$.  Thus $\mathscr{G}_w(w_1, \lambda_1)$ is an isomorphism.  The analytic implicit function theorem implies that the zero-set of $\mathscr{G}$ locally consists of a curve through $(w_1, \lambda_1)$ that is real analytic, and, by uniqueness, coincides with $\borderedcm_\loc$.  This confirms that at every point, $\borderedcm_\loc$ (and hence $\cm_\loc$) admits a local real-analytic reparameterization.  In fact, the same argument also shows that there can be no secondary bifurcation points along $\cm_\loc$; it locally comprises the complete zero-set of $\mathscr{F}$.

Fixing $\varepsilon_1 \in (0,\varepsilon_0)$ once more and defining $\mathscr G$ as above, we also can apply Theorem~\ref{homoclinic global ift} to the bordered problem at $(w_1, \lambda_1)$, yielding a global curve $\borderedcm \subset \mathscr{G}^{-1}(0)$.  As it is maximal in $\mathcal{W}$, we must have that $\borderedcm_\loc \subset \borderedcm$.  But note that $\mu = 0$ for each $(w,\lambda) \in \borderedcm_\loc$, and so local analyticity implies that $\mu = 0$ on the entirety of $\borderedcm$.  Projecting away this trivial component gives the desired global curve $\cm \subset \mathscr{F}^{-1}(0)$ extending $\cm_\loc$. 

In light of hypothesis~\eqref{local singular assumption}, we can reparameterize the curve so that 
  \[ (u(s), \lambda(s)) \to (u_0,\lambda_0) \qquad \textrm{as } s \to {0+},\]
  where $(u_0,\lambda_0)$ does not satisfy \eqref{spectral assumption}.
  This in particular means that $\cm$ (and hence $\borderedcm$) is not a closed loop.  Part~\ref{gen alternatives} now follows exactly as in the poof of Theorem~\ref{global ift}.  The local real analyticity and maximality of $\cm$ is likewise a consequence of Theorem~\ref{homoclinic global ift}.  Finally, \ref{gen reconnect} can be established as in the proof of \cite[Theorem~6.1(c)]{chen2018existence}. 
\end{proof}

\section{Application to variational and transmission problems} \label{variational section}

\subsection{Invariant quantities and conjugate flows for variational problems}
In many applications, the equation \eqref{fully nonlinear elliptic pde} formally arises as a variational problem
\begin{equation}\label{variational principle}
  \delta \int_\Omega \mathcal L(y,u,\nabla u, \Lambda)\, dx\, dy = 0,
\end{equation}
where the variations are taken with respect to $u$  vanishing on $\Gamma_0$, and $\mathcal L$ is regular enough that the mapping
\begin{equation}
  \label{var regularity elliptic coef} 
  \begin{aligned}
    (z,\xi,\Lambda) &\longmapsto \mathcal L(\placeholder, z, \xi, \Lambda),
    \qquad 
    \mathcal W \longrightarrow     C^{k+2+\alpha}_\bdd(\overline{\Omega^\prime})
  \end{aligned}
\end{equation}
is real analytic for some open set $\mathcal W \subset \R \times \R^n \times \R^n$ related to $\genU$. In this case \eqref{fully nonlinear elliptic pde} has the form
\begin{equation}
  \label{quasilinear variational pde}
  \left\{ 
  \begin{aligned}
    \nabla \cdot \mathcal L_\xi(y,u,\nabla u,\Lambda)
    - \mathcal L_z(y,u,\nabla u,\Lambda)
    &= 0 \qquad \textrm{in } \Omega, \\
    \nu \cdot \mathcal L_\xi(y,u,\nabla u,\Lambda)  & = 0 \qquad  \textrm{on } \Gamma_1, \\
    u & = 0 \qquad \textrm{on } \Gamma_0,
  \end{aligned} 
  \right.
\end{equation}
where as always $\nu=\nu(y)$ is the outward normal to $\Gamma_1$.

The variational principle \eqref{variational principle} is invariant under translations in $x$, and so one expects a corresponding conserved quantity, which is the content of the following lemma.
\begin{lemma}[Conserved quantity] \label{conserved quantity lemma}
  If $(u,\Lambda) \in \genU$ solves \eqref{quasilinear variational pde}, then the integral
  \begin{align*}
    \flowforce(u,\Lambda;x) :=  
    \int_{\Omega'} \Big(\mathcal L\big(y,u,\nabla u,\Lambda\big) - \mathcal L_{\xi_1}\big(y,u,\nabla u,\Lambda\big) u_x \Big)\, dy
  \end{align*}
  is a constant independent of $x$.
  \begin{proof}
    First differentiating under the integral and then using \eqref{quasilinear variational pde} to eliminate $\partial_x \mathcal L_\xi$, we find that
    \begin{align*}
      \frac d{dx} \flowforce 
      &=  
      \int_{\Omega'} \Big(\mathcal L_z u_x  + \mathcal L_{\xi_1} u_{xx} + \mathcal L_{\xi'} \cdot \nabla_y u_x
      - (\partial_x\mathcal L_{\xi_1}) u_x - \mathcal L_{\xi_1} u_{xx} \Big)\, dy\\
      &=  
      \int_{\Omega'} \Big( (\nabla_y \cdot \mathcal L_{\xi'}) u_x + \mathcal L_{\xi'} \cdot \nabla_y u_x\Big)\, dy\\
      &=  
      \int_{\Omega'} \nabla_y \cdot \big(\mathcal L_{\xi'} u_x \big)\, dy
      =  
      \int_{\Gamma_0' \cup \Gamma_1'} \big(\nu \cdot \mathcal L_\xi \big) u_x \, dS(y).
    \end{align*}
    Since $\nu \cdot \mathcal{L}_\xi$ vanishes on $\Gamma_1$ while $u$ and hence $u_x$ vanish on $\Gamma_0$, we obtain $\partial_x \flowforce = 0$ as desired.
  \end{proof}
\end{lemma}

Sending $x\to\pm\infty$ for a front solution, we obtain the following corollary.

\begin{corollary}\label{conjugate corollary}
  Suppose that $(u,\Lambda) \in \genU_\infty$ solves $\F(u,\Lambda) = 0$, and let $U_\pm$ be the limits of $u$ as $x \to \pm\infty$. Then $U_\pm \in C^{k+2}(\Omega')$, 
  \begin{enumerate}[label=\rm(C\arabic*)]
  \item \label{conjugate solution part} $\F(U_\pm,\Lambda) = 0$, and
  \item \label{conjugate flowforce part} $\flowforce(U_+,\Lambda) = \flowforce(U_-,\Lambda)$.
  \end{enumerate}
\end{corollary}
\begin{proof}
  The regularity of $U_\pm$ and \ref{conjugate solution part} were already shown in Lemma~\ref{fronts are Xinfty lemma}. The remaining statement \ref{conjugate flowforce part} follows at once from Lemma~\ref{conserved quantity lemma},
  \begin{equation*}
    \flowforce(U_+,\Lambda) = \lim_{x\to+\infty} \flowforce(u,\Lambda;x)
    = \lim_{x\to-\infty} \flowforce(u,\Lambda;x) = \flowforce(U_-,\Lambda).
    \qedhere
  \end{equation*}
\end{proof}

\begin{definition}[Conjugate flows]\label{conjugate definition}
  For a fixed parameter value $\Lambda$, we say that two distinct functions $U_\pm \in C^{k+2}(\Omega')$ are \emph{conjugate} (or \emph{conjugate flows}) if they satisfy \ref{conjugate solution part} and \ref{conjugate flowforce part}. 
\end{definition}

The terminology ``conjugate flows'' comes from steady hydrodynamics; see~\cite{benjamin1971unified}. By studying the conjugate flow problem in our applications below, we will be able to rule out heteroclinic degeneracy \ref{gen hetero degeneracy} using the following lemma.

\begin{lemma}[Triple conjugacy]\label{triple conjugacy lemma}
  Suppose that the PDE \eqref{fully nonlinear elliptic pde} has the variational structure \eqref{quasilinear variational pde}. Then in Lemma~\ref{compactness fronts lemma}\ref{triple conjugacy alt} the three distinct limiting states
   \begin{equation*}
     \lim_{x \to \mp\infty} u_*(x, \placeholder),
     \quad 
     \lim_{n \to \infty} \lim_{x \to +\infty} u_n(x, \placeholder),
     \quad 
     \lim_{n \to \infty} \lim_{x \to -\infty} u_n(x, \placeholder),
   \end{equation*}    
   are also conjugate in the sense of Definition~\ref{conjugate definition}.
\end{lemma}
\begin{proof}
  As in the proof of Lemma~\ref{compactness fronts lemma}, the limiting states $U$ above all lie in $C^{k+2+\alpha}(\overline{\Omega'})$ and solve $\F(U,\Lambda_*) = 0$. Moreover, by Lemma~\ref{conserved quantity lemma} they all have the same value of the conserved quantity $\flowforce$,
  \begin{equation*}
    \flowforce(U,\Lambda_*) = \lim_{n \to \infty} \flowforce(u_n,\Lambda_n; 0).
    \qedhere
  \end{equation*}
\end{proof}

\subsection{Transmission conditions} \label{sec:trans}
Finally, with only superficial modifications, the above global bifurcation theory applies to a broad class of quasilinear transmission problems.  Suppose that $\Omega = \Omega_1 \cup \Omega_2$ is a (slitted) cylinder with $\Omega_1 \cap \Omega_2 = \emptyset$.  Let $\Gamma_1 := \partial \Omega_1 \cap \partial \Omega_2 \neq \emptyset$ be their common boundary and set $\Gamma_0 := \partial \Omega \setminus \Gamma_1$.  We require that $\Gamma_0$ and $\Gamma_1$ are $C^{k+2+\alpha}$ and disjoint.  Set 
\begin{equation}
 \label{quasilinear transmission elliptic pde} 
  \begin{aligned}
    \mathcal{F}(y, u, \nabla u, D^2 u, \Lambda) &:= \nabla \cdot \A(y,u, \nabla u, \Lambda) + \B(y, u, \Lambda) \\
    \mathcal{G}(y,u, \nabla u, \Lambda) & := \jump{ - \nu\cdot \A(y,u, \nabla u, \Lambda)} + \mathcal{E}(y,u, \Lambda) ,
  \end{aligned}
\end{equation}
where $\mathcal{A}, \mathcal{B}, \mathcal{E}$ are such that the same regularity \eqref{gen regularity elliptic coef} and ellipticity conditions \eqref{ellipticity assumption}  hold, where the function spaces are naturally adapted to the current case.  The main novelty is the introduction of $\jump{\placeholder}=(\placeholder)_2-(\placeholder)_1$, which denotes the jump of a quantity across the interface $\Gamma_1$ from $\Omega_1$ to $\Omega_2$.
As before, we express \eqref{quasilinear transmission elliptic pde} as the abstract operator equation
\[ \F(u,\Lambda) = 0\]
where $\F \colon \genU \subset \Xspace_\bdd \times \mathbb{R}^m \to \Yspace_\bdd$ is real analytic, 
and the basic spaces are redefined to be
\begin{equation}
  \begin{aligned}
    \Xspace &:= \left\{ u \in C^0(\overline{\Omega}) \cap C^{k+2+\alpha}(\overline{\Omega_1}) \cap C^{k+2+\alpha}(\overline{\Omega_2})  : u|_{\Gamma_0} = 0 \right\} \\ 
    \Yspace & = \Yspace_1 \times \Yspace_2 := \left( C^{k+\alpha}(\overline{\Omega_1}) \cap C^{k+\alpha}(\overline{\Omega_2}) \right)  \times C^{k+1+\alpha}(\Gamma_1),
  \end{aligned}
  \label{transmission spaces} 
\end{equation}
 with $\Xspace_\bdd$, $\Xspace_\infty$, $\Yspace_\bdd$, and $\Yspace_\infty$ given analogously.

\begin{corollary}[Global continuation for transmission problems]\label{cor_transmission GBT}
Consider the quasilinear elliptic transmission problem \eqref{quasilinear transmission elliptic pde}. 
\begin{enumerate}[label=\rm(\alph*)]
\item \label{transmission global ift part} \textup{(Global IFT)} If there are two parameters $\Lambda = (\lambda, \mu)$ and $(u, \lambda, \mu) \in \genU$ is a strictly monotone front solution satisfying \eqref{kernel assumption}, \eqref{spectral assumption}, and \eqref{transversality condition}, then there exists a global curve $\cm$ of strictly monotone front solutions to \eqref{quasilinear transmission elliptic pde} such that the properties \ref{gen ift alternatives}--\ref{gen ift maximality} of Theorem~\ref{global ift} hold.
\item \label{transmission global bifurcation part} \textup{(Global bifurcation)} If there is one parameter $\Lambda = \lambda$ and $\cm_\loc \subset \genU$ is a local curve of strictly monotone front solutions satisfying \eqref{kernel assumption}, \eqref{spectral assumption}, and \eqref{local singular assumption}, then there exists a global curve $\cm$ of strictly monotone front solutions extending $\cm_\loc$ and exhibiting the properties \ref{gen alternatives}--\ref{gen reconnect} of Theorem \ref{general global bifurcation theorem}. 
\end{enumerate}
\end{corollary}
\begin{proof}
We will prove this result by verifying the results of Section~\ref{monotonicity section}--Section~\ref{fredholm theory section} in the current setting of \eqref{quasilinear transmission elliptic pde} and with the spaces \eqref{transmission spaces}.

Note first that the linearization of \eqref{quasilinear transmission elliptic pde} at $u \in \Xspace_\infty$ results in an elliptic operator of the form of \eqref{appendix L operator transmission}, which allows one to characterize the maximum principle and Hopf boundary-point lemma in the same way as before. Therefore the spectral condition \eqref{spectral assumption} together with Corollary~\ref{cor pe transmission} guarantees the existence of a positive comparison function as in Lemma~\ref{supersolution lemma}. The same ``quasilinearization" technique (using Lemma~\ref{lem hopf} in place of the one-phase Hopf lemma) leads to an asymptotic monotonicity result of the form Lemma~\ref{asymptotic monotonicity lemma}.  It is then immediate that Theorem~\ref{monotonicity theorem} holds for the transmission problem. 

The characterization of the loss of compactness, namely Lemma~\ref{compactness fronts lemma}, relies only on Schauder estimates and so it follows in exactly the same way.

Lastly, we show in Appendix \ref{principal eigenvalue appendix transmission} that all the prerequisites for the Fredholm theory lemmas in Section \ref{fredholm theory section} still hold true in the setting of \eqref{quasilinear transmission elliptic pde}. The arguments in Section~\ref{general theorem proof section} now apply verbatim, and we obtain the desired results.
\end{proof}

\section{Large fronts for semilinear Robin problems} \label{robin section}

In this section, we study a simple class of semilinear elliptic equations that is nonetheless rich enough to illustrate all of the alternatives in the global bifurcation theory.  Specifically, consider the PDE  
\begin{equation}
 \label{toy problem} \left\{ 
  \begin{aligned}
    \Delta u & = 0 & \qquad & \textrm{in } \Omega \\
    u_y - u + g(u,\lambda) & = 0 & \qquad & \textrm{on } \Gamma_1 \\
    u &=  0 & \qquad &\textrm{on } \Gamma_0,
  \end{aligned} \right.
\end{equation}
set on the infinite cylinder $\Omega := \mathbb{R} \times (0,1)$ with upper boundary $\Gamma_1 := \{ y = 1\}$ and lower boundary $\Gamma_0 := \{ y =0 \}$.  Here $g = g(z,\lambda)$ is assumed to be real analytic in its arguments and $\lambda$ is a real parameter.

Under certain natural hypotheses on $g$, we are able to construct global curves of solutions to \eqref{toy problem}.  By tailoring the choice of nonlinearity, we can in fact control precisely which of the limiting scenarios in Theorem~\ref{general global bifurcation theorem}\ref{gen alternatives} occurs, which shows that they are \emph{sharp}.
 
Typically, it is very difficult to infer the behavior of solutions along a global bifurcation curve.  What makes it possible here is an illuminating analogy between the conjugate flow problem \ref{conjugate solution part}--\ref{conjugate flowforce part} for the PDE \eqref{toy problem} and  classical finite-dimensional Hamiltonian mechanics.  Recall that the ODE  
\begin{equation}
  \ddot q = -\nabla V(q),\label{toy ODE} 
\end{equation}
has conserved energy $\frac{1}{2} |\dot q|^2 + V(q)$.  Clearly, the critical points of the potential $V$ are rest points for the system.  Likewise, a heteroclinic orbit can connect two rest points only if they are on the same level set of  $V$.   

We first observe that the PDE \eqref{toy problem} is of the variational form \eqref{quasilinear variational pde} with 
\[ 
  \mathcal L(y,z,\xi,\eta,\lambda)
  = \frac{\xi^2+\eta^2}2 - z\eta + g(z,\lambda)\eta.
\]
Hence by Lemma~\ref{conserved quantity lemma} there is an ``energy''   
\begin{align*}
  \flowforce(u, \lambda; x) &:= \int_0^1 \Big( \mathcal L(y,u,u_x,u_y,\lambda) - \mathcal L_\xi(y,u,u_x,u_y,\lambda)u_x \Big)\, dy
  \\&\phantom{:}= \int_0^1 \Big( \frac {u_y^2-u_x^2}2 - uu_y + g(u,\lambda)u_y \Big)\, dy
\end{align*}
that will be independent of $x$ for any solution $(u,\lambda)$.  Introducing the primitive $G = G(z,\lambda)$ of $g(\placeholder, \lambda)$ vanishing when $z=0$, and recalling the Dirichlet condition on $\Gamma_0 = \{y=0\}$, we can integrate the second two terms above to obtain the alternate formula
\begin{equation}\label{toy flow force}
  \flowforce(u, \lambda; x) =
  \frac 12 \int_0^1 \left( u_y^2(x,y)-u_x^2(x,y) \right) \, dy - \frac 12 u^2(x,1) + G(u(x,1),\lambda).
\end{equation}
We can then use $\flowforce$ to characterize the conjugate flows in the sense of Definition~\ref{conjugate definition}.  It is easy to confirm that $U = U(y)$ solves \eqref{toy problem} if and only if $U = r y$, for $r$ a root of $g(\placeholder, \lambda)$. The corresponding energy \eqref{toy flow force} for $U$ is simply
\begin{equation}
  \flowforce(U, \lambda) = G(r, \lambda). \label{toy flow force G relation} 
\end{equation}
  
Thus the conjugate flows for \eqref{toy problem} have an easy visual interpretation:  they are linear functions whose slopes are (distinct) critical points of $G$ that lie on the same level set.  In that sense, $G$ plays an analogous role to the potential $V$ in the Hamiltonian system \eqref{toy ODE}.  

With that in mind, we now impose some structural assumptions on $g$ and $G$.  As we wish to construct heteroclinic orbits, suppose that
\begin{equation}
  g(0, \lambda) = 0, \quad g(\lambda, \lambda) = 0 \qquad \textrm{for all }  \lambda \in \mathbb{R}, \label{g assumptions} 
\end{equation}
and
\begin{equation}
  \begin{alignedat}{2}
    &G(0, \lambda) = G(\lambda, \lambda) = 0 &\qquad& \textrm{for all }  \lambda \in \mathbb{R},\\
     &G_{rr}(0, \lambda), \, G_{rr}(\lambda,\lambda)  > 0  && \text{for all } \lambda \in \mathcal O \setminus \{0\},\\
      &G(\placeholder,\lambda)^{-1}(0) \cap g(\placeholder, \lambda)^{-1}(0) = \{0, \lambda \}
    &\qquad& \textrm{for all } \lambda \in \mathcal{O},
  \end{alignedat}
    \label{G assumptions} 
\end{equation}
  for some neighborhood of the origin $\mathcal{O} \subset \mathbb{R}$.  Note that \eqref{g assumptions} ensures that $0$ and $\lambda y$ are $x$-independent solutions for all $\lambda$, while the first line of \eqref{G assumptions} says that they are conjugate.   As we show in Lemma~\ref{toy ripple lemma}, the convexity of $G$ asked for in \eqref{G assumptions} is equivalent to the spectral assumption \eqref{spectral assumption}, while the remaining condition is related to heteroclinic degeneracy \ref{gen hetero degeneracy}.

We can now state the main result.  As in the previous section, choose $\alpha \in (0,1)$ and take 
\[ \Xspace := \left\{ u \in C^{2+\alpha}(\overline{\Omega}) : u|_{\Gamma_0} = 0 \right \}, \qquad \Yspace := C^{\alpha}(\overline{\Omega}) \times C^{1+\alpha}(\mathbb{R}) \]
with the spaces $\Xspace_\bdd$, $\Yspace_\bdd$, $\Xspace_\infty$, $\Yspace_\infty$ defined accordingly.  Let $\F \colon \Xspace_\bdd \times \mathbb{R}^2 \to \Yspace_\bdd$ be the nonlinear operator corresponding to the PDE \eqref{toy problem}.   

\begin{theorem}[Global bifurcation] \label{toy global bifurcation theorem} Consider the semilinear elliptic problem \eqref{toy problem}.   Assume that the structural conditions \eqref{g assumptions}, \eqref{G assumptions} hold, and also that 
  \begin{equation}
    g_{zz\lambda}(0,0) < 0.  \label{toy nondegeneracy assumption} 
  \end{equation} 
Then there exist global continuous curves $\cm^\pm$ of monotone front solutions to \eqref{toy problem} lying in $\Xspace_\infty$.  
  \begin{enumerate}[label=\rm(\alph*)]
  \item \label{toy global bifurcation monotone part} On $\cm^+$, all the fronts are strictly monotone increasing while on $\cm^-$ they are strictly monotone decreasing.  
  \item \label{toy global bifurcation alternatives part} As one follows $\mathscr{C}^\pm$, one of the three alternatives in Theorem~\ref{general global bifurcation theorem}\ref{gen alternatives} must occur.
  \item \label{toy global bifurcation uniqueness part} Moreover, the curves $\mathscr C^\pm$ leave a small neighborhood of $(0,0)$ in $\Xspace_\infty \times \R$ and do not re-enter.
  \end{enumerate}
\end{theorem}

The proof of this theorem  represents a fairly concise application of Theorem~\ref{general global bifurcation theorem}.  In Section~\ref{toy local section}, we begin the process by constructing local curves $\cm_\loc^\pm$.  The hypotheses of the general theory are then verified in Section~\ref{toy global section}, furnishing Theorem~\ref{toy global bifurcation theorem}.  Finally, in Section~\ref{toy alternatives section}, we discuss the realizability of the alternatives.    

\subsection{Small-amplitude theory} \label{toy local section}

As the first step, we will prove the existence of small-amplitude monotone front solutions to \eqref{toy problem}.  The mechanical intuition makes it clear how to proceed.  At $\lambda = 0$, $G$ has a unique critical point corresponding to the trivial solution $u = 0$.  As $\lambda$ moves to the left or right, a second rest point develops that will be conjugate to $0$.  We therefore seek a local curve of heteroclinics solutions parameterized by $\lambda$ that bifurcate at $\lambda =0$ from the trivial solution.

\begin{theorem}[Small-amplitude fronts] \label{toy small amplitude theorem}  Under the hypotheses of Theorem~\ref{toy global bifurcation theorem}, there exists a $C^0$ curve $\cm_\loc$ of solutions $(u,\lambda)$ to \eqref{toy problem} that admits the parameterization
\[ \cm_\loc = \left\{ \left( u(\lambda), \lambda \right) :  |\lambda| < \varepsilon_0 \right\} \subset \Xspace_\infty \times \mathcal{O},  \]
for some $\varepsilon_0 > 0$.   Moreover, the following statements hold along $\cm_\loc$.
\begin{enumerate}[label=\rm(\alph*)]
\item \label{toy asymptotics part} \textup{(Asymptotics)} The solutions have leading form expressions given by
\begin{equation}
 \label{toy problem local asymptotics} 
    u(\lambda)(x,y)  = \frac{\lambda}{2} \left(1+ \tanh{(\kappa_1 |\lambda| x)}\right) y   + O(\lambda^2)  \qquad \textrm{in } \Xspace_\bdd 
\end{equation}
for an explicit positive constant $\kappa_1$ given in \eqref{toy def coeff}.   

\item \label{toy local monotonicity part} \textup{(Monotonicity)}  If $\lambda > 0$, then $(u(\lambda), \lambda)$ is a strictly increasing monotone front, while for $\lambda < 0$, it is  a strictly decreasing monotone front.  

\item \label{toy uniqueness part} \textup{(Uniqueness)} In a neighborhood of $(0,0)$ in $\Xspace_\infty \times \mathbb{R}$, $\cm_\loc$ comprises all monotone fronts (up to translation).

\item \label{toy kernel part} \textup{(Kernel)}  For all $0 < |\lambda| < \varepsilon_0$, the kernel of $\F_u(u(\lambda), \lambda)$ is one dimensional and generated by $\partial_x u(\lambda)$.    

\end{enumerate} 
\end{theorem}

Our approach is based on the center manifold reduction theory recently introduced in \cite{chen2022center}, which is more convenient here than traditional spatial dynamics methods.  First, it is formulated in H\"older spaces as required by the global theory.  As we shall see shortly, it also allows us to compute the reduced equation through a power series expansion that is particularly simple for the present problem.  It also gives us the freedom to choose the projection onto the kernel of the linearized operator, which is useful in establishing the monotonicity of solutions.

\subsubsection*{Notation}
In order to construct the center manifold, it will be necessary to (temporarily) expand our function spaces to include solutions exhibiting some growth at infinity.  For $\nu \in \mathbb{R}$, we define the exponentially weighted H\"older space 
\[ C_\nu^{k+\alpha}(\overline{\Omega}) := \left\{ u \in C^{k+\alpha}(\overline{\Omega}) : \| u \|_{C_\nu^{k+\alpha}(\Omega)} < \infty \right\}, \]
where
\[ \| u \|_{C_\nu^{k+\alpha}(\Omega)} := \sum_{|\beta| \leq k} \| \sech{(\nu x)} \partial^\beta u \|_{C^0(\Omega)} + \sum_{|\beta| = k} \| \sech(\nu x) |\partial^\beta u|_\alpha \|_{C^0(\Omega)}, \]
and $|\placeholder|_\alpha$ is the local H\"older seminorm
\[ |u|_\alpha(x,y) := \sup_{(\tilde x, \tilde y) \in B_1(x,y) \cap \Omega} \frac{| u(\tilde x, \tilde y) - u(x,y)|}{ |(\tilde x -x, \tilde y -y)|^\alpha}. \]
Let $\Xspace_\nu$ and $\Yspace_\nu$ denote the corresponding versions of $\Xspace_\bdd$ and $\Yspace_\bdd$.  

\subsubsection*{Center manifold reduction}

Using \eqref{g assumptions} we compute that  the linearized transversal operator at $(u,\lambda) = (0,0)$ is given by  
\[ \limL_-^\prime  \colon \Xspace^\prime \to \Yspace^\prime \qquad w \mapsto \begin{pmatrix}  w_{yy} \\  \left( w_y - w  \right)\Big|_{\Gamma_1} \end{pmatrix}.\]
It is easily confirmed that $0$ is a simple eigenvalue of $\limL_-^\prime$ with corresponding eigenfunction $\varphi_0 = \varphi_0(y) := y$, while the remainder of the spectrum is strictly negative. 
As a consequence, if $\mathscr{F}_u(0,0)$ is viewed as a mapping $\Xspace_\nu \to \Yspace_\nu$, for $0 < \nu \ll 1$, its null space is two dimensional and characterized by
\[  \ker\F_u(0,0) \colon \Xspace_\nu \to \Yspace_\nu = \left\{ \left( A + B x\right) \varphi_0 \in \Xspace_\nu : (A,B) \in \mathbb{R}^2 \right \}.\]
A convenient projection $\FSproj$ onto this kernel is found by setting $(A,B)$ to be $(u, u_x)$ evaluated at the chosen point $(0,1) \in \Gamma_1$:  
\[ \FSproj u := \left( u(0,1) +  u_x(0,1) x \right) \varphi_0(y).\]

We have now verified the hypotheses of the center manifold reduction theorem \cite[Theorem 1.1]{chen2022center}, the conclusions of which are recorded below.

\begin{lemma}[Center manifold] There exists $\nu > 0$, neighborhoods $\mathscr N \subset \Xspace_\bdd \times \mathcal{O}$ and $N \subset \mathbb{R}^3$ and a coordinate map $\Psi = \Psi(A,B, \lambda)$ satisfying 
  \[ \Psi \in C^4(\mathbb{R}^3, \Xspace_\nu), \qquad \Psi(0,0,\lambda) = 0 \quad \textrm{for all } \lambda, \qquad \Psi_A(0,0,0) = \Psi_B(0,0,0) = 0,\]
such that the following hold
\begin{enumerate}[label=\rm(\alph*)]
\item Suppose that $(u, 0, \lambda) \in \mathscr N$ solves \eqref{toy problem}.  Then $v := u(\placeholder, 1)$ solves the second-order ODE
  \begin{equation}
    v^{\prime\prime} = f(v, v^\prime, \lambda), \label{gen reduced equation} 
  \end{equation}
where $f \colon \mathbb{R}^3 \to \mathbb{R}$ is the $C^4$ mapping 
\[ f(A,B, \lambda) := \frac{d^2}{dx^2} \Big|_{x=0} \Psi(A,B,\lambda)(x,0).\]
\item Conversely, if $v$ satisfies the ODE \eqref{gen reduced equation} and $(v(x), v^\prime(x), \lambda) \in N$ for all $x$, then $v = u(\placeholder, 1)$ for a solution $(u,\lambda) \in \mathscr N$ of the PDE \eqref{toy problem}.  Moreover, 
\[ u(x+\tau, y) = (\FSproj u)(\tau,y) + \Psi(v(x), v^\prime(x), \lambda)(\tau, y), \qquad \textrm{for all } \tau \in \mathbb{R}.\]
\end{enumerate}
\end{lemma} 

The next step is to compute $f$ and the corresponding reduced equation \eqref{gen reduced equation} for the problem at hand.  Our approach relies on \cite[Theorem 1.6]{chen2022center}, which states that $\Psi$ admits the Taylor expansion 
\[ \Psi(A,B,\lambda) = \sum_{\substack{ 2 \leq i+j+k \leq 3 \\ i + j \geq 1}} \Psi_{ijk} A^i B^j \lambda^k + O\Big( ( |A| + |B| ) ( |A| + |B| + |\lambda| )^3 \Big) \qquad \textrm{in } \Xspace_\nu.  \]
Here the coefficients $\Psi_{ijk}$ are the unique functions in $\Xspace_\nu$ with $\FSproj \Psi_{ijk} = 0$ and 
\[ \partial_A^i \partial_B^j \partial_\lambda^k\Big|_{(A,B,\lambda) = 0} \F( (A+Bx) \varphi_0 + \Psi(A,B,\lambda),  \lambda) = 0 \qquad \textrm{for all } i+j+k \leq 3, \]
where the derivatives above are taken in the formal G\^ateaux sense. 

This leads to a hierarchy of linear equations of the general form
\[ \left\{ \begin{aligned}
\F_u(0,0) \Psi_{ijk} & = R_{ijk} \\
\FSproj \Psi_{ijk} & = 0,
\end{aligned} \right.\]
for some $R_{ijk} \in \Yspace_\mu$ that are explicit given previously computed terms; for more discussion, see \cite[Section 4]{chen2022center}. By Fredholm theory \cite[Lemma 2.3]{chen2022center}, these problems are uniquely solvable in $\Xspace_\nu$.  Following this procedure, we ultimately find that 
\begin{equation}
  f(A,B,\lambda) =   3g_{12} \lambda^2 A +  3 g_{21} \lambda A^2 + 3 g_{30} A^3  + r(A, B, \lambda)  \label{toy reduced ODE} 
\end{equation}
where $g_{\ell m} := (\partial_z^\ell \partial_\lambda^m  g)(0,0)/(\ell!m!)$, and $r \in C^3$ is a remainder term with 
\[ r(A,B,\lambda) = O\Big( |A| \big( |A| + |B|^{1/2} +|\lambda| \big)^3 + |B| \big(|A| + |B|^{1/2} +|\lambda| \big)^2 \Big).\] 
Notice that we do not see any $B$ dependence in the (expected) leading order part of $f$; this is a consequence of the invariance of \eqref{toy problem} under reflection in $x$.  Indeed, that symmetry actually guarantees that $f$ is even in $B$.  

The coefficients in \eqref{toy reduced ODE} are in fact related to one another due to the structural assumptions on $g$ and $G$ in \eqref{g assumptions} and \eqref{G assumptions}.  Differentiating  the identity $g(\lambda,  \lambda) = 0$ three times and evaluating at $\lambda = 0 $ gives
\[ g_{zzz} + 3g_{zz\lambda} + 3g_{z\lambda\lambda} = 0 \qquad \textrm{at } (0,0). \]
Similarly, differentiating the equation $G(\lambda, \lambda) = 0$ four times, and recalling that $G_z = g$, we find that
\[ g_{zzz} + 4g_{zz\lambda} + 6 g_{z\lambda\lambda} = 0 \qquad \textrm{at } (0,0).\]
Together, these imply that
\begin{equation}
  g_{12} = -\frac{1}{3} g_{21}, \qquad g_{30} = -\frac{2}{3} g_{21}.
  \label{toy coefficient relations} 
\end{equation} 

Taking $f$ and neglecting the remainder term gives the following truncated reduced equation on the center manifold:
\begin{equation}
  \label{toy reduced truncated ODE} 
  v_{xx}^{0} 
  = g_{21} \Bigl( -\lambda^2 v^0 + 3  \lambda (v^0)^2 - 2  (v^0)^3 \Bigr).
\end{equation}
Under the assumption \eqref{toy nondegeneracy assumption}, the above ODE has the explicit heteroclinic orbit 
\begin{equation}
  v^0(x) =  \lambda  \frac{1+\tanh{(\kappa_1 |\lambda| x)}}{2},\label{toy truncated heteroclinic} 
\end{equation}
where
\begin{equation}
  \kappa_1  := \frac{|g_{21}|^{1/2}}{2}.\label{toy def coeff} 
\end{equation}

It is also important to confirm that the conserved quantity $\flowforce$ for the full PDE has an analog on the center manifold.  Setting 
\[ h(A,B,\lambda) := \flowforce\big( (A+Bx) \varphi_0(y) + \Psi(A,B,\lambda); x\big), \]
we indeed have that, for any solution $v$ of the reduced ODE \eqref{toy reduced ODE}, $h(v, v^\prime, \lambda)$ is independent of $x$.  Expanding as before, we find that 
\begin{equation}
  \label{toy flow force expansion} 
  h(A,B,\lambda) = 
  \frac{g_{30}}{4}  A^4 + \frac{g_{21}}{3} \lambda A^3 + \frac{g_{12}}{2} \lambda^2 A^2  - \frac{1}{6} B^2 + \tilde r(A, B, \lambda),  
\end{equation} 
for a $C^4$ remainder 
\[ \tilde r(A,B,\lambda) = O\Big( |A| \big( |A| + |B|^{1/2} +|\lambda| \big)^4 + |B| \big(|A| + |B|^{1/2} +|\lambda| \big)^3 \Big).\] 
We are now prepared to prove the existence of small-amplitude monotone front solutions to \eqref{toy problem} 

 \begin{proof}[Proof of Theorem~\ref{toy small amplitude theorem}] 
 Working in the rescaled variables 
\[ x =: |\lambda|^{-1} X, \quad v(x) =: \lambda V(X), \quad v_x(x) =: \lambda |\lambda| W(X),\]
 the (full) reduced ODE \eqref{gen reduced equation} can be recast as the planar system 
\begin{equation}
  \left\{ \begin{aligned} 
    V_X &= W \\
  W_X &= 3g_{12} V + 3g_{21} V^2 + 3g_{30} V^3 + R(V, W, \lambda),\end{aligned} \right.\label{toy planar reduced ODE} 
\end{equation}
where the remainder term is
\[ R(V,W,\lambda) = O\big( |\lambda|(|V| + |W|) \big).\]

   When $\lambda = 0$, \eqref{toy planar reduced ODE} has the solution
\[ V^0(X) := \frac{1+\tanh(\kappa_1 X)}{2}, \qquad W^0(X) := \kappa_1\frac{ \sech^2(\kappa_1 X) }{2}\]
   shown in Figure~\ref{toy phase portrait figure}, which corresponds to the rescaling of $v^0$.  We must now show that this solution persists when $|\lambda| > 0$.    For that, we will need to make use of the conserved quantity
\[ H(V,W,\lambda) := \frac{1}{2} W^2  + \frac{g_{21}}{2}  V^2 - g_{21} V^3 + \frac{ g_{21}}{2}  V^4 + \tilde R(V,W, \lambda) \]
that results from composing $h$ with the scaling and using the relations \eqref{toy coefficient relations}.   One checks that the error term 
\[ \tilde R(V, W, \lambda) = O\big(|\lambda| (|V| + |W|) \big).\]  

\begin{figure}
  \centering
  \includegraphics[scale=1.1]{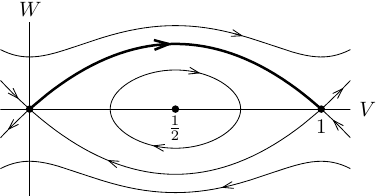}
  \caption{Phase portrait for the rescaled reduced ODE \eqref{toy planar reduced ODE} at $\lambda = 0$.  The thick curve is the heteroclinic orbit $(V^0,W^0)$ connecting the rest points $(0,0)$ and $(1,0)$.}
  \label{toy phase portrait figure}
\end{figure}

   At $\lambda = 0$, the system has rest points $(0,0)$ and $(1, 0)$.  They are connected via $(V^0,W^0)$ and lie on the level set $\{H(\placeholder,\placeholder,0) = 0\}$. For $|\lambda| > 0$, our assumptions \eqref{g assumptions}--\eqref{G assumptions} imply that $(0,0)$ and $(1,0)$ continue to be rest points and to lie on the level set $\{H(\placeholder,\placeholder,\lambda)=0\}$. The existence of a heteroclinic orbit $(V^\lambda, W^\lambda)$ between $(0,0)$ and $(1,0)$ then follows from the nondegeneracy of this level set. Undoing the scaling gives the leading order asymptotics stated in  \eqref{toy problem local asymptotics}, proving part~\ref{toy asymptotics part}.   
 
Next, consider the monotonicity claimed in part~\ref{toy local monotonicity part}. Analyzing the phase portrait of the reduced ODE, we easily verify that 
\[ ( \signum{\lambda} ) \partial_x u(\lambda)(\placeholder, 1) = ( \signum{\lambda} )\partial_x v(\lambda) > 0.\]
 But $\partial_x u(\lambda)$ is harmonic in $\Omega$ and vanishes on $\Gamma_0$ as well in the limits $x \to \pm\infty$.  The maximum principle then tells us that $(\signum{\lambda}) \partial_x u(\lambda) > 0$ in $\Omega \cup \Gamma_1$, meaning $u(\lambda)$ is strictly monotone.  
 
   To show \ref{toy uniqueness part}, first note that for $\lambda=0$, the conjugate flow analysis above together with \eqref{G assumptions} forces all fronts to vanish identically. For $\lambda > 0$, we likewise find that the only conjugate state to $0$ is $\lambda y$, corresponding to the rest points $(0,0)$ and $(1,0)$ of \eqref{toy planar reduced ODE}. Consider a monotone front $(u,\lambda) \in \mathscr N$ with $u \in \Xspace_\infty$. Then the corresponding orbit $(V,W)$ of \eqref{toy planar reduced ODE} must satisfy $0 \le V \le \lambda$ and $W \ge 0$. Since the phase portrait of \eqref{toy planar reduced ODE} is qualitatively the same as in Figure~\ref{toy phase portrait figure}, we conclude that the only possibility is that $u$ is a translate of $u(\lambda)$. The argument for $\lambda < 0$ is similar.

Finally, to prove~\ref{toy kernel part} we will use \cite[Theorem 1.6]{chen2022center} which reduces the issue  to the center manifold.  Specifically, this result tells us that $\dot u \in \ker\F_u(u, 0, \lambda)$ only if $\dot v := \dot u(\placeholder, 1)$ solves the linearized reduced equation
\[  \dot v'' = \nabla_{(A,B)} f(v,  v', \lambda) \cdot ( \dot v,  \dot v'),\]
where $v = u(\lambda)(\placeholder, 1)$.  Let $(\dot V, \dot W)$ be the corresponding rescaled quantities, which solve a nonautonomous planar system of the form
\[ \begin{pmatrix} \dot V_X \\ \dot W_X \end{pmatrix} = \mathcal{M}(X) \begin{pmatrix} \dot V \\ \dot W \end{pmatrix}.\]
Using the expansion of $f$ in \eqref{toy reduced ODE} and the relations \eqref{toy coefficient relations}, we find that 
 \begin{align*}
 \lim_{X \to \pm\infty} \mathcal{M}(X) & =  
 \begin{pmatrix} 0 & 1 \\ 
 |g_{21}| +O(\lambda) & O(\lambda) \end{pmatrix}.
 \end{align*}
Thus $\mathcal{M}$ is strictly hyperbolic upstream and downstream with one negative and one positive eigenvalue.  By a familiar dynamical systems argument, this implies that there cannot be two linearly independent solutions of the reduced linearized problem that are uniformly bounded.  We may then conclude that the kernel of $\F_u(u(\lambda), \lambda) \colon \Xspace_\bdd \to \Yspace_\bdd$ is indeed generated by $\partial_x u(\lambda)$.
\end{proof}

 \subsection{Large fronts for the semilinear problem} \label{toy global section}
 Now that we have local curves of small-amplitude fronts, we seek to continue them globally using Theorem~\ref{general global bifurcation theorem}.  For that, it will be necessary to understand the linearized problem at an arbitrary front.
 
 Letting $(u, \lambda) \in \Xspace_\infty \times \mathbb{R}$ be given, a simple computation shows that 
\begin{align*} \mathscr{F}_u(u,\lambda) \dot u &= \begin{pmatrix} \Delta \dot u \\ \left( \dot u_y - \dot u + g_z(u, \lambda) \dot u \right)|_{\Gamma_1} \end{pmatrix} \\
\mathscr{F}_\lambda(u,\lambda) \dot \lambda &= \begin{pmatrix} 0 \\ g_\lambda(u|_{\Gamma_1}, \lambda) \dot \lambda \end{pmatrix}.\end{align*}
 Now suppose $(u, \lambda)$ is a front-type solution to \eqref{toy problem}.  Our characterization of the conjugate flows \eqref{toy flow force G relation} ensures that $\lim_{x \to \infty} (u - ry) = 0$ for some root $r$ of $g(\placeholder,\lambda)$.  The transversal linearized operators upstream and downstream are therefore
 \[ \limL_-^\prime(u,\lambda) \dot u = \begin{pmatrix}  \dot u_{yy} \\ \left(\dot u_y - \dot u + g_z(0,\lambda) \dot u\right)|_{\Gamma_1} \end{pmatrix},  \quad \limL_+^\prime(u, \lambda) \dot u = \begin{pmatrix}  \dot u_{yy} \\ \left(\dot u_y - \dot u + g_z(r, \lambda) \dot u \right)|_{\Gamma_1} \end{pmatrix}. \]

 It is readily seen that $\xi^2 \geq 0$ is in the spectrum of $\limL_-^\prime(u,\lambda)$ if and only if 
 \[ 1 = \frac{\tanh{\xi}}{\xi} (1- g_z(0, \lambda)  ),\]
 which is possible if and only if $g_{z}(0,\lambda) \leq 0$.  Likewise, $\xi^2 \geq 0$ is an eigenvalue of $\limL_+^\prime(u,\lambda)$ provided that 
 \[ 1 = \frac{\tanh{\xi}}{\xi} (1- g_z(r,\lambda)  ). \]
 Thus the spectral degeneracy alternative \ref{gen ripples} is equivalent to the loss of strict convexity of $G$ at the corresponding critical point.   
In summary, we have proved the following.
\begin{lemma}[Convexity condition] \label{toy ripple lemma}
Let $(u, \lambda) \in \Xspace_\infty \times  \mathbb{R}$ be a monotone front solution of \eqref{toy problem} with $\lim_{x \to \infty} u_y = r$.  Then $\prineigenvalue^-(u,\lambda) < 0$  if and only if $G_{zz}(0,\lambda) > 0$, and $\prineigenvalue^+(u,\lambda) < 0$ if and only if $G_{zz}(r, \lambda) > 0$.
\end{lemma}

The existence of the global bifurcation curves now follows easily. 

\begin{proof}[Proof of Theorem~\ref{toy global bifurcation theorem}]
Let 
  \[ \cm_{\loc}^\pm := \left\{ \left(u(\lambda), \lambda \right) \in \cm_\loc : 0 < |\lambda|  < \varepsilon, ~\pm \lambda > 0 \right\}.\] 
Then by Theorem~\ref{toy small amplitude theorem}\ref{toy local monotonicity part},  solutions on $\cm_\loc^+$ are strictly monotone increasing fronts, while those on $\cm_\loc^-$ are strictly monotone decreasing fronts.  

  Moreover, the structural assumptions on $G$ in \eqref{G assumptions} and Lemma~\ref{toy ripple lemma} imply that the spectral condition \eqref{spectral assumption} holds at each front $(u, \lambda)$ with $\lambda \in\mathcal{O} \setminus \{0\}$.  In particular, this includes all solutions on the local curves.  On the other hand, $G_{zz}(0,0)=0$ means that the bifurcation point $(u_0,\lambda_0) = (0,0)$ satisfies \eqref{local singular assumption}. We have already confirmed in Theorem~\ref{toy small amplitude theorem}\ref{toy kernel part} that the kernel assumption~\eqref{kernel assumption} is satisfied.  Theorem~\ref{general global bifurcation theorem} may therefore be applied to $\cm_\loc^\pm$ yielding the global curves $\cm^\pm$ satisfying \ref{toy global bifurcation monotone part} and \ref{toy global bifurcation alternatives part}. Finally, \ref{toy global bifurcation uniqueness part} follows from Theorem~\ref{toy small amplitude theorem}\ref{toy uniqueness part}.
\end{proof}

\subsection{Realization of the alternatives} \label{toy alternatives section}

In this subsection, we show that the alternatives for the global curve given in Theorem~\ref{toy global bifurcation theorem} are essentially sharp.  Indeed, one of the most appealing features of the PDE \eqref{toy problem} is that it is simple to connect each of the scenarios above with qualitative properties of $G$.

As a first step, we establish some basic uniform regularity results that more precisely characterize which quantities are unbounded in the event that the blowup alternative~\ref{gen blowup alternative} occurs.  

\begin{lemma}[Uniform regularity] \label{toy uniform regularity lemma}
If  $(u, \lambda) \in \Xspace_\infty \times \mathbb{R}$ is a monotone front solution to \eqref{toy problem} with  $|\lambda| < M$, then
  \begin{equation}
    \| u \|_{C^{2+\alpha}(\Omega)} \leq C_1 \| u \|_{L^\infty(\Omega)} \leq C_2, \label{uniform bound} 
  \end{equation}
where  $C_1,C_2 > 0$ depend only on $M$ and 
\[ R :=  \sup \left\{  |r| : (r,\lambda^\prime) \in g^{-1}(0) \cap G^{-1}(0),\,  |\lambda^\prime| < M \right\}. \] 
\end{lemma}
\begin{proof}
Throughout the course of the proof, $C$ denotes a generic positive constant depending only on $M$ and $R$.  Because $u \in \Xspace_\infty$, it has well-defined upstream and downstream states that are constrained by $\lambda$.  Indeed, monotonicity implies
\[ \| u \|_{L^\infty} \leq R \]
which proves the second inequality in \eqref{uniform bound}.  

It remains to show that one can control the full $C^{2+\alpha}$ norm of $u$ by in terms of $\| u \|_{L^\infty}$.  Towards that end, note that $u$ solves
  \begin{equation}
    \left\{ \begin{aligned}
      \Delta u & = 0 & \qquad & \textrm{in } \Omega \\
      u_y  & = h & \qquad & \textrm{on } \Gamma_1 \\
      u &=  0 & \qquad &\textrm{on } \Gamma_0
    \end{aligned} \right.
    \label{uniform bound neumann problem} 
  \end{equation}
where 
\[ h := u - g(u, \lambda)  \in \Xspace_\infty, \qquad \| h \|_{L^\infty} \leq C \| u \|_{L^\infty}.\]
It follows from standard elliptic estimates for the Dirichlet problem that, for any $\Omega_0 \subset \subset \Omega \cup \Gamma_0$, 
  \begin{equation}
    \| u \|_{C^{2+\alpha}(\Omega_0)} < C.\label{near B estimate} 
  \end{equation}
In particular, $C$ above is independent of axial translation, so we uniformly control $u$ in $C^{2+\alpha}$ away from $\Gamma_1$ by $C \| u \|_{L^\infty(\Omega)}$.

Next, let $(x_0, 1) \in \Gamma_1$ be given and consider the half-balls
 \[ \Omega_{1/4} := \Omega \cap B_{1/4}((x_0,1)) \quad \subset \quad  \Omega_{1/2} := \Omega \cap B_{1/2}((x_0,1)).\]
If we can prove that the bound \eqref{near B estimate} holds with $\Omega_{1/4}$ in place of $\Omega_0$, then the conclusion of the lemma is immediate.  To accomplish that, we will reflect the problem \eqref{uniform bound neumann problem} and use interior estimates.  Denote by $\tilde \Omega_{1/4}$ and $\tilde \Omega_{1/2}$ the balls $B_{1/4}$ and $B_{1/2}$ centered at $(x_0,1)$, respectively, and define $\tilde u$ to be the even extension of $u$ over $\Gamma_1$.  Thus $\tilde u \in H^1(\tilde \Omega_{1/2})$ is a weak solution to 
\[ \Delta \tilde u = 2\partial_y \tilde h \qquad \textrm{in } \tilde \Omega_{1/2} \]
with $\tilde h := h(x) \chi_{\{ y > 1 \}} \in L^\infty(\tilde \Omega_{1/2})$.  Using the De Giorgi--Nash-type result \cite[Theorem 8.24]{gilbarg2001elliptic}, we conclude that for any $p > 2$, there exists $\tilde \alpha \in (0,1)$ such that 
\[ \| u \|_{C^{\tilde \alpha}(\Omega_{1/4})} \leq \| \tilde u \|_{C^{\tilde \alpha}(\tilde \Omega_{1/4})} \lesssim \| \tilde u \|_{L^2(\tilde \Omega_{1/2})} + \| \tilde h \|_{L^p(\tilde \Omega_{1/2})} \leq C \| u \|_{L^\infty( \Omega_{1/2})}.\]
Therefore, $u$ is uniformly bounded in $C^{\tilde{\alpha}}$ on the segment of $\Gamma_1$ lying in $\tilde \Omega_{1/4}$.  Because $h$ is real analytic, the desired bound \eqref{near B estimate} for the half-ball $\Omega_{1/4}$ is now a consequence of Schauder estimates for \eqref{uniform bound neumann problem} and a standard bootstrapping argument.  This completes the proof. 
\end{proof}
\begin{remark} \label{semilinear bounds remark}
A similar argument yields versions of this result whenever the PDE \eqref{fully nonlinear elliptic pde} is semilinear.  The key point above is to control $\| u \|_{C^{\alpha}}$ by $\|u\|_{L^\infty}$ and $|\lambda|$.  For divergence form semilinear operators (on quite rough domains), this follows from \cite[Proposition 3.6]{nittka2011regularity}.  
\end{remark}

Now we are able to give scenarios in which each of the three alternatives in Theorem~\ref{toy global bifurcation theorem} are expected; these are illustrated in Figure~\ref{toy alternatives figure}.  In particular, we can guarantee that the unboundedness alternative \ref{gen blowup alternative} and spectral degeneracy alternative \ref{gen ripples} happen for some explicit classes of $G$.  We also provide a necessary condition for heteroclinic degeneracy \ref{gen hetero degeneracy}.  

\begin{figure}
\includegraphics[scale=1]{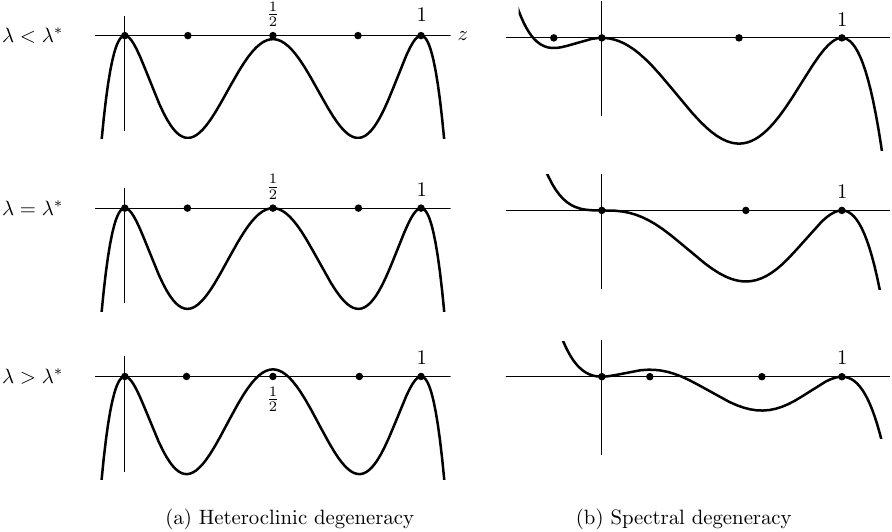}
  \caption{The curves above depict the graph of  $-G(\placeholder, \lambda)$ for two different families of potential and at a subcritical $(\lambda < \lambda^*)$, critical $(\lambda = \lambda^*)$, and supercritical $(\lambda > \lambda^*)$ parameter value.    In (a), as $\lambda$ passes through $\lambda^*$, a second critical point at the energy level $0$ forms, which permits heteroclinic degeneracy. The corresponding orbits of the ODE $\ddot u = G_z(u,\lambda)$ are shown in Figure~\ref{hetero degen phase portrait figure}. In (b), at $\lambda = \lambda^*$, $G(\placeholder, \lambda)$ loses strict convexity at $0$ implying the onset of spectral degeneracy. The corresponding orbits of the ODE $\ddot u = G_z(u,\lambda)$ are shown in Figure~\ref{ripples phase portrait figure}. }
\label{toy alternatives figure}
\end{figure}

 \begin{proposition}[Realization] \label{realization proposition} Let $\cm^\pm$ be the global bifurcation curve furnished by Theorem~\ref{toy global bifurcation theorem}. 
 \begin{enumerate}[label=\rm(\alph*)]
 \item \label{toy unbounded realization} If $G$ satisfies \eqref{G assumptions} for all $\lambda \in \mathbb{R}$, then the unboundedness alternative \ref{gen blowup alternative} must occur.  In particular, 
   \begin{equation}
     |\lambda(s)|  \to \infty \qquad \textrm{as } s \to \pm\infty. \label{toy better blowup} 
   \end{equation}
 
 \item \label{toy degeneracy realization} Heteroclinic degeneracy \ref{gen hetero degeneracy} can occur only if there exists $r \ne 0,\lambda_*$ with $g(r, \lambda_*) = 0$ and $G(r,\lambda_*) = 0$. 
 
 \item \label{toy ripple realization} Suppose that there exist bounded open sets $\mathcal{N}_2 \supset\supset \mathcal{N}_1 \supset \mathcal{O}$ such that
 \[ G(\placeholder, \lambda)^{-1}(0) \cap g(\placeholder, \lambda)^{-1}(0) = \{ 0, \lambda\} \qquad \textrm{for all }  \lambda \in \mathcal{N}_2\]
  and $G$ satisfies \eqref{G assumptions} on $\mathcal N_1$, but
 \[ \min\left\{ G_{zz}(0,\lambda),\, G_{zz}(\lambda,\lambda) \right\}  \to 0 \qquad \textrm{as }  \lambda \to \partial \mathcal{N}_1.\]
 Then necessarily spectral degeneracy \ref{gen ripples} occurs.     
 \end{enumerate}
 \end{proposition}
\begin{proof}

  First we consider part~\ref{toy degeneracy realization}.  Suppose that the heteroclinic degeneracy does indeed occur as in \ref{gen hetero degeneracy}, in which case there are three distinct limiting states
  \begin{equation*}
    \lim_{x \to \mp\infty} u_*(x, \placeholder) =: r_* y,
    \quad 
    \lim_{n \to \infty} \lim_{x \to +\infty} u(s_n)(x, \placeholder) = \lambda_* y,
    \quad 
    \lim_{n \to \infty} \lim_{x \to -\infty} u(s_n)(x, \placeholder) = 0,
  \end{equation*}    
  all of which are $x$-independent solutions of \eqref{toy problem}. By Lemma~\ref{triple conjugacy lemma}, they are all mutually conjugate, and so in addition to $g(r_*,\lambda_*)=g(0,\lambda_*)=0$ we have $G(r_*,\lambda_*)=G(0,\lambda_*)=0$. The statement follows.
 
  Next, suppose that $G$ satisfies \eqref{G assumptions} globally. Lemma~\ref{toy ripple lemma} implies that spectral degeneracy \ref{gen ripples} can only happen if $\cm^\pm$ approaches $(0,0)$ as $s \to \pm\infty$. But this is ruled out by Theorem~\ref{toy global bifurcation theorem}\ref{toy global bifurcation uniqueness part}. From part~\ref{toy degeneracy realization}, we see that the heteroclinic degeneracy alternative is likewise impossible.  Thus blowup must occur.  In view of Lemma~\ref{toy uniform regularity lemma}, unboundedness of $\cm^\pm$ in $\Xspace_\bdd \times \mathbb{R}$ can happen only if $|\lambda(s)| \to \infty$.  
 
  Finally, assume that $G$ is given as in part~\ref{toy ripple realization}.   As the set $\mathcal{N}_1$ is bounded, we have by the previous paragraph that $\| u(s) \|_{L^\infty}$ (and hence $\| u(s) \|_{\Xspace}$)  is uniformly bounded so long as $\lambda(s) \in \mathcal{N}_1$.  Likewise, the unboundedness alternative cannot occur if $\lambda(s) \subset \subset \mathcal{N}_2$ for all $|s| > 0$, so there must exist some $|s_0| < \infty$ such that $\lambda(s) \to \partial \mathcal{N}_1$ as $s \to s_0\mp$.  In light of Lemma~\ref{toy ripple lemma}, this leads to spectral degeneracy \ref{gen ripples}.
\end{proof}

\section{Large-amplitude hydrodynamic bores} \label{bores section}

\begin{figure}
  \centering
  \includegraphics[scale=1]{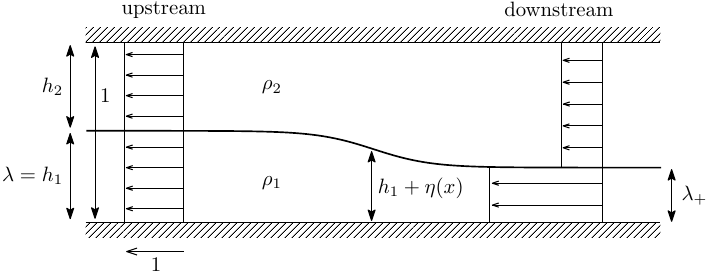}
  \caption{A monotone bore in a two-layer system. 
  \label{basic bore figure} }
  \end{figure}
  In this section, we apply our abstract results to the classical problem of bores traveling along the interface between two fluids of different densities, which can be formulated as the quasilinear transmission problem \eqref{eqn:u} below. While this system is much more complicated than \eqref{toy problem}, we will show that it enjoys many of the same qualitative features. In particular, thanks to its variational structure, it is possible to rule out both heteroclinic degeneracy \ref{gen hetero degeneracy} and spectral degeneracy \ref{gen ripples} via careful study of the conjugate flow problem. On the other hand, since the system is quasilinear rather than semilinear, the analysis of the blowup alternative \ref{gen blowup alternative} becomes quite delicate.

Consider a stably stratified configuration in which a lighter fluid with constant density $\rho_2 > 0$ lies atop a heavier fluid with constant density $\rho_1 > \rho_2$. Working in a reference frame moving with the wave, we assume that interface between the two layers as well as the fluid velocity fields are independent of time. Both for simplicity and because it is the setting with the most applied interest, suppose that the fluid velocity tends to some constant value $(-c,0)$ in the upstream limit $x \to -\infty$, where here $c>0$ is interpreted as wave speed. Letting $h_1,h_2 > 0$ be the upstream thicknesses of the two layers and $g > 0$ be the constant acceleration due to gravity, the dimensionless Froude number
\begin{align*}
  F := \frac c{\sqrt{g(h_1+h_2)}} 
\end{align*}
measures the relative importance of inertial and gravitational effects. Switching to dimensionless units with $h_1+h_2$ as the length scale and $c$ as the velocity scale, we are left with the three dimensionless parameters $F$, $\rho_2/\rho_1$, and $h_1$.

The origin of coordinates is chosen so that the interface between the two layers is $y=\eta(x)$ where $\eta \to 0$ as $x \to -\infty$. The bottom boundary of the channel is therefore $y=-h_1$, and the upper boundary is $y=h_2=1-h_1$.  We denote the upper fluid domain in these variables by $\fluidD_2$, the lower fluid by $\fluidD_1$, and set $\fluidD := \fluidD_1 \cup \fluidD_2$.  Finally, we write $\fluidS$ for the internal interface. 

\begin{subequations}\label{eqn:stream}
  Requiring the flow to be irrotational and incompressible in each layer, the velocity field can be expressed as $(\partial_y \psi, -\partial_x \psi)$ for some \emph{stream function} $\psi$ satisfying
  \begin{equation} \label{eqn:psi harmonic}
    \Delta \psi = 0 \quad \text{in } \fluidD.
  \end{equation}
  The stream function is constant along the interface as it is a material surface, and we normalize this constant to be zero:
  \begin{align}
    \label{eqn:stream:kinint}
    \psi = 0 \qquad \text{on } \fluidS.
  \end{align}
  In particular, $\psi$ is continuous across the interface. Our assumptions upstream can be written as
  \begin{align}
    \label{eqn:stream:asym}
    \nabla\psi \to (0,-1),
    \quad 
    \eta \to 0
    \qquad \text{ as } x \to -\infty.
  \end{align}
  The rigid boundaries $y=-h_1$ and $y=h_2$ are also material surfaces, and hence level curves of the stream function. Calculating these values using \eqref{eqn:stream:kinint} and \eqref{eqn:stream:asym}, we find
  \begin{align}
    \begin{alignedat}{2}
      \psi &= h_1 &\quad& \text{ on } y=-h_1,  \\
      \psi &= -h_2 &\quad& \text{ on } y=h_2. 
    \end{alignedat}
    \label{psi value on top}
  \end{align}
  Finally, the dynamic boundary condition on $y=\eta(x)$ asserts the continuity of the pressure, 
  \begin{equation}
    \label{eqn:stream:dynamic}
    \frac 12 \jump{\rho\abs{\nabla\psi}^2}
    + \frac {\jump\rho}{F^2} y = \frac {\jump\rho}2  \qquad  \ona \fluidS.
  \end{equation}
  As in Section~\ref{sec:trans}, $\jump{\placeholder}=(\placeholder)_2-(\placeholder)_1$ denotes the jump of a quantity across the interface $y=\eta(x)$. For general traveling waves, \eqref{eqn:stream:dynamic} would contain an undetermined Bernoulli constant, but in our case we have calculated this constant explicitly using \eqref{eqn:stream:asym}.
\end{subequations}

For more background on \eqref{eqn:stream} and related problems, see Section~\ref{intro bores section} in the introduction or the lengthier discussion in \cite[Section 1.2]{chen2018existence}. Our first theorem is the following global bifurcation result for monotone bores.

\begin{theorem}[Large monotone bores] \label{global bore theorem} Fix an $\alpha \in (0,1)$ and densities $0 < \rho_2 < \rho_1$.  There exist global $C^0$ curves
\[ \cm^\pm = \left\{ (\psi(s), \eta(s), \lambda(s)) : \pm s \in (0,\infty) \right\}\]
 of solutions to the internal wave problem \eqref{eqn:stream} with $h_1 = \lambda(s)$, $h_2 = 1-h_1$, and $F$ given by \eqref{F equation}.   They enjoy the  H\"older regularity  
 \[ \psi(s) \in C_\bdd^{2+\alpha}(\overline{\fluidD_1(s)}) \cap C_\bdd^{2+\alpha}(\overline{\fluidD_2(s)}) \cap C_\bdd^0(\overline{\fluidD(s)}), \qquad \eta(s) \in C_\bdd^{2+\alpha}(\mathbb{R}), \]
 where $\fluidD(s)$ denotes the fluid domain corresponding to $\eta(s)$ and $\lambda(s)$.  
\begin{enumerate}[label=\rm(\alph*)]
\item \label{global bore monotone part} \textup{(Strict monotonicity)} Each solution on $\cm^\pm$ is a strictly monotone bore:
\begin{equation}
  \begin{aligned}
    \pm \partial_x \eta(s) &< 0 & \qquad &  \textrm{on } \mathbb{R}, \\
    \pm \partial_x \psi(s) & > 0 & \qquad & \textrm{in } \fluidD(s) \cup \fluidS(s), \\
    \partial_y \psi(s) & < 0 & \qquad & \textrm{in } \overline{\fluidD(s)}.
  \end{aligned} \label{monotonicity Euler variables} 
\end{equation}
\item \label{global bore limit part} \textup{(Stagnation limit)} Following $\cm^\pm$, we encounter waves that are arbitrarily close to having a horizontal stagnation point on the internal interface:
  \begin{equation}
    \lim_{s \to \pm\infty} \sup_{\fluidS(s)} \partial_y\psi_i(s) = 0, \qquad \textrm{for $i = 1$ or $2$}.  \label{stagnation limit} 
  \end{equation}
\item \label{global bore laminar part} \textup{(Laminar origin)} Both $\cm^-$ and $\cm^+$ emanate from the same laminar solution in that 
  \[ \eta(s) \to 0, \quad \nabla \psi(s) \to (0,-1), \quad \lambda(s) \to {\lambdastar\pm} \qquad \textrm{as } s \to 0\pm,\]
where $\lambdastar$ is the constant \eqref{unique conjugate}.  
\end{enumerate}

\end{theorem}

The next result characterizes the limiting form of the profile along $\cm^\pm$.  

\begin{theorem}[Limiting interface] \label{limit eta theorem} 
Consider the behavior of the profile as we traverse $\cm^\pm$.
\begin{enumerate}[label=\rm(\alph*)]
\item \textup{(Overturning or singularity)} \label{elevation part} In the limit along $\cm^-$, either the interface \emph{overturns} in that
  \begin{equation}
    \limsup_{s \to -\infty} \| \partial_x \eta(s) \|_{L^\infty(\mathbb{R})} = \infty, \label{overturning} 
  \end{equation}
or it becomes \emph{singular} in that we can extract a translated subsequence 
\[ \eta(s) \longrightarrow \eta^* \in \Lip(\mathbb{R}) \quad \textrm{in $C_\loc^\varepsilon$ for all } \varepsilon \in (0,1) \]
such that $\{ y < \eta^*(x) \}$ simultaneously fails to satisfy both an interior sphere and exterior sphere condition at a single point on its boundary. 
\item \textup{(Overturning or contact)} \label{depression part} Following $\cm^+$, either the interface overturns or it comes into contact with the upper wall:
  \begin{equation}
    \limsup_{s \to \infty} \lambda(s) = 1 \qquad \textrm{or} \qquad  \limsup_{s \to \infty} \| \partial_x \eta(s) \|_{L^\infty(\mathbb{R})} = \infty.
    \label{wall contact or overturning} 
  \end{equation}
\end{enumerate}
\end{theorem}
\begin{remark} \label{elevation overturning remark} The singularity alternative along $\cm^-$ is considerably more exotic than a corner or cusp, as both of these would satisfy exterior sphere conditions.  In fact, the singular point will be one where the flow in \emph{both} layers approaches stagnation.  Observe that the dynamic condition \eqref{eqn:stream:dynamic} can be rearranged as 
\[ \frac{1}{2} \jump{\rho |\nabla \psi|^2} = -\frac{\jump{\rho}}{F^2} \left( y - \frac{F^2}{2} \right).  \]
For elevation waves, $y > 0$ along $\fluidS$, so the right-hand side above vanishes at the (unique) point on the interface with height $y = F^2/2$.  This highly degenerate scenario is outside the scope of current free boundary regularity theory.  Indeed,  the one-phase case was only recently considered by Varvaruca and Weiss \cite{varvaruca2011geometric}, and several open questions regarding it remain.  We strongly believe, however, that through a more detailed analysis the singularity alternative can be eliminated.   This will be the subject of future work.

On the other hand, it is interesting to note that clearly the height $y = F^2/2$ is special.  For example, in the numerical paper \cite{maklakov2018almost}, Malakov and Sharipov compute a family of overhanging solitary internal waves.  Roughly speaking, the free surface profile for the near-limiting waves resembles a mushroom or capital $\Omega$.  They observe that the point where the free boundary folds back occurs exactly at $y = F^2/2$, and the flow there is close to stagnation in both layers.
\end{remark}
\begin{remark} \label{depression overturning remark}
The numerical results in \cite{dias2003internal} suggest that along $\cm^+$, the interface always hits the wall rather than overturning.  Formally, one expects that in this scenario the entire upper layer will become stagnant, resulting in what is known as a gravity current.  Formal computations by von K\'arm\'an \cite{vonkarman1940engineer} indicate that, at the point of contact, the interface makes a precise $60^\circ$ angle with the wall (see also the paper of Benjamin \cite{benjamin1968gravity}).  In a very real sense, this is the analogue of the Stokes conjecture in the context of internal bores.  We hope to address it as well in a forthcoming paper.
\end{remark}

\subsection{Reformulation} \label{formulation section}
To fix the domain, we now switch to the new coordinates
\begin{align} \label{def qp coords} 
  q = x,
  \qquad 
  p = -\frac\psi h,
\end{align}
where here $h$ denotes the piecewise constant function which is $h_1$ in the lower layer and $h_2$ in the upper layer.  This is sometimes called the Dubreil-Jacotin or partial hodograph transform.  It is valid whenever there are no horizontal stagnation points:
\begin{equation}
  \label{psi no stagnation} 
  \sup_{\fluidD}  \psi_y < 0. 
\end{equation}
The image of $\fluidD$ under this change of variables is the slitted cylinder 
\[ \Omega = \Omega_1 \cup \Omega_2, \quad   \Omega_1 := \mathbb{R} \times (-1,0), \quad \Omega_2 := \mathbb{R} \times (0,1).\]
In keeping with the general theory, denote by $\Gamma_0 := \{ p = -1\} \cup \{ p = 1\}$ the corresponding lower and upper boundaries, and let $\Gamma_1 := \{ p = 0\}$ be the image of the internal interface $\fluidS$.  
  As a new dependent variable, we consider
\begin{align*}
  y = y(q,p),
\end{align*}
which measures the vertical deflection of the streamlines (level sets of $\psi$) relative to their upstream heights.  
This transforms \eqref{eqn:stream} into
\begin{subequations}\label{eqn:h}
  \begin{alignat}{2}
    \bigg( {- \frac{1+y_q^2}{2y_p^2}} \bigg)_p 
    + 
    \bigg( \frac{y_q}{y_p} \bigg)_q 
    &= 0 &\qquad& \text{ for } 0 < \abs p < 1. \\
    \frac 12 \bigg\llbracket 
      \rho h^2 
    \frac{1+y_q^2}{y_p^2}
    \bigg\rrbracket
    + \frac {\jump\rho}{F^2} y &= \frac {\jump\rho}2  && \ona p=0,\\
    y &= -h_1 &\qquad& \text{ on } p=-1,\\
    y &= h_2 &\qquad& \text{ on } p=1,
  \end{alignat}
  together with the asymptotic condition
  \begin{align}
    \label{eqn:h:asym}
    y \to hp
    \qquad \text{ as } q \to -\infty.
  \end{align}
\end{subequations}

The upstream state in \eqref{eqn:stream:asym} corresponds to $y=hp$, and so to conform to the conventions in the previous sections we introduce the difference
\begin{align*}
  u(q,p) := y(q,p) - hp.
\end{align*}
While $u$ often denotes to the horizontal fluid velocity in the fluid literature, we emphasize that it has no such connotation here. In terms of $u$, \eqref{eqn:h} becomes
\begin{subequations}\label{eqn:u}
  \begin{alignat}{2}
    \label{eqn:u:lap}
    \bigg( {- \frac{1+u_q^2}{2(h+u_p)^2}} \bigg)_p 
    + 
    \bigg( \frac{u_q}{h+u_p} \bigg)_q 
    &= 0 &\qquad& \ina \Omega, \\
    \label{eqn:u:dyn}
    -\frac 12 \bigg\llbracket 
      \rho h^2 
    \frac{1+u_q^2}{(h+u_p)^2}
    \bigg\rrbracket
    - \frac {\jump\rho}{F^2} u + \frac {\jump\rho}2 &= 0 && \ona \Gamma_1,\\
    \label{eqn:u:kin}
    u &= 0 &\qquad& \ona \Gamma_0,
  \end{alignat}
  with the asymptotic conditions
  \begin{align}
    \label{eqn:u:asym}
    u \to 0 \qquad \text{ as } x \to -\infty.
  \end{align}
\end{subequations}

In what follows we will hold the densities $0 < \rho_2 < \rho_1$ as fixed constants, and allow the layer thicknesses $h_1,h_2$ to vary subject to the constraint $h_1+h_2=1$. To conform with the notation of the previous sections, we therefore write
\begin{align*}
  h_1 = \lambda,
  \qquad h_2 = 1-\lambda
\end{align*}
where $\lambda \in (0,1)$ is our parameter. As we will see below in Section~\ref{conjugate flow section}, in order to have a nontrivial solution with distinct limits as $u\to\pm\infty$, the Froude number $F$ must be given by 
\begin{align}
  \label{F equation}
  F^2 = \frac{\sqrt{\rho_1}-\sqrt{\rho_2}}{\sqrt{\rho_1}+\sqrt{\rho_2}},
\end{align}
independently of $\lambda$.

An elementary calculation shows that 
\begin{equation}
  \label{grad u grad psi} 
  u_q = -\frac{\psi_x}{\psi_y}, \qquad h + u_p = -\frac{h}{\psi_y}.
\end{equation}
Hence, $u_q$ measures the slope of the relative velocity field, while $h+u_p$ is inversely proportional to the relative horizontal velocity.   The absence of horizontal stagnation \eqref{psi no stagnation} is equivalent to the requirement that 
\begin{equation}
  \label{u no stangation} 
  \sup_{\Omega} \left( h+u_p \right) < \infty.
\end{equation}

Following the conventions for transmission problems in Section~\ref{sec:trans}, fix an $\alpha \in (0,1)$ and let 
\[ \begin{aligned} 
\Xspace & := \left\{ u \in C^{2+\alpha}(\overline{\Omega_1}) \cap C^{2+\alpha}(\overline{\Omega_2}) \cap C^0(\overline{\Omega}) : u|_{\Gamma_0} = 0 \right\}, \\
\Yspace & = \Yspace_1 \times \Yspace_2 :=  \left( C^{\alpha}(\overline{\Omega_1}) \cap C^\alpha(\overline{\Omega_2})\right) \times C^{1+\alpha}(\Gamma_1),
\end{aligned} \]
with $\Xspace_\bdd$, $\Xspace_\infty$, $\Yspace_\bdd$, and $\Yspace_\infty$ defined accordingly.  For each $\delta > 0$, take
\begin{equation}
  \genU^\delta := \Big\{ (u,\lambda) \in \Xspace_\bdd \times (0,1) : \inf_{\Omega} ( u_p + h) > \delta, ~ \frac{1}{\lambda} + \frac{1}{1-\lambda} < \frac{1}{\delta} \Big\}, \label{def U delta} 
\end{equation}
so that the nonlinear operator $\F$ corresponding to \eqref{eqn:u} viewed as a map $\genU^\delta \subset \Xspace_\bdd \times \mathbb{R} \to \Yspace_\bdd$ satisfies the real analyticity \eqref{gen regularity elliptic coef} and uniform ellipticity and obliqueness requirements \eqref{ellipticity assumption}.    This collection is nested, so we write $\genU := \cup_{\delta > 0} \genU^\delta$, and denote by $\genU_\infty$ the set of $(u, \lambda) \in \genU$ for which $u \in \Xspace_\infty$. Note that the sets $\genU^\delta$ are open, as are $\genU$ and $\genU_\infty$.
\subsection{Variational structure and weak formulations}
The problem \eqref{eqn:u} has a variational structure \eqref{variational principle} with
\begin{align*}
  \mathcal L(p,z,\xi,\lambda)
  = \rho \left[ h^2\frac{1+\xi_1^2}{2(h+\xi_2)^2} + \frac 12 - \frac 1{F^2} (h p + z) \right] (h+\xi_2).
\end{align*}
In particular, \eqref{eqn:u:lap}--\eqref{eqn:u:kin} can be written as
\begin{equation}\label{variational bore}
  \left\{ \begin{aligned}
    \nabla \cdot \mathcal L_\xi(p,u,\nabla u,\lambda)
    - \mathcal L_z(p,u,\nabla u,\lambda)
    &= 0 \qquad \textrm{in } \Omega, \\
    \jump{\mathcal L_{\xi_2}(p,u,\nabla u,\lambda)}  & = 0 \qquad  \textrm{on } \Gamma_1, \\
    u & = 0 \qquad \textrm{on } \Gamma_0.
  \end{aligned} \right.
\end{equation}
Using an analogue of Lemma~\ref{conserved quantity lemma}, we see that \eqref{eqn:u} has the conserved quantity
\begin{align*}
  \flowforce(u,\lambda;q)
   &:=  
   \int_{-1}^1 
   \rho  \bigg( h^2\frac{1-u_q^2}{2(h+u_p)^2}
    + \frac 12 - \frac 1{F^2} (h p + u) \bigg) (h+u_p)\, dp,
\end{align*}
as can be verified by a direct calculation.

Notice that the first two lines of \eqref{variational bore} can be interpreted as $\nabla \cdot \mathcal L_\xi - \mathcal L_z = 0$ holding in a weak sense. A related weak formulation is
\begin{equation}\label{bore weak formulation} \left\{
  \begin{alignedat}{2}
    \nabla \cdot \Big( \rho \nabla f (\nabla u,h) - \frac \rho{F^2} u \mathbf{e}_2 \Big) 
    + \frac\rho{F^2} u_p &= 0 &\qquad& \ona \Omega\cup\Gamma_1 ,\\
    u &= 0 &\qquad& \ona \Gamma_0,
  \end{alignedat} \right.
\end{equation}
where the function $f\colon \mathbb{R}^2 \times \mathbb{R} \to \mathbb{R}$ is given by
\begin{align*}
  f(\xi,a) = \frac{a^2 \xi_1^2 + \xi_2^2}{2(a+\xi_2)}.
\end{align*}
Here and in what follows, the gradient $\nabla_\xi f$ and Hessian $D_\xi^2 f$ will simply be denoted by $\nabla f$ and $D^2 f$, respectively. 

\subsection{Conjugate flows}\label{conjugate flow section}
As in Section~\ref{toy global section}, our ability to refine the alternatives in Theorem~\ref{general global bifurcation theorem} depends in a large part on the fact that the conjugate flow problem \ref{conjugate solution part}--\ref{conjugate flowforce part} is explicitly solvable. This fact is well known in the literature on internal waves; see for instance \cite[appendix~A]{Laget1997interfacial}.
\begin{lemma}\label{conjugate lemma}
  Fix $\lambda \in (0,1)$. Then there is a flow $U_+$ conjugate to $U_-\equiv 0$ in the sense of Definition~\ref{conjugate definition} if and only if 
  the Froude number $F$ is given by \eqref{F equation} and $\lambda \ne \lambdastar$ where
  \begin{equation}
    \label{lambda star}
    \lambdastar := \frac {\sqrt{\rho_1}}{\sqrt{\rho_1}+\sqrt{\rho_2}}.
  \end{equation}
  Moreover, in this case $U_+$ is unique and given by
  \begin{equation}
    \label{unique conjugate}
    U_+(p) = (\lambdastar-\lambda)(1-\abs p).
  \end{equation}
\end{lemma}
\begin{proof}
  From \eqref{eqn:u:lap}, \eqref{eqn:u:kin}, and the continuity of $U$, we see that \ref{conjugate solution part} forces \eqref{unique conjugate} for some $\lambdastar \in (0,1) \setminus\{\lambda\}$, which can be interpreted physically as the downstream depth of the lower layer. Plugging into \eqref{eqn:u:dyn}, we conclude that \ref{conjugate solution part} is satisfied if and only if $\lambdastar$ satisfies
  \begin{equation}
    \label{conjugate 1}
    \frac 12 \left(
    \rho_2 \frac{(1-\lambda)^2}{(1-\lambdastar)^2}
    - \rho_1\frac{\lambda^2}{\lambdastar^2}
    \right)
    + \frac{\rho_2-\rho_1}{F^2} (\lambda-\lambdastar)
    = \frac{\rho_2-\rho_1}2.
  \end{equation}
  On the other hand,
  \begin{align*}
    \flowforce(U,\lambda) 
    &= \lambdastar \rho_1\int_{-1}^0 \bigg(
    \frac{\lambda^2}{2\lambdastar^2} + \frac 12 -
    \frac {(\lambdastar-\lambda)(1+p)+\lambda p}{F^2}
    \bigg)\, dp
    \\ &\qquad 
    + (1-\lambdastar)\rho_2 \int_0^1 \bigg(
    \frac{(1-\lambda)^2}{2(1-\lambdastar)^2} + \frac 12 -
    \frac {(\lambdastar-\lambda)(1-p)+(1-\lambda) p}{F^2}
    \bigg)\, dp\\
    &=
    \lambdastar \rho_1 \bigg(
    \frac{\lambda^2}{2\lambdastar^2} + \frac 12 +
    \frac {2\lambda-\lambdastar}{2F^2}
    \bigg)
    + (1-\lambdastar)\rho_2  \bigg(
    \frac{(1-\lambda)^2}{2(1-\lambdastar)^2} + \frac 12 +
    \frac {2\lambda-\lambdastar-1}{2F^2}
    \bigg).
  \end{align*}
  Setting $\lambdastar=\lambda$ to calculate $\flowforce(0,\lambda)$, we see that \ref{conjugate flowforce part} is equivalent to
  \begin{equation}
    \label{conjugate 2}
    \begin{split}
      \lambdastar \rho_1 \bigg(
      \frac{\lambda^2}{2\lambdastar^2} + \frac 12 +
      \frac {2\lambda-\lambdastar}{2F^2}
      \bigg)
      + (1-\lambdastar)\rho_2  \bigg(
      \frac{(1-\lambda)^2}{2(1-\lambdastar)^2} + \frac 12 +
      \frac {2\lambda-\lambdastar-1}{2F^2}
      \bigg)\\
      =
      \lambda \rho_1 \bigg(
      1 + 
      \frac\lambda{2F^2}
      \bigg)
      + (1-\lambda)\rho_2  \bigg(
      1+
      \frac {\lambda-1}{2F^2}
      \bigg).
    \end{split}
  \end{equation}
Eliminating $F^2$ between \eqref{conjugate 1} and \eqref{conjugate 2}, we discover that the only possible value of $\lambdastar \in (0,1) \setminus \{\lambda\}$ is \eqref{lambda star}. Substituting \eqref{lambda star} into \eqref{conjugate 1} then yields \eqref{F equation}.
\end{proof}

\begin{lemma}[Spectral nondegeneracy] \label{no ripples lemma}
  Suppose that the Froude number $F$ is given by \eqref{F equation}, and let $(u,\lambda) \in \genU_\infty$ be a monotone front solution of $\F(u,\lambda) = 0$. If $\lambda\ne\lambdastar$, then the spectral assumption \eqref{spectral assumption} holds, while $\lambda=\lambdastar$ instead implies $\prineigenvalue^-(u,\lambda) = 0$.
\end{lemma}
\begin{proof}
  Letting $U_\pm$ denote the limits of $u$ as $x \to \pm\infty$, the operators $\limL_\pm^\prime(u,\lambda)$ are given by
  \begin{align*}
    \limL_{\pm1}'(u,\lambda) w &= \bigg(\frac{w_p}{(h+\partial_p U_\pm)^3}\bigg)_p,\\
    \limL_{\pm2}'(u,\lambda) w &= 
     \bigg\llbracket 
      \rho h^2 
    \frac{w_p}{(h+\partial_p U_\pm)^3}
    \bigg\rrbracket
    - \frac {\jump\rho}{F^2} w.
  \end{align*}
  The upstream limit $U_- \equiv 0$ since $u \in \Xspace_\infty$, so let us first consider 
  \begin{align*}
    \limL_{-1}'(u,\lambda) w = \frac{w_{pp}}{h^3},
    \qquad 
    \limL_{-2}'(u,\lambda) w = 
     \bigg\llbracket 
    \frac{\rho w_p}{h}
    \bigg\rrbracket
    - \frac {\jump\rho}{F^2} w.
  \end{align*}
  Setting $w = 1-\abs p$, we have $\limL_{+1}'(u,\lambda)w = 0$ and
  \begin{align*}
    \limL_{-2}'(u,\lambda) w 
    &= 
    - \frac{\rho_2}{h_2} - \frac{\rho_1}{h_1}
    - \frac{\rho_2-\rho_1}{F^2}\\
    &= 
    - \frac{\rho_2}{1-\lambda} - \frac{\rho_1}\lambda
    + (\sqrt{\rho_1}+\sqrt{\rho_2})^2.
  \end{align*}
  Optimizing over $\lambda \in (0,1)$ we see that $\limL_{-2}'(u,\lambda) w \le 0$ with equality only if $\lambda=\lambdastar$. For $\lambda\ne\lambdastar$, $w$ is therefore a positive strict supersolution in the sense of Appendix~\ref{principal eigenvalue appendix transmission}, implying that the principal eigenvalue $\prineigenvalue^-(u,\lambda)$ of $\limL_-'(u,\lambda)$ is strictly negative as desired. For $\lambda=\lambdastar$, on the other hand, $w$ is a positive eigenfunction of $\limL_-'(u,\lambda)$ with eigenvalue $0$, and hence $\prineigenvalue^-(u,\lambda)=0$.
 
  If $u \equiv 0$ then the argument for $\limL_+'(u,\lambda)$ is identical, so suppose that $u \not \equiv 0$. Then Lemma~\ref{conjugate lemma} forces $U_+$ to be given by \eqref{unique conjugate}. Letting $\hplus$ be the function which is $\lambdastar$ in the lower layer and $1-\lambdastar$ in the upper layer, we find
  \begin{align*}
    \limL_{+1}'(u,\lambda) w = \frac{w_{pp}}{\hplus^3},
    \qquad 
    \limL_{+2}'(u,\lambda) w = 
     \bigg\llbracket 
    \frac{\rho h^2 w_p}{\hplus^3}
    \bigg\rrbracket
    - \frac {\jump\rho}{F^2} w.
  \end{align*}
  Once more taking $w = 1-\abs p$, we have $\limL_{+1}'(u,\lambda)w = 0$. The boundary term is 
  \begin{align*}
    \limL_{+2}'(u,\lambda) w 
    &= 
    -\frac{(1-\lambda)^2\rho_2}{(1-\lambdastar)^3} - \frac{\lambda^2\rho_1}{\lambdastar^3}
    + (\sqrt{\rho_1}+\sqrt{\rho_2})^2.
  \end{align*}
  Optimizing over $\lambda \in (0,1)$ we again find that $\limL_{+2}'(u,\lambda) w \ge 0$ with equality if and only if $\lambda=\lambdastar$. Arguing as above we conclude that $\prineigenvalue^+(u,\lambda) \le 0$ with equality if and only if $\lambda=\lambdastar$.
\end{proof}

\subsection{Small-amplitude bores} \label{small section}

We next recall the small-amplitude existence theory from \cite[Section 7]{chen2022center}.  That paper treated the more general setting where the vorticity in the upper layer is constant, not necessarily zero.  This substantially complicates the conjugate flow equations, as well as potentially allowing for critical layers (lines of stagnation points) in the bulk.  As one consequence, the authors were forced to use a much less elegant change of variables to fix the fluid domain.  For the layer-wise irrotational case \eqref{eqn:psi harmonic}, however, translating between these two coordinate systems is tedious but rather straightforward.  The result is then as follows.

\begin{theorem}[Small-amplitude bores] \label{small bores theorem} 
Fix densities $0 < \rho_2 < \rho_1$, and let the Froude number $F$ be given by \eqref{F equation}.   There exists a $C^0$ curve $\cm_\loc$ of solutions to the internal wave problem \eqref{eqn:u} that admits the parameterization
\[ \cm_\loc = \left\{ \left( u(\varepsilon), \lambda(\varepsilon) \right) :  |\varepsilon| < \varepsilon_0 \right\} \subset \Xspace_\infty \times (0,1),  \]
for some $\varepsilon_0 > 0$.   Moreover, the following statements hold along $\cm_\loc$.
\begin{enumerate}[label=\rm(\alph*)]
\item \label{bore asymptotics part} \textup{(Asymptotics)} The height function $u(\varepsilon)$ and upstream depth ratio $\lambda(\varepsilon)$ have leading form expressions given by
\begin{equation}
 \label{small bore asymptotics} 
  \begin{aligned}
    u(\varepsilon)(q,p) & = -\frac{\varepsilon}{2} \left(1+ \tanh{(\kappa_1 |\varepsilon| q)}\right)(1-|p|)   + O(\varepsilon^2)  \qquad \textrm{in }  \Xspace_\bdd, \\
    \lambda(\varepsilon) & =  \lambdastar + \varepsilon = \frac{\sqrt{\rho_1}}{\sqrt{\rho_1}+\sqrt{\rho_2}} + \varepsilon, 
  \end{aligned}
\end{equation}
for the explicit positive constant 
\[ \kappa_1^2 := \frac{3(\sqrt{\rho_1}+\sqrt{\rho_2})^4}{4\rho_1 ( \rho_1 -\sqrt{\rho_1}\sqrt{\rho_2} + \rho_2)}.\]  

\item \label{small bore monotone part} \textup{(Monotonicity)} If $\varepsilon > 0$, then $(u(\varepsilon), \lambda(\varepsilon))$ is a strictly decreasing monotone front, while for $\varepsilon < 0$, it is a strictly increasing monotone front.  

\item \label{small bore uniqueness part} \textup{(Uniqueness)} In a neighborhood of $(0,\lambdastar)$ in $\Xspace_\infty \times \mathbb{R}$, $\cm_\loc$ comprises all monotone fronts (up to translation).

\item \label{small bore kernel part} \textup{(Kernel)}  For all $0 < |\varepsilon| < \varepsilon_0$, the kernel of $\F_u(u(\varepsilon), \lambda(\varepsilon))$ is one dimensional and generated by $\partial_x u(\varepsilon)$.

\end{enumerate}
\end{theorem}
\begin{proof}
The existence of $\cm_\loc$ and the asymptotics in \eqref{small bore asymptotics} are essentially given by \cite[Theorem 7.1 and Corollary 7.3]{chen2022center}.  For that reason, we will only sketch the argument.  In fact, the proof is remarkably similar to that of Theorem~\ref{toy small amplitude theorem}: one applies a center manifold reduction to obtain a planar system whose bounded solutions give rise to solutions of height equation \eqref{eqn:u}.  This ODE at leading order coincides with \eqref{toy reduced truncated ODE}, only the coefficients involve $\rho, \lambda$, and $F$.  The conjugate flow analysis in Lemma~\ref{conjugate lemma} identifies the correct parameter regime in which the truncated reduced ODE admits a heteroclinic orbit, which is then shown to persist for the full reduced ODE using a conserved quantity derived from the flow force $\flowforce$.   

The monotonicity claimed in part~\ref{small bore monotone part} is a consequence of \cite[Remark 7.9]{chen2022center} and the maximum principle.  In particular, the center manifold reduction argument gives directly that $\pm (\partial_x u)(\varepsilon)(\placeholder, 0) < 0$ for $\pm \varepsilon > 0$.  To see that the rest of the streamlines are likewise monotone, we argue as in \cite[Theorem 7.8(a)]{chen2022center}.  

  Part~\ref{small bore uniqueness part} follows exactly as in the proof of Theorem~\ref{toy small amplitude theorem}\ref{toy uniqueness part}.
Verifying that the kernel along $\cm_\loc$ is one dimensional can likewise be accomplished using \cite[Theorem 1.6]{chen2022center} to reduce the question to the linearized reduced equation.  By the same argument as in Theorem~\ref{toy small amplitude theorem}, it is readily seen that this planar system cannot have  two linearly independent bounded solutions, and hence \eqref{kernel assumption} holds.  
\end{proof}
\begin{remark}\label{ekdv remark}
It bears mentioning that the truncated reduced ODE on the center manifold is exactly the stationary extended Korteweg--de Vries equation (eKdV).  This is a nonlinear dispersive PDE that has been derived previously as a model for bores in a certain long wave scaling (see, \cite{helfrich2006review} and the references contained therein).  
\end{remark}

\subsection{Velocity bounds} \label{velocity bounds section}

We now establish some a priori estimates for solutions of the internal wave problem.  These are crucial to the proof of Theorem~\ref{global bore theorem} as they will eventually allow us to conclude that the blowup alternative implies the stagnation limit \eqref{stagnation limit}.   The first task is to derive $L^\infty$ control of the velocity $(\partial_y\psi,-\partial_x\psi)$; this will ensure that the height equation \eqref{eqn:u} is uniformly elliptic along the global curve.

In the case of constant density rotational waves \cite{varvaruca2009extreme} or continuously stratified waves \cite{chen2018existence}, one can show that a modified pressure enjoys a maximum principle and hence has a lower bound. This can then be translated to an upper bound on the magnitude of the relative velocity through Bernoulli's law. For two layer flows, however, it is not clear whether a similar approach is feasible due to the transmission condition. We therefore follow the strategy of \cite{amick1986global} to first derive a local $H^1$ estimate on the stream function, and then use the classical monotonicity formula of Alt--Caffarelli--Friedman \cite{alt1984variational} in conjunction with Bernoulli's law to produce the bound for the velocity.  Because our non-dimensionalization differs from that in \cite{amick1986global}, we cannot simply quote their result.  Indeed, the precise dependence on parameters is crucially important, so it is necessary to carefully reprove the theorem in order to understand the effects of the new scaling.

\begin{lemma}\label{lem est for grad psi}
Let $(\psi, \eta, \lambda)$ be a solution of \eqref{eqn:stream} with $\psi_y < 0$ and $\lambda \in (0, 1)$.   Then, for all $m\in \R$ it holds that
\begin{equation}\label{est grad psi}
\int^{m+1}_{m-1}\int_{-\lambda}^{1 - \lambda} |\nabla \psi|^2 \,dy \, dx \le C_0
\end{equation}
where the constant $C_0 = C_0(\rho, F) > 0$.
\end{lemma}
\begin{proof}
As $\psi_y < 0$, we can use the $(q,p)$ variables to write
\begin{equation}\label{grad psi}
\int^{m+1}_{m-1}\int_{-\lambda}^{1 - \lambda} |\nabla \psi|^2 \,dy \, dx = \int^{m+1}_{m-1} \int_{-1}^1 h^2 {1 + u_{q}^2 \over h + u_{p}}\,dp \, dq.
\end{equation}
A second consequence is that $u_p > -h$, and hence $|u| < 1$.  

For any $m\in \R$, consider a cutoff function $\zeta = \zeta(q) \in C^\infty_c([m-2, m+2])$ with $0 \le \zeta \le 1$ and $\zeta \equiv 1$ on $[m-1, m+1]$. Multiplying \eqref{bore weak formulation} by $\zeta^2 u$ and integrating over $\Omega$, we get
\begin{equation}\label{test eqn}
\begin{split}
\iint_\Omega {\rho h^2 \zeta^2 \over 2} {2h + u_p \over (h + u_p)^2} \left( u_q^2 +{u_p^2 \over h^2} \right) \, dq \, dp
& = - \iint_\Omega 2\zeta\zeta' \rho h^2 {u u_q \over h + u_p} \, dq \, dp \\ 
& \qquad + {1 - \rho_2 \over F^2} \int_{\Gamma_1} \zeta^2 u^2 \,dq \\
& \le \epsilon \iint_\Omega \rho h^3 \zeta^2 {u^2 u_q^2 \over (h + u_p)^2} \, dq \, dp \\
& \qquad + {1\over \epsilon} \int^{m+2}_{m-2} \rho h (\zeta')^2 \, dq + {4(1 - \rho_2) \over F^2},
\end{split}
\end{equation}
where in the last line we used Young's inequality and the fact $|u| < 1$. Choosing $\epsilon = 1 /2$ in each layer, this gives the bound
\begin{equation}\label{est grad}
{1 \over 2} \int^{m+1}_{m-1} \int_{-1}^1 \rho {2h + u_p \over (h + u_{p})^2} u_p^2 \, dp \, dq + {1 \over 2} \int^{m+1}_{m-1} \int^1_{-1} \rho h^2 {u_q^2 \over h + u_{p}} \, dp \, dq \le C(\rho, F).
\end{equation}

On the other hand, 
\begin{equation*}
{1\over h + u_{p}} \le 
  \begin{cases}
    \displaystyle {2 \over h} & \text{ when } \displaystyle u_{p} \ge -{h \over 2},\\[2ex]
    \displaystyle {4\over h^2} {u_{p}^2 \over h + u_{p}} & \text{ otherwise},
  \end{cases}
\end{equation*}
and therefore
\begin{equation*}
\int^{m+1}_{m-1}\int^1_{-1} {h^2\over h + u_{p}}\,dp \, dq \le C(\rho, F).
\end{equation*}
Substituting the above estimate and \eqref{est grad} into \eqref{grad psi} yields the desired estimate \eqref{est grad psi}. 
\end{proof}

\begin{theorem}[Bounds on velocity] \label{thm est for u-deriv}
Let $(u, \lambda) \in \Xspace_\infty \times (\delta,1-\delta)$ be a strictly monotone front solution of \eqref{eqn:u} for some $\delta > 0$. Then, 
\begin{equation}\label{est u-deriv 1}
\left\| {1\over h + u_{p}} \right\|_{L^\infty(\Omega)} \le C
\end{equation}
and
\begin{equation}\label{est u-deriv 2}
\left\| {u_{q}\over h + u_{p}} \right\|_{L^\infty(\Omega)} \le C
\end{equation}
where the constant $C = C(\delta, \rho, F) > 0$.
\end{theorem}
\begin{remark} \label{uniform grad psi bound remark} Observe that, in light of \eqref{grad u grad psi}, this leads directly to a bound on the corresponding velocity field
\[ \| \nabla \psi \|_{L^\infty(\fluidD)} \leq C,\]
with $C = C(\delta, \rho, F) > 0$.
\end{remark}
\begin{proof}[Proof of Theorem~\ref{thm est for u-deriv}]
Differentiating \eqref{eqn:u:lap} and then applying the maximum principle, we know that $u_{p}$ attains its minimum either on $\partial\Omega$ or in the limits $q \to \pm\infty$.  As $u \in \Xspace_\infty$ is a monotone front, $u$ does not change signs in $\Omega$; for definiteness, say $u$ is strictly decreasing so that $u<0$ in $\Omega \cup \Gamma_1$.  By Lemma~\ref{conjugate lemma}, this implies that 
\[ \lim_{q \to -\infty} u_p = 0, \qquad 
\lim_{q \to \infty} u_p = U_+^\prime = (\lambda - \lambdastar) \signum{p},  \]
where note $\lambdastar < \lambda$.  Because $u$ vanishes on $\Gamma_0$, the Hopf boundary-point lemma tell us that $u_{1p} < 0$ on $\{p=-1\}$ and $u_{2p} > 0$ on $\{p=1\}$. Thus the minimum of $u_p$ is negative and is attained on $\{p=-1\}$, $\{p=0\}$, or in the downstream limit.  

Likewise, differentiating \eqref{eqn:u:lap} in $q$, we see that $u_{1q}$ obeys the maximum principle. Since $u_q < 0$ in $\Omega \cup \Gamma_1$ and $u_{1q} = 0$ on $\{p=-1\}$, the Hopf boundary-point lemma implies that $u_{1pq} < 0$ on $\{p=-1\}$. Therefore
\[ \min_{\Gamma_0} u_p = U_+^\prime(-1) = \lambdastar - \lambda.\] 
In other words, both the minimums of $u_p$ on $\Gamma_0$ and in the downstream limit are controlled by $\delta$ and $\rho$.  

The above argument and the relationship between $\nabla \psi$ and $\nabla u$ given by \eqref{grad u grad psi} imply that the estimate \eqref{est u-deriv 1} will follow from an upper bound of $\partial_y \psi_1$ and $\partial_y \psi_2$ on $\fluidS$, whereas the estimate \eqref{est u-deriv 2} is equivalent to an upper bound of $\psi_x$ in $\fluidD$. But $\psi_x$ is harmonic in each layer, and so it attains its maximum on $\partial\fluidD$.  As it vanishes identically on the upper and lower boundaries as well as in the limits $q \to \pm\infty$, we need only estimate it on $\fluidS$.

Consider now the pseudo-stream function $\tilde \psi : = \sqrt\rho \psi$. Obviously, it suffices to control $|\nabla \tilde \psi|$ on $\fluidS$. Following \cite[Theorem 6.8]{amick1986global}, let a point $(x,y) \in \fluidS$ be given and denote by $B_r$ the ball of radius $r$ centered there. We apply the monotonicity result of \cite[Lemma 5.1]{alt1984variational} to conclude that the function
\begin{equation*}
\phi(r) := \left({1\over r^2} \iint_{B_r \cap \fluidD_1} |\nabla \tilde \psi|^2 \, dx \, dy \right)\left( {1\over r^2} \iint_{B_r \cap \fluidD_2} |\nabla \tilde \psi|^2 \, dx \, dy \right)
\end{equation*}
is increasing in $r$ for $0< r < a:= \text{dist}(\fluidS, \partial \overline\fluidD)$. The regularity of $\tilde \psi$ implies that $\phi(r) \to \phi(0)$ as $r \to 0$, and hence by Lemma \ref{lem est for grad psi},
\begin{equation}\label{ACF}
\phi(0) = {\pi^2 \over 4} |\nabla \tilde \psi_1(x,y)|^2 |\nabla \tilde \psi_2(x,y)|^2 < \phi(a) \le C.
\end{equation}

Recall that by the Bernoulli condition \eqref{eqn:stream:dynamic} we have
\begin{equation*}
{1\over 2} \left( |\nabla \tilde \psi_2|^2 - |\nabla \tilde \psi_1|^2 \right) = {\rho - 1\over 2}\left( 1 - {2\over F^2} \eta \right) \qquad \textrm{on } \fluidS.
\end{equation*}
Since $\lambdastar < \eta < 0$, this combined with \eqref{ACF} implies that $|\nabla \tilde \psi|$ is bounded by $C$.  
\end{proof}
\begin{remark}\label{rk C1 control}
From the above theorem we see that $\|u_q\|_{L^\infty(\Omega)} \le C \|h + u_p\|_{L^\infty(\Omega)}$. Moreover, because we can always write
\begin{equation*}
u(q,p) = \left\{\begin{array}{ll}
\int^p_{-1} u_p(q, p')\,dp' \quad & \text{in } \Omega_1,\\
-\int^1_{p} u_p(q, p')\,dp'  \quad & \text{in } \Omega_2,
\end{array}\right.
\end{equation*}
it follows that
\begin{equation*}
\|u\|_{L^\infty(\Omega)} \le \|u_p\|_{L^\infty(\Omega)}, \qquad \|u\|_{C^1(\Omega)} \le C (1+ \|u_p\|_{C^0(\Omega)}).
\end{equation*}
\end{remark}

\subsection{Uniform regularity} \label{uniform regularity section}

The purpose of this section is to show that $\| u \|_\Xspace$ is controlled by $\| u_p \|_{L^\infty(\Omega)}$ together with lower bounds on $\lambda$ and $1-\lambda$.  When we construct the global bifurcation curve, this will allow us to conclude that blowup in norm coincides with either the stagnation limit or the interface coming into contact with the walls.   Arguments of this kind are well-known, for example, in the constant density and continuously stratified settings, where the system for the height function is elliptic with oblique boundary condition.  Here, however, we are dealing with a transmission boundary condition. Instead of applying a Schauder-type argument, we will again follow the idea of Amick--Turner \cite{amick1986global} and appeal to weak solution theory for elliptic equations of divergence form. For this, we quote the following theorem of Meyers, adjusted slightly to fit in our setting.
\begin{theorem}[Meyers, \cite{meyers1963estimate}]\label{thm meyers}
Let $\mathcal{D} \subset \R^2$ be a smooth bounded domain. Consider the equation
\begin{equation}\label{div prob}
\nabla \cdot (\mathbf A \nabla u) = \nabla \cdot G + g \ \text{ in }\ \mathcal{D}, \qquad u = 0 \ \text{ on } \partial \mathcal{D},
\end{equation}
where the matrix $\mathbf A = \mathbf A(x)$ has measurable coefficients and satisfies $c_1\mathbf{I} \le \mathbf A  \le {1\over c_1} \mathbf{I}$ for some $c_1 > 0$, where $\mathbf{I}$ is the $2\times 2$ identity matrix. Then there exists some $r = r(c_1) > 2$ such that, for any $G \in L^r(\mathcal{D})$ and $g \in L^2(\mathcal{D})$, \eqref{div prob} admits a unique solution $u \in W^{1,r}_0(\mathcal{D})$, which obeys the estimate
\begin{equation*}
\| \nabla u \|_{L^r(\mathcal{D})} \le C \left( \|G\|_{L^r(\mathcal{D})} + \|g\|_{L^2(\mathcal{D})} \right),
\end{equation*}
where $C = C(\mathcal{D},c_1, r) > 0.$
\end{theorem}

Our strategy will be to repeatedly differentiate \eqref{bore weak formulation} with respect to $q$ then apply Theorem \ref{thm meyers} to obtain $W^{1,r}$ estimates on successively higher-order $q$ derivative of $u$.  Eventually, this will lead to sufficient uniform H\"older regularity of the trace $u|_{\Gamma_1}$ so that Schauder theory for the Dirichlet problem in $\Omega_i$ furnishes the desired result.  Unlike Amick and Turner, we only have that $u$ and its derivatives are locally integrable, so these estimates must be carried out with additional care.

Now for a fixed $\delta > 0$, we consider solutions $(u,\lambda)$ of the height equation satisfying
\begin{equation}\label{fixing}
\inf_{\Omega}(h + u_{p}) > \delta, \qquad \|u\|_{C^1(\Omega)} + {1\over h_1} + {1\over h_2} < {1\over \delta},
\end{equation}
where recall that $h_1 = \lambda$ and $h_2 = 1-\lambda$.    It is easy to see that under condition \eqref{fixing}, we have uniform ellipticity:
\begin{equation}\label{bounds Hessian}
D^2 f (\nabla u,h) \ge {\delta^5 \over 4 (1 + \delta)^3} \mathbf{I} =: c_1 \mathbf{I}, \qquad |D^2 f(\nabla u, h)| \lesssim {1\over \delta^5}.
\end{equation}

  Throughout this section, we will denote  $\Omega_{m,k} := [m-k, m+k] \times (-1,1)$, for $m \in \R$ and $k > 0$.  It is important to note that that these are connected domains that contain a portion of the internal interface $\Gamma_1$.  
\begin{lemma}\label{lem local est}
Let $(u,h)$ be a solution of the height equation \eqref{eqn:u} satisfying \eqref{fixing}. Then there exists an $r = r(\delta) > 2$ such that for any $m\in \R$,
\begin{align}
\|\nabla v\|_{L^2(\Omega_{m,2})} & \le C \|\nabla u\|_{L^2(\Omega_{m,3})} \le C, \label{L2 est}\\
\|\nabla u\|_{L^r(\Omega_{m,1})} & \le C, \label{Holder est} \\
\|\nabla v\|_{L^r(\Omega_{m,1})} & \le C, \label{deriv Holder est}
\end{align}
where $v: = u_q$ and $C = C(\delta) > 0$.
\end{lemma}
\begin{proof}
Consider a cutoff function $\tilde \zeta \in C^\infty_c([m - 4, m+4])$ with $\tilde \zeta \equiv 1$ on $[m-3, m+3]$ and $0 \le \tilde \zeta \le 1$. Multiplying the equation by $\tilde{\zeta}^2 u$ and integrating by parts as in \eqref{test eqn} yields 
\begin{equation*}
\iint_\Omega {\rho h^2 \tilde \zeta^2 \over 2} {2h + u_p \over (h + u_p)^2} \left( u_q^2 +{u_p^2 \over h^2} \right) \, dq \, dp
 = - \iint_\Omega 2 \tilde \zeta \tilde \zeta' \rho h^2 {u u_q \over h + u_p} \, dq \, dp + {2 \over F^2} \iint_\Omega \rho \tilde \zeta^2 uu_p \, dq \, dp. 
\end{equation*}
Condition \eqref{fixing} guarantees that ${2h + u_p \over (h + u_p)^2} \ge C(\delta) > 0$. Applying Young's inequality we conclude that
\[
\|\nabla u\|_{L^2(\Omega_{m,3})} \le C \|u\|_{L^2(\Omega_{m,4})},
\]
and hence from Remark \ref{rk C1 control} the second estimate in \eqref{L2 est} holds. 

Now,  differentiating \eqref{eqn:u} in $q$, we obtain the following equation for $v := u_q$,
\begin{equation}\label{eqn:v}
\left\{
\begin{aligned}
\nabla \cdot \left( \rho D^2 f(\nabla u, h) \nabla v - {\rho \over F^2} v \mathbf{e}_2  \right) + {\rho \over F^2} v_p &= 0 && \text{ in } \Omega \cup \Gamma_1, \\
v &= 0 && \text{ on } \Gamma_0,
\end{aligned} \right.
\end{equation}
where $\mathbf{e}_2 := (0,1)^T$.  Consider another cutoff function $\tilde\zeta \in C^\infty_c([m - 3, m+3])$ with $\tilde \zeta \equiv 1$ on $[m-2, m+2]$ and $0 \le \tilde \zeta \le 1$. Testing the function $\tilde \zeta^2 v$ against \eqref{eqn:v},  we have
\begin{equation}\label{int v}
\begin{split}
 \iint_{\Omega} \rho \tilde \zeta^2 \nabla v \cdot \left( D^2 f(\nabla u,h) \nabla v \right) \, dq \, dp & = \iint_{\Omega} \rho  (\tilde \zeta^2)' v \left( D^2 f(\nabla u,h) \nabla v \right) \cdot \mathbf{e}_1 \, dq \, dp \\
 & \qquad + {2 \over F^2} \iint_\Omega \tilde\zeta^2 v v_p \, dq \, dp.
\end{split}
\end{equation}
Thus a use of Young's inequality and \eqref{bounds Hessian} leads to 
\[
\|\nabla v\|_{L^2(\Omega_{m,2})} \le C \|v\|_{L^2(\Omega_{m,3})},
\]
and hence the first estimate of \eqref{L2 est}.

Next, we will use Theorem~\ref{thm meyers} to derive the $\dot W^{1,r}$ estimates asserted in \eqref{Holder est} and \eqref{deriv Holder est}.  Rewrite \eqref{bore weak formulation} as
\begin{equation*}
\left\{ \begin{aligned}
\nabla \cdot \left( \mathcal{A}(q,p, h) \nabla u - {\rho \over F^2} u \mathbf{e}_2  \right) + {\rho \over F^2} u_p & = 0 &\qquad& \text{in }  \Omega \cup \Gamma_1 \\
u & = 0 & \qquad & \textrm{on } \Gamma_0, \end{aligned}\right.
\end{equation*}
where 
\begin{equation*}
\mathcal{A}(q, p, h) := \rho \int^1_0 D^2 f\big( t \nabla u(q, p), h \big)\,dt.
\end{equation*}
Since \eqref{fixing} holds for convex combination of $u$, we know that $c_1 \mathbf{I} \leq \mathcal{A} \leq \frac{1}{c_1} \mathbf{I}$, for $c_1 = O(\delta^5)$ as in \eqref{bounds Hessian}. Let $\bar u: = \zeta u$ for $\zeta$ a cutoff function as in the proof of Lemma \ref{lem est for grad psi} and consider a domain $\mathcal D$ with a smooth boundary and $\Omega_{m,2} \subset \mathcal D \subset \Omega_{m,3}$. Then $\bar u$ satisfies
\begin{equation*} \left\{
\begin{aligned}
\nabla \cdot \big( \mathcal A \nabla \bar u \big) & = \nabla \cdot \left( u \mathcal A \nabla \zeta + {\rho \over F^2} \zeta u \mathbf{e}_2 \right) + \nabla \zeta \cdot \big( \mathcal A \nabla u \big) - {\rho \over F^2} \zeta u_p & \qquad & \text{in } \mathcal D, \\
\bar u  & = 0 & \qquad & \text{on } \partial \mathcal D. 
\end{aligned} \right.
\end{equation*}
Hence Theorem \ref{thm meyers} and translation invariance imply
\begin{equation*}
\begin{split}
\| \nabla \bar u \|_{L^r(\mathcal D)} & \le C \left( \left\| u \mathcal A \nabla \zeta +  {\rho \over F^2} \zeta u \mathbf{e}_2  \right\|_{L^r(\mathcal D)} + \left\| \nabla \zeta \cdot \big( \mathcal A \nabla u \big) - {\rho\over F^2} \zeta u_p \right\|_{L^2(\mathcal D)} \right) \le C,
\end{split}
\end{equation*}
where $r = r(\delta) >2$ and $C = C(\delta)$. We have therefore proved \eqref{Holder est}.  Applying a similar argument to \eqref{eqn:v}, likewise gives \eqref{deriv Holder est}.
\end{proof}

If we further differentiate \eqref{eqn:v}, we can obtain an $L^r$ bound on the gradient of $\tilde v := v_q$. The equation for $\tilde v$ is
\begin{equation}\label{eqn:vq} \left\{
\begin{aligned}
\partial_i \left( \rho f^{ij} \partial_j \tilde v  \right) + \partial_i \left( \rho f^{ijk} \partial_j v \partial_k v  \right) - \partial_2\left( {\rho \over F^2} \tilde v \right) + {\rho \over F^2} \partial_2 \tilde v &= 0 && \text{ in } \Omega \cup \Gamma_1, \\
\tilde v &= 0 && \text{ on } \Gamma_0,
\end{aligned} \right.
\end{equation}
where $\partial_1 = \partial_q$, $\partial_2 = \partial_p$ and we are using the shorthand
  \[  f^{ij} := (\partial_{\xi_i}\partial_{\xi_j} f)(\nabla u,h), \quad f^{ijk} := (\partial_{\xi_i} \partial_{\xi_j} \partial_{\xi_k} f)(\nabla u,h).\]

As before, we can test against a cutoff function $\zeta= \zeta(q)$ to obtain a local version for $\tilde w := \zeta \tilde v$ that reads
\begin{subequations}\label{subeqn vq}
\begin{equation}\label{eqn:vq local}
\nabla \cdot \left( \rho D^2 f \nabla \tilde w \right) =  \nabla \cdot \tilde{G} + g \qquad \text{in }\ \mathcal D,
\end{equation}
for a smooth domain $\Omega_{m,2} \subset \mathcal D \subset \Omega_{m,3}$.  Here, the forcing terms are given explicitly by
\begin{equation}
\tilde{G} = \rho \tilde v D^2 f \nabla \zeta - \zeta G, \quad g = \rho \nabla \zeta \cdot D^2 f \nabla \tilde v + \nabla \zeta \cdot G - {\rho \over F^2} \zeta  \partial_2 \tilde v
\end{equation}
where 
\begin{equation}
G = \rho \begin{pmatrix}  f^{1jk} \partial_j v \partial_k v \\  f^{2jk} \partial_j v \partial_k v  \end{pmatrix} - {\rho \over F^2} \tilde v \mathbf{e}_2.
\end{equation}
\end{subequations}
Therefore, in order to apply Theorem \ref{thm meyers}, we need to ensure that $\nabla \tilde v \in L^2$ and $\nabla v \in L^s$ for some $s > 2$ sufficiently large.

\begin{lemma}\label{lem gradv L^4 est}
Let $(u,h)$ be a solution of \eqref{eqn:u} that satisfies \eqref{fixing}. Then there exist an $r = r(\delta) > 2$ and a constant $C = C(\delta)$ such that, for any $m\in \R$,
\begin{equation}\label{gradv L^4 est}
\| \nabla v \|^4_{L^4(\Omega_{m,1})} \le C \left( \|\nabla v_q\|^{4 - r}_{L^2(\Omega_{m,1})}  + 1 \right),
\end{equation}
where $v = u_q$.
\end{lemma}
\begin{proof}
Throughout the proof, let $C > 0$ denote a generic positive constant depending only on $\delta$ and $\| u_p \|_{C^0}$.  Recall that $v$ satisfies \eqref{eqn:v}. From \eqref{L2 est} and \eqref{deriv Holder est} we know that there exists some $r = r(\delta) > 2$ such that
\begin{equation}\label{v W1r est}
\|v\|_{W^{1,\tilde r}(\Omega_{m,1})} \le C \|u_p\|_{C^0(\Omega)} \leq C
\end{equation}
for all $\tilde r \in [2, r]$ and $m \in \R$.  Thus \eqref{gradv L^4 est} holds if $r \ge 4$. So we assume that $r < 4$.

By the Gagliardo--Nirenberg inequality, we have for any $g \in H^1(\Omega_{m,k})$,
\begin{equation}\label{GN}
\|g\|_{L^4(\Omega_{m,k})} \lesssim \|\nabla g\|^{1-{\theta/4}}_{L^2(\Omega_{m,k})} \| g \|^{\theta/4}_{L^s(\Omega_{m,k})}
\end{equation}
for all $\theta \in [1, 4]$. In particular, applying \eqref{GN} with $\theta=2$ to \eqref{v W1r est}, we infer that 
\begin{equation}\label{v L4 est}
\|v\|^4_{L^4(\Omega_{m,1})} \lesssim \| \nabla v \|^2_{L^2(\Omega_{m,1})} \|v\|^2_{L^2(\Omega_{m,1})} \le C \|v\|^4_{L^2(\Omega_{m,2})} \le C \|u_p\|^4_{C^0(\Omega)}.
\end{equation}

Now, the first equation of \eqref{eqn:v} can be rewritten as
\begin{align*}
\partial_p w = -\rho \left( f^{111}\tilde v^2 + (f^{112} + f^{121}) v_p \tilde v + f^{122} v_p^2 + f^{1i} \partial_i \tilde v   \right) - {\rho \over F^2} v_p,
\end{align*}
where 
\begin{equation}\label{def w}
w := \rho f^{2i} \partial_i v  - {\rho \over F^2} v \quad \text{and} \quad \tilde v := v_q. 
\end{equation}
Moreover, from the definition of $w$ we have
\begin{align*}
\partial_q w =  \rho \left( f^{211}\tilde v^2 + (f^{212} + f^{221}) v_p \tilde v + f^{122} v_p^2 + f^{2i} \partial_i \tilde v  \right) - {\rho \over F^2} v_q.
\end{align*}
It follows from this and \eqref{v W1r est} that
\begin{equation}\label{w est}
\begin{split}
& \|w\|_{L^{\tilde r}(\Omega_{m,1})} \le C, \\
& \|\nabla w\|^2_{L^2(\Omega_{m,1})} \le C \left( \|\nabla \tilde v\|^2_{L^2(\Omega_{m,1})} + \|\nabla v\|^2_{L^2(\Omega_{m,1})} + \|v_q\|^4_{L^4(\Omega_{m,1})} + \|v_p\|^4_{L^4(\Omega_{m,1})} \right).
\end{split}
\end{equation}
Hence from the definition of $w$ \eqref{def w}, \eqref{bounds Hessian}, \eqref{v L4 est} and the Gagliardo--Nireberg inequality \eqref{GN} with $\theta = r$, we have
\begin{equation*}
\begin{split}
\|v_p\|^4_{L^4(\Omega_{m,1})} & \le C \left( \|w\|^4_{L^4(\Omega_{m,1})} + \|v_q\|^4_{L^4(\Omega_{m,1})} + \|v\|^4_{L^4(\Omega_{m,1})} \right) \\
& \leq C \left( \|\nabla w\|^{4-r}_{L^2(\Omega_{m,1})}\|w\|^r_{L^r(\Omega_{m,1})} + \|v_q\|^4_{L^4(\Omega_{m,1})} + 1 \right) \\
& \le C \left( \left( \|\nabla \tilde v\|^2_{L^2(\Omega_{m,1})} + \|v_q\|^4_{L^4(\Omega_{m,1})} + \|v_p\|^4_{L^4(\Omega_{m,1})} \right)^{2 - r/2} + \|v_q\|^4_{L^4(\Omega_{m,1})} + 1 \right).
\end{split}
\end{equation*}
Since $r > 2$, the above bound yields
\begin{equation*}
\|v_p\|^4_{L^4(\Omega_{m,1})} \le C \left( \|\nabla v_q\|^{4 - r}_{L^2(\Omega_{m,1})} + \|v_q\|^4_{L^4(\Omega_{m,1})} +1 \right).
\end{equation*}
Applying \eqref{GN} to $v_q$ with $\theta = r$ and using \eqref{v W1r est}, we see that
\begin{equation*}
\|v_q\|^4_{L^4(\Omega_{m,1})} \lesssim \|\nabla v_q\|^{4 - r}_{L^2(\Omega_{m,1})} \|v_q\|^r_{L^r(\Omega_{m,1})} \le C  \|\nabla v_q\|^{4 - r}_{L^2(\Omega_{m,1})}.
\end{equation*}
Therefore
\begin{equation*}
\|v_p\|^4_{L^4(\Omega_{m,1})} \le C \left( \|\nabla v_q\|^{4 - r}_{L^2(\Omega_{m,1})} + 1 \right),
\end{equation*}
which implies \eqref{gradv L^4 est}.
\end{proof}

\begin{lemma}\label{lem gradv_q L2 est}
Let $(u,h)$ be a solution of \eqref{eqn:u} satisfying \eqref{fixing}. Then there exists a constant $C = C(\delta)$ such that for any $m\in \R$,
\begin{equation}\label{gradv_q L2 est}
\| \nabla v_q \|_{L^2(\Omega_{m,1})} \le C, 
\end{equation}
where $v = u_q$.
\end{lemma}
\begin{proof}
Recall that $\tilde v := v_q$ solves \eqref{eqn:vq}. We now test $\zeta^2 \tilde v$ against \eqref{eqn:v} with the cutoff function $\zeta$ given as in Lemma \ref{lem est for grad psi}. Using the ellipticity condition \eqref{bounds Hessian}, Young's inequality, \eqref{L2 est} and \eqref{gradv L^4 est}, we discover that
\begin{equation*}
\begin{split}
c_1 \iint_\Omega \rho \zeta^2 |\nabla \tilde v|^2 \, dq \, dp & \le \iint_\Omega \rho \zeta^2 f^{ij} \partial_j \tilde v \partial_i \tilde v \, dq \, dp \\
& = - \iint_\Omega \rho \zeta^2 f^{ijk} \partial_j v \partial_k v \partial_i \tilde v \, dq \, dp - \iint_\Omega \rho  (\zeta^2)' \tilde v f^{1j} \partial_j \tilde v \, dq \, dp \\
& \quad - \iint_\Omega \rho (\zeta^2)' \tilde v f^{1jk} \partial_j v \partial_k v \, dq \, dp + \iint_\Omega {2\rho \over F^2} \zeta^2 \tilde v \tilde v_p \, dq \, dp \\
& \le {c_1 \over 2} \iint_\Omega \rho \zeta^2 |\nabla \tilde v|^2 \, dq \, dp + C \left( \| \nabla v\|^4_{L^4(\Omega_{m,2})} + \|\tilde v\|^2_{L^2(\Omega_{m,2})} \right),
\end{split}
\end{equation*}
which implies that
\[
\| \nabla \tilde v \|^2_{L^2(\Omega_{m,1})} \le C \left( \| \nabla \tilde v\|^{4 - r}_{L^2(\Omega_{m,2})} + 1 \right).
\]
Taking the supremum over all $m$, we find that 
\[
\sup_m \| \nabla \tilde v \|^2_{L^2(\Omega_{m,1})} \le C\left( \sup_m \| \nabla \tilde v \|_{L^2(\Omega_{m,2})}^{4-r} + 1 \right) \lesssim 
C\left( \sup_m \| \nabla \tilde v \|_{L^2(\Omega_{m,1})}^{4-r} + 1 \right).
\]
Finally, because $r > 2$, this gives
\[
\sup_m \| \nabla \tilde v \|_{L^2(\Omega_{m,1})} \le C,
\]
proving \eqref{gradv_q L2 est}.
\end{proof}

\begin{corollary}\label{cor gradv Lp est}
Let $(u,h)$ be a solution to \eqref{eqn:u} that satisfies \eqref{fixing}. Then, for all $s \ge 2$, there exists a constant $C = C(\delta,s)$ such that, for any $m\in \R$,
\begin{equation}\label{gradv Lp est}
\|\nabla v\|_{L^s(\Omega_{m,1})} \le C.
\end{equation}
\end{corollary}
\begin{proof}
From \eqref{gradv L^4 est}, \eqref{gradv_q L2 est} and \eqref{w est} we know that
\begin{equation*}
\|\nabla v\|^4_{L^4(\Omega_{m,1})} + \|\nabla w \|^2_{L^2(\Omega_{m,1})} \le C,
\end{equation*}
with $w$ is given in \eqref{def w}. Sobolev embedding then implies that $w \in L^s(\Omega_{m,1})$ for all $s\ge 2$. Similarly, from \eqref{v W1r est} and \eqref{gradv_q L2 est} we know that $v, v_q \in L^s(\Omega_{m,1})$. In this way, \eqref{gradv Lp est} follows from the fact that $f^{22}$ is bounded from below by a positive constant depending on $\delta$.
\end{proof}

Now we can apply the elliptic regularity result Theorem~\ref{thm meyers} to $v_q$, which gives the following.
\begin{lemma}\label{lem gradtildev est}
Let $(u,h)$ be a solution of \eqref{eqn:u} satisfying \eqref{fixing}. Then there exist some $s_* = s_*(\delta_*) > 2$ and a constant $C = C(\delta)$ such that, for any $m\in \R$,
\begin{equation}\label{gradtildev est}
\|\nabla v_q\|_{L^{s_*}(\Omega_{m,1})} \le C.
\end{equation}
\end{lemma}
\begin{proof}
Recall that we obtained an equation \eqref{subeqn vq} governing the a cutoff version of $\tilde v = v_q$ on a smooth subdomain $\Omega_{m,2} \subset \mathcal D \subset \Omega_{m,3}$. From \eqref{gradv Lp est} we know that $G \in L^s(\Omega_{m,2})$ for all $s \ge 2$, and hence $\tilde G \in L^s(\mathcal D)$. Also \eqref{gradv_q L2 est} implies that $g \in L^2(\mathcal D)$. Therefore, \eqref{gradtildev est} is a direct consequence of Theorem \ref{thm meyers}.
\end{proof}

Repeating the same argument as in Lemmas \ref{lem gradv L^4 est}--\ref{lem gradtildev est} one can actually obtain
\begin{corollary}\label{cor higher gradient}
Let $(u,h)$ be a solution of \eqref{eqn:u} satisfying \eqref{fixing}. Then for any integer $k \ge 2$, there exist some $s_* = s_*(k, \delta) > 2$ and a constant $C = C(k, \delta)$ such that, for any $m\in \R$,
\begin{equation}
\|\nabla \partial^k_q u\|_{L^{s_*}(\Omega_{m,1})} \le C.
\end{equation}
\end{corollary}

Recalling that $\Gamma_1 \subset \cup_{m} \Omega_{m,1}$, the above result and Morrey's inequality ensure that the trace $u|_{\Gamma_1}$ can be bootstrapped to arbitrarily high H\"older regularity.  This is the key to proving our main result of the subsection.

\begin{theorem}[Uniform regularity and analyticity] \label{thm reg}
Let $(u,h)$ be a solution to \eqref{eqn:u} satisfying \eqref{fixing}. Then for any integer $k \geq  2$ and any $\alpha \in (0,1)$ there is a constant $C = C(k, \delta) > 0$ such that $u \in C^{k+ \alpha}(\overline{\Omega_i})$ and 
\[
\|u\|_{C^{k+\alpha}(\overline{\Omega_i})} \le C.
\]
Moreover $u$ is real analytic in $\overline{\Omega_i}$.
\end{theorem}
\begin{proof}
From Lemma \ref{lem gradtildev est} and Sobolev embedding we know that the trace $u|_{\Gamma_1}$ is of class $C^{k + \alpha}$. Furthermore the conditions of \cite[Theorem 6.19]{gilbarg2001elliptic} are met and hence the uniform regularity follows from standard Schauder theory. 

The analyticity can be inferred from the classical result of \cite{kinderlehrer1978regularity}. More specifically, returning to the stream function formulation \eqref{eqn:stream} one may apply \cite[Theorem 3.2]{kinderlehrer1978regularity} to obtain analyticity.
\end{proof}

\subsection{Global bifurcation} \label{final proof section}

We now have all the necessary tools for the proof of our existence theorem for large-amplitude bores.    

\begin{proof}[Proof of Theorem~\ref{global bore theorem}]
  In Theorem~\ref{small bores theorem}, it was shown that there exists a local curve $\cm_\loc$ of solutions to the internal wave problem, and the kernel condition \eqref{kernel assumption} holds along it.  Let $\cm_\loc^\pm$ denote the segment corresponding to the parameter values $\pm\varepsilon \in (0, \varepsilon_0)$.  By construction, both $\cm_\loc^-$ and $\cm_\loc^+$ emanate from the trivial solution $(0,\lambdastar)$, so that \eqref{local singular assumption} follows from Lemma~\ref{no ripples lemma}.  It is also clear that they lie in $\genU_\infty$ and $\genU^\delta$, for all $0 < \delta \ll 1$.  Moreover, we confirmed in Lemma~\ref{no ripples lemma} that the spectral assumption \eqref{spectral assumption} holds along $\cm_\loc^\pm$.

Applying Corollary~\ref{cor_transmission GBT} to $\cm_\loc^\pm$ with $\F|_{\genU^\delta}$ gives a global curve $\cm_\delta^\pm \subset \genU^\delta$ of strictly monotone bore solutions.  By maximality, they are nested, and hence their union $\cm^\pm := \cup_{\delta > 0} \cm_\delta^\pm \subset \genU_\infty$ is also a $C^0$ curve of strictly monotone fronts.  Like the local curves, the fronts on $\cm^+$ are strictly decreasing while those on $\cm^-$ are strictly increasing.  Expressed in Eulerian variables through \eqref{grad u grad psi}, this gives \eqref{monotonicity Euler variables} proving part~\ref{global bore monotone part}.  We already concluded that part~\ref{global bore laminar part} holds, so it remains only to prove part~\ref{global bore limit part}.  

  With that in mind, consider the limiting behavior along $\cm^+$; an identical argument will apply to $\cm^-$.  
  Thanks to our characterization in Lemma~\ref{conjugate lemma} of the flows that are conjugate to $0$,
  Corollary~\ref{conjugate corollary} implies that every solution along $(u(s),\lambda(s)) \in \cm^+$ has
  \begin{align*}
    \lim_{q \to +\infty} u(s)(q,p) = (\lambdastar-\lambda(s))(1-\abs p),
  \end{align*}
  where $\lambdastar$ is given in \eqref{lambda star}. This shows that $\liminf_{s} \lambda(s) \geq \lambdastar$, as otherwise strict monotonicity would fail along $\cm^+$.  
  It also rules out heteroclinic degeneracy \ref{gen hetero degeneracy}. Indeed, if \ref{gen hetero degeneracy} were to occur, then by Lemma~\ref{triple conjugacy lemma} the three flows 
   \begin{equation*}
     \lim_{q \to \mp\infty} u_*(q, \placeholder),
     \quad 
     \lim_{n \to \infty} \lim_{q \to +\infty} u(s_n)(q, \placeholder) = (\lambdastar-\lambda_*)(1-\abs p),
     \quad 
     \lim_{n \to \infty} \lim_{q \to -\infty} u(s_n)(q, \placeholder) = 0,
   \end{equation*}    
   must all be conjugate and distinct, contradicting Lemma~\ref{conjugate lemma}. Likewise, Lemma~\ref{no ripples lemma} and Theorem~\ref{small bores theorem}\ref{small bore uniqueness part}
     ensure that the spectral degeneracy alternative \ref{gen ripples} does not happen.  Thus we are left only with blowup \ref{gen blowup alternative}.

Suppose first that the internal interface stays bounded away from the walls: 
\[  \limsup_{s \to \infty} \lambda(s) < 1.\] 
Note that we already have a uniform lower bound $\lambda(s) > \lambdastar > 0$.   By Theorem~\ref{thm est for u-deriv}, it follows that $\cm^+ \subset \genU^\delta$ for some $\delta > 0$.  Recalling the definition of $N(s)$ in \eqref{gen global blowup}, this implies that $\|u(s)\|_{\Xspace} \to \infty$ as $s \to \infty$.  In light of Theorem~\ref{thm reg}, this can occur only if $\| \partial_p u(s) \|_{C^0(\Omega)} \to \infty$.
 
 To see that this leads to stagnation on the interface \eqref{stagnation limit}, let $\psi(s)$ be the corresponding stream functions and $\fluidD(s)$ the fluid domain.  Translated to Eulerian variables using \eqref{grad u grad psi}, the blowup alternative becomes 
 \[ \lim_{s \to \infty} \left\| \frac{1}{\partial_y \psi(s)} \right\|_{C^0(\fluidD(s))} = \infty.\]
 Because $\partial_y \psi(s)$ is harmonic in $\fluidD(s)$, it cannot attain its minimum in the interior of either layer.   Also, $\partial_x \psi(s)$ vanishes identically on the upper and lower boundaries, and so we have $\partial_y^2 \psi(s) = -\partial_x^2 \psi(s) = 0$ along them as well.  The Hopf boundary-point lemma therefore implies that $\partial_y \psi(s)$ does not attain its minimum or maximum on the walls.  But the upstream limit of $\partial_y \psi(s)$ is simply $-1$, while the downstream limit is bounded uniformly away from $0$ in terms of $1-\lambda(s)$.  We must then have 
 \[ \inf_{\fluidD(s)} |\partial_y \psi(s)| = \inf_{\fluidS(s)} |\partial_y \psi(s)| = -\sup_{\fluidS(s)} \partial_y \psi(s) \to 0 \qquad \textrm{as } s \to \infty.\] 

Assume instead that $\lambda(s) \to 1$.  Then from \eqref{psi value on top} we have
\[ h_2 = 1-\lambda(s) = -\int_{\eta(s)(x)}^{1-\lambda(s)} \partial_y \psi(s)(x, y) \, dy \qquad \textrm{for all } x \in \mathbb{R},\]
and thus 
\[ \inf_{\eta(s)(x) < y < 1-\lambda(s)} | \partial_y \psi(s)(x,y) | \leq \frac{1-\lambda(s)}{1-\lambda(s) - \eta(s)(x)},\]
which in particular means that a stagnation point develops on the interface.  The proof of the theorem is therefore complete.  
\end{proof}

\subsection{Limiting interfaces} \label{limit interface section}
We now turn to the proof of Theorem~\ref{limit eta theorem} on the limiting behavior of the free surface profile along the curves $\cm^\pm$.  

\subsubsection*{Lower bound for elevation bores} 

An important step in proving Theorem~\ref{limit eta theorem}\ref{elevation part} is to ensure that along $\cm^-$, the interface does not come into contact with the lower boundary. Throughout the subsection we will suppose that the interface does not overturn meaning that $\eta(s)$ is uniformly bounded in Lipschitz norm:
\begin{equation}
  \limsup_{s \to -\infty} \| \partial_x \eta(s) \|_{L^\infty(\mathbb{R})} = M < \infty. \label{no overturning upstream} 
\end{equation}  
It follows that  $|\nabla \psi_i(s)| \eqsim_M |\partial_y \psi_i(s)|$ on $\fluidS(s)$, and thus in the stagnation limit \eqref{stagnation limit} we have that the full gradient vanishes.  

Observe also that, because each $\eta(s)$ is monotonically increasing and vanishes upstream, there exists a unique $x(s) \in \mathbb{R}$ such that $\eta(s)(x(s)) = \sigma$, where $\sigma := \min\{F^2/4, \lambdastar/4\}$.  The uniform Lipschitz bound \eqref{no overturning upstream} then gives 
\begin{equation}
  \frac{\sigma}{2} < \eta(s) < {\sigma} \qquad \textrm{on } (x(s)-L, \, x(s)), \label{lower bound eta} 
\end{equation}
for $L := \sigma/(2M)$.  

Since $0 < \lambda(s) < \lambdastar$, we may pass to a subsequence along which $\lambda(s) \to \lambda^* \in [0,\lambdastar]$.  The case $\lambda^* = 0$ indicates the lower layer has collapsed, while the fact $\lambda^* \leq \lambdastar$ is a consequence of the preservation of monotonicity.  

Consider now the translated family of profiles 
\[ \widetilde \eta(s) := \eta(s)(\placeholder + x(s)-L/2)\]
with domain $\mathcal{I} := (-L/2,L/2)$.  As it is uniformly bounded in $\Lip(\overline{\mathcal{I}})$, we can extract a  subsequence converging in $C^{\varepsilon}$ for all $\varepsilon \in (0,1)$ to some $\eta^* \in \Lip(\overline{\mathcal{I}})$.  By \eqref{lower bound eta}, the entire family, $\tilde\eta(s)$ is uniformly positive  on $\mathcal{I}$.  

Denote $\mathcal{S}(s) := \{ (x,\widetilde \eta(s)(x)) : x \in \mathcal{I} \}$, $\mathcal{D}(s) := \mathcal{I} \times (-\lambda(s), 1-\lambda(s))$,  and let $\widetilde \psi_i(s)$ be the corresponding translated stream functions.  From Lemma~\ref{lem est for grad psi}, we have that $\nabla \widetilde\psi(s)$ is bounded uniformly in $L^2(\mathcal{D}(s))$.  As in the proof of Theorem~\ref{thm est for u-deriv}, the Alt--Caffarelli--Friedman monotonicity formula \eqref{ACF} and dynamic condition \eqref{eqn:stream:dynamic} together give 
\[ \limsup_{s \to \infty} \| \nu \cdot \nabla \widetilde \psi_i(s) \|_{L^\infty(\mathcal{S}(s))} \lesssim 1,\]
where $\nu$ is the outward unit normal to $\mathcal{D}_1(s)$ along $\mathcal{S}(s)$.  Since each $\mathcal{S}(s)$ is $C^{2}$, this leads to uniform Lipschitz control of $\widetilde \psi_i(s)$ in a neighborhood $\mathcal{N}$ of $\mathcal{S}(s)$; see, for example, \cite[Lemma 11.19]{caffarelli2005geometric}.  On the other hand, $\widetilde \psi(s)$ satisfies homogeneous Neumann conditions on the top and bottom of $\mathcal{D}$, so Schauder theory gives uniform bounds on $\widetilde \psi(s)$ in $C^{2+\alpha}(\mathcal{D}(s) \setminus \mathcal{N})$.

Now, let $\bar \psi_1(s) \in \Lip(\overline{\mathcal{I}} \times [-1,1])$ be the function found by extending $\widetilde\psi_1(s)$ by zero above $\mathcal{S}(s)$ and by $\lambda(s)$ below $y = -\lambda(s)$.  Letting $\bar \psi_2(s)$ be defined likewise, the above argument show that $\bar\psi_i(s)$ is uniformly Lipschitz on $\mathcal{D}(s)$.  Then, along a subsequence we have 
\begin{equation}
 \label{psi_i convergence}
  \begin{aligned}
    \bar \psi_i(s) & \overset{C^{\varepsilon}}{\longrightarrow}  \psi_i^* \in \Lip(\overline{\mathcal{D}})  & \quad & \textrm{for all } \varepsilon \in (0,1),\\
    \nabla \bar \psi_i(s) & \longrightarrow \nabla \psi_i^*& \quad& \textrm{weak-$\ast$ } L^\infty,  
  \end{aligned}
\end{equation} 
where $\mathcal{D} := \mathcal{I} \times (-\lambda^*,1-\lambda^*)$.  Denote 
\begin{equation}
 \label{def elevation psi*} 
  \begin{aligned}
    \psi^* & := \psi_1^*+\psi_2^*,  & \qquad
    \mathcal{S} & := \{ y = \eta^*(x) \} \cap \mathcal{D}, \\
    \mathcal{D}_1 & :=  \{ y < \eta^*(x) \}  \cap \mathcal{D}, & \qquad
  \mathcal{D}_2 & :=  \{ y > \eta^*(x) \} \cap \mathcal{D}.  
  \end{aligned} 
\end{equation}

\begin{lemma} \label{elevation limiting problem lemma}
Assume that along $\cm^-$ the interface does not overturn \eqref{no overturning upstream}.  Let $\psi^*$ and the sets $\mathcal{D}_1$, $\mathcal{D}_2$, and $\mathcal{S}$ be given as in \eqref{def elevation psi*}.  Then
 \[ \psi^* \in C^{2+\alpha}(\mathcal{D}_1) \cap C^{2+\alpha}(\mathcal{D}_2), \qquad \psi^* < 0 \textrm{ in } \mathcal{D}_2.\]
 If $\lambda^* = 0$, then $\psi^* = 0$ in $\overline{\mathcal{D}_1}$.  On the other hand, if $\lambda^* > 0$, then $\psi^* > 0$ in $\mathcal{D}_1$.  
\end{lemma}
\begin{proof}
The interior regularity is obvious: $\widetilde\psi(s)$ is harmonic in each layer, and thus $\psi^*$ is harmonic in $\mathcal{D}\setminus \mathcal{S}$.  Since $\tilde \psi(s) = \lambda(s)-1 < \lambdastar - 1 < 0$ on the top boundary of $\fluidD(s)$, we must have that $\psi_2^* < 0$ on the upper boundary of $\mathcal{D}_2$.  Because it vanishes on $\mathcal{S}$, the strong maximum principle implies that $\psi^* < 0$ in the upper layer.  

If $\lambda^* > 0$, the same argument ensures that $\psi^* > 0$ in ${\mathcal{D}_1}$.   
On the other hand, if $\lambda^* = 0$, then because $\widetilde \psi(s) = \lambda(s)$ on the lower boundary of $\fluidD(s)$, it must be that $\psi^*$ vanishes on both the top and bottom of $\mathcal{D}_1$.  But, by continuity it holds that $\psi_y^* \leq 0$ in $\mathcal{D}_1$, and so we must have $\psi^* = 0$ in $\overline{\mathcal{D}_1}$.
\end{proof}

\begin{lemma}[Lower bound] \label{no contact lemma} In the limit along $\cm^-$, if the interface does not overturn \eqref{no overturning upstream},  then  
\[ \liminf_{s \to -\infty} \lambda(s) > 0.\]
\end{lemma}
\begin{proof}
Seeking a contradiction, suppose that \eqref{no overturning upstream} holds but along some subsequence $\lambda(s) \to 0$.   From Lemma~\ref{elevation limiting problem lemma}, we see that by passing to a further subsequence we have $\bar\psi_1(s) \to 0$ in $C^\varepsilon(\overline{\mathcal{D}})$ for all $\varepsilon \in (0,1)$, and so in view of \eqref{psi_i convergence}, $\nabla \bar \psi_1(s) \to 0$ weakly-$\ast$ in $L^\infty(\mathcal{D})$.  

Let $\xi$ be a smooth nonnegative test function supported in a sufficiently small neighborhood of $\mathcal{S}$ in $\mathcal{I} \times (-1,1)$.  Then we have 
\[ \int_{\widetilde{\fluidS}(s)} \nu \cdot \nabla \bar\psi_1(s) \xi \, dS = \int_{\widetilde{\fluidD}(s)} \nabla \bar\psi_1(s) \cdot \nabla \xi \, dx \, dy \longrightarrow 0 \qquad \textrm{as } s \to \infty, \] 
where $\widetilde{\fluidD}(s)$ is the translated fluid domain determined by $\widetilde \eta$, $dS$ is surface measure, and $\nu$ is the outward unit normal to $\widetilde{\fluidD}_1(s)$.  Since $\nu \cdot \nabla \bar \psi_1(s) = \nu \cdot \nabla \widetilde\psi_1(s)$ is strictly negative on $\fluidS(s)$ and $\xi$ is nonnegative, we infer that, passing to a further subsequence,
\[ \nu \cdot \nabla \widetilde \psi_1(s)|_{y = \widetilde\eta(s)(x)} \longrightarrow 0 \qquad \textrm{a.e. on } \mathcal{I}.\]
But then the Bernoulli condition \eqref{eqn:stream:dynamic} leads to the inequality 
\[  \widetilde \eta(s) > \frac{F^2}{2} +  \frac{F^2 \rho_1 }{2 \jump{\rho}}  |\nabla \widetilde \psi_1(s)|^2 > \frac{F^2}{4} \qquad \textrm{a.e. on } \mathcal{I}, \textrm{ for } s \gg 1, \]
which contradicts the upper bound on $\widetilde \eta(s)$ in \eqref{lower bound eta}.
\end{proof}

\subsubsection*{$C^{1+\varepsilon}$ regularity at regular boundary points and the proof of Theorem~\ref{limit eta theorem}} 

The next (and last) step of the argument is to show that, if overturning does not occur and the interface remains bounded away from the walls, then we can extract a subsequence converging to a limiting wave defined near the stagnation point.  In fact, this will be a solution in the viscosity sense of a two-phase free boundary problem.  Through the classical work of Caffarelli \cite{caffarelli1987harnack1}, we will prove that the interface is smooth enough to apply the Hopf boundary-point lemma, and thus obtain a contradiction since there is a stagnation point.  

For completeness, let us first recall the following definition.
\begin{definition} \label{def viscosity solution} Let $\mathcal{D} \subset \mathbb{R}^2$ be a smooth bounded domain.  We say that $\fbu \in C^0(\mathcal{D})$ is a (weak) viscosity solution to the elliptic free boundary problem 
  \begin{equation}
    \left\{
      \begin{aligned}
        \Delta \fbu & = 0 & \qquad & \textrm{in } \{ \fbu > 0 \}  \\
        \Delta \fbu & = 0 & \qquad & \textrm{in } \operatorname{int}\{ \fbu \leq 0 \}   \\
        \partial_{\invec} \fbu^+ & =  g(\partial_{\invec} \fbu^-, X)  & \qquad & \textrm{on } \partial \{ \fbu > 0 \},
      \end{aligned} \right.
    \label{caffarelli free boundary problem} 
  \end{equation}
provided it satisfies the interior equations in the usual viscosity sense, and the boundary condition is interpreted as follows:
\begin{itemize}
\item
 If at $X_0 \in \partial\{ \fbu > 0 \}$ there is a tangent ball $B \subset \{ \fbu > 0\}$, and $a, b \geq 0$ such that
 \[ \begin{aligned}
 \fbu^+ &\geq  a \left\langle \placeholder - X_0,\,  \invec \right\rangle^+ +o(|\placeholder-X_0|) &\qquad& \textrm{in } B  \\
  \fbu^- &\leq  b \left\langle \placeholder-X_0, \,  \invec \right\rangle^- +o(|\placeholder - X_0|) &\qquad& \textrm{in } B^c,
 \end{aligned} \]
with equality on non-tangential domains, then $a \leq g(b, X_0)$.  Here $\invec$ is normal to $B$ at $X_0$ pointing into $\{ \fbu > 0\}$ and the superscripts $\pm$ denote the positive and negative parts.
\item If at $X_0 \in \partial\{ \fbu > 0\}$ there is a tangent ball $B \subset \operatorname{int}\{ \fbu \leq 0\}$, and $a, b \geq 0$ such that
\[ \begin{aligned}
 \fbu^- &\geq  b \left\langle \placeholder-X_0, \, \invec \right\rangle^- +o(|\placeholder-X_0|) &\qquad& \textrm{in } B  \\
  \fbu^+ &\leq  a \left\langle \placeholder-X_0, \, \invec \right\rangle^+ +o(|\placeholder-X_0|) &\qquad& \textrm{in } B^c,
 \end{aligned} \]
with equality on non-tangential domains and $\invec$ defined as before, then $a \geq g(b, X_0)$.
\end{itemize}
\end{definition}

Note that there are several equivalent characterizations of viscosity solutions; see \cite{caffarelli2005geometric,desilva2019recent} for background and further discussion.  
For our purposes, we need only the following two facts, both of which are well-known consequences of the definition.  First, any classical solution is a viscosity solution, and second, if $\{ \fbu_n \}$ is a sequence of solutions converging uniformly to $\fbu$, then $\fbu$ is a viscosity solution.

Our interest in viscosity solutions lies in the the next theorem, which is essentially the main result of \cite{caffarelli1987harnack1} (with a later refinement in \cite{caffarelli1989harnack2}).  It can also be obtained using the techniques of De Silva \cite{desilva2011free}. 

\begin{theorem}[Caffarelli] \label{lipschitz implies better theorem} Assume that $\fbu \in C^0(\mathcal{D})$ solves the free boundary problem \eqref{caffarelli free boundary problem} in the weak viscosity sense of Definition~\ref{def viscosity solution} and that in a neighborhood of $X_0 \in \partial \{ \fbu > 0\}$, the free boundary is the graph of a Lispchitz function.  Assume that $g$ is smooth; $g(0, \placeholder)$ is uniformly positive near $X_0$ and $ g(\placeholder, X)$ is strictly increasing; and, for some $N \gg 1$, $t \mapsto t^{-N} g(t, X)$ is strictly decreasing.  
  Then in a small neighborhood of $X_0$, the free boundary is $C^{1+\varepsilon}$, for some $\varepsilon \in (0,1)$.  
\end{theorem}

With that in mind, let us now show that if the alternatives in Theorem~\ref{limit eta theorem} do not occur, then we can extract a limiting bore that satisfies the equations locally in a viscosity sense.  In what follows, we suppose that $(x(s), y(s)) \in \fluidS(s)$ is a family of boundary points along which the stagnation limit \eqref{stagnation limit} occurs as $s \to \pm\infty$.  First, we look at the family of depression bores.  

\begin{lemma}[Depression limiting problem] \label{depression limiting lemma} If along $\cm^+$, neither alternative in \eqref{wall contact or overturning} occurs, then we can extract translated subsequences so that in the limit $s \to \infty$, 
\begin{equation}
 \label{depression convergence} 
  \begin{aligned}
    (x(s), y(s)) & \longrightarrow (0, y^*), & & \\
    \lambda(s) & \longrightarrow \lambda^* \in [\lambdastar, 1), & & \\
    \eta(s) & \overset{C^\varepsilon}{\longrightarrow} \eta^* \in \Lip(\overline{\mathcal{I}}) &\quad &\textrm{for all } \varepsilon \in (0,1), \\
    \bar \psi_i(s) & \overset{C^\varepsilon}{\longrightarrow}  \psi_i^* \in \Lip(\overline{\mathcal{D}}) &  \quad & \textrm{for all } \varepsilon \in (0,1), \\
    \nabla \bar \psi_i(s) & \longrightarrow  \nabla \psi_i^* & \quad & \textrm{weak-$\ast$ } L^\infty(\mathcal{D}),
  \end{aligned}
\end{equation}
where $\mathcal{I}$ is a sufficiently small open interval containing $0$ and $\mathcal{D} := \mathcal{I} \times (-\lambda^*, 1-\lambda^*)$.   Moreover,  $\psi^*$ is a viscosity solution of the free boundary problem \eqref{caffarelli free boundary problem} for the function 
  \begin{equation}
    g( \partial_{\invec} \fbu^-, y) := \frac{1}{\rho_1}\left( \rho_2 |\partial_{\invec} \fbu^- |^2 + \frac{2 \jump{\rho}}{F^2} y - \jump{\rho} \right)^{1/2}. \label{def depression g} 
  \end{equation}   
\end{lemma}
\begin{proof}
By hypothesis, $\limsup \lambda(s) < 1$ and so we can extract a subsequence converging to $\lambda^* \in [\lambdastar, 1)$.  Repeatedly translating $x(s)$ to the origin and arguing as in the proof of Lemma~\ref{elevation limiting problem lemma}, the rest of the limits in \eqref{depression convergence} follow easily.  Moreover, we see that 
\[ \psi^* > 0 \textrm{ in } \mathcal{D}_1 := \{ y < \eta^*(x) \} \cap \mathcal{D},  \qquad  \psi^* < 0 \textrm{ in } \mathcal{D}_2 := \{ y > \eta^*(x) \} \cap \mathcal{D},\]
so that $\partial \{ \psi^* > 0\} = \{ y = \eta^*(x)\}$.   By the Hopf lemma, we know that the inward normal derivative $\invec \cdot \nabla \psi_1(s) > 0$ on $\fluidS(s)$.  Observe also that along $\cm^+$, 
\[ \eta(s) < \lim_{x \to \infty} \eta(s) = \lambdastar - \lambda(s) < 0.\]
This justifies rearranging the Bernoulli condition \eqref{eqn:stream:dynamic} as 
\[ \partial_\invec \psi_1(s) = g(\partial_\invec \psi_2(s), y)  \qquad \textrm{on } \fluidS(s) = \partial\{ \psi_1(s) > 0\}\]
with $g$ given by \eqref{def depression g}.  Thus,  $\psi^*$ is the uniform limit of viscosity solution to the free boundary problem, and so it too is a viscosity solution.
\end{proof}

The situation for $\cm^-$ is nearly the same except that we must be wary of the case $y^* = F^2/2$ as this permits double stagnation.  
\begin{lemma}[Elevation limiting problem] \label{elevation limiting lemma} 
If along $\cm^-$, the interface does not overturn, then in the limit $s \to -\infty$ we can extract a convergent subsequence as in \eqref{depression convergence} except that now $\lambda^* \in (0, \lambdastar]$.  If $y^* < F^2/2$, then perhaps shrinking $\mathcal{I}$, we have that $\psi^*$ is a viscosity solution of the free boundary problem \eqref{caffarelli free boundary problem} for $g$ as in \eqref{def depression g}.   If instead $y^* > F^2/2$, then $-\psi^*$ is a viscosity solution taking 
  \begin{equation}
    g(\partial_{\invec} \fbu^-, y) := \frac{1}{\rho_2} \left( \rho_1 |\partial_{\invec} \fbu^-|^2 - \frac{2 \jump{\rho}}{F^2} y + \jump{\rho} \right)^{1/2}. \label{def g small F} 
  \end{equation}  
\end{lemma}
\begin{proof}
 Assuming that overturning \eqref{overturning} does not occur, we have by Lemma~\ref{no contact lemma} that $\liminf \lambda(s) > 0$.  The existence of the subsequential limits \eqref{depression convergence} follows exactly as before.   It remains only to consider the Bernoulli condition.  If $y^* < F^2/2$, then in a neighborhood $(0, y^*)$ we have that the radicand of $g$ in \eqref{def depression g} is strictly positive.  On the other hand, if $y^* > F^2/2$, then we must solve \eqref{eqn:stream:dynamic} for $\partial_\mu \psi_2(s)$ in terms of $y$ and $\partial_\mu \psi_1(s)$.  This results in the function $g$ defined in \eqref{def g small F}.
\end{proof}

Applying Theorem~\ref{lipschitz implies better theorem}, we arrive at the following immediate corollary of these lemmas.

\begin{corollary}[$C^{1+\varepsilon}$ regularity] \label{first regularity lemma} Consider the regularity of the limiting profiles along $\cm^\pm$. 
\begin{enumerate}[label=\rm(\alph*)]
\item \label{depression smooth part} In the setting of Lemma~\ref{depression limiting lemma}, the limiting profile $\eta^*$ along $\cm^+$ is of class $C^{1+\varepsilon}$, for some $\varepsilon \in (0,1)$, on a neighborhood of $0$. 
\item \label{elevation smooth part} The same statement holds in the setting of Lemma~\ref{elevation limiting lemma} for the limiting profile along $\cm^-$ provided $y^* \neq F^2/2$. 
\end{enumerate}
\end{corollary} 

We are at last ready to prove the main result on the limiting behavior of the interface.

\begin{proof}[Proof of Theorem~\ref{limit eta theorem}]
First consider the statement in part~\ref{elevation part}.  Seeking a contradiction, suppose that \eqref{overturning} does not occur.  Then Lemma~\ref{elevation limiting lemma} gives us the existence of a limiting profile $\eta^*$ and stream function $\psi^* = \psi_1^* + \psi_2^*$.  

If $y^* \neq F^2/2$, we have by Corollary~\ref{first regularity lemma}\ref{elevation smooth part}  that $\eta^* \in C^{1+\varepsilon}$ in a neighborhood of $0$, and thus $\psi_1^*$ and $\psi_2^*$ are $C^{1+\varepsilon}$ up to that portion of the boundary.  This is a contradiction:  by construction, $(0,y^*)$ is the limiting stagnation point and so one of $|\nabla \bar\psi_i|$ must vanish there, which cannot be true in view of the Hopf boundary-point lemma.  Note that this requires a sharper version than the classical result; see Lemma~\ref{Hopf lemma}.

On the other hand, if $y^* = F^2/2$, then from the Bernoulli condition we see that \emph{both} phases limit to stagnation at $(0, y^*)$, that is, 
\[ \nabla \psi_1(s)(x(s), y(s)), ~ \nabla \psi_2(s)(x(s), y(s)) \to 0 \qquad \textrm{as } s \to -\infty.  \]
Thus $\{ y < \eta^*\}$ satisfies  neither an interior nor an exterior sphere condition at $(0, y^*)$ as otherwise the Hopf lemma would be violated.   We have therefore proved that either overturning \eqref{overturning} occurs, or else the free boundary becomes singular.  

The argument for part~\ref{depression part} is essentially the same: assuming that \eqref{wall contact or overturning} does not occur, we may apply Lemma~\ref{depression limiting lemma} and Corollary~\ref{first regularity lemma}\ref{depression smooth part} leading to a contradiction with Hopf.
\end{proof}

\section*{Acknowledgments}

The research of RMC is supported in part by the NSF through DMS-1907584.  The research of SW is supported in part by the NSF through DMS-1812436.  

The authors also wish to thank the Erwin Schr\"odinger Institute for Mathematics and Physics, University of Vienna for their hospitality.  A portion of this work was completed during a Research-in-Teams Program generously supported by the ESI.  Finally, the authors thank the anonymous referees whose comments lead to several improvements in the paper.

\appendix

\section{Principal eigenvalues and the spectral condition} \label{principal eigenvalue spec condition appendix}

\subsection{Principal eigenvalues}\label{principal eigenvalue appendix}

In this subsection, we recall some important information about the principal eigenvalue of linear elliptic operators on bounded domains. We will work on the base of the cylinder $\Omega^\prime$ which has boundary components $\Gamma_0^\prime$ and $\Gamma_1^\prime$. Note that from the beginning of Section~\ref{sec_intro} on the setup of the problem, these boundary components are both $C^{k+2+\alpha}$.

Let $\mathfrak{L} : C^{k+2+\alpha}(\overline{\Omega^\prime}) \to C^{k+\alpha}(\overline{\Omega^\prime}) \times C^{k+1+\alpha}(\Gamma_1^\prime)$ be given by 
\begin{equation}
 \label{appendix eigenvalue L} 
  \begin{aligned}
    \mathfrak{L}_1 v & := \mathfrak{a}^{ij}(y) \partial_i \partial_j v + \mathfrak{b}^i(y) \partial_i v + \mathfrak{c}(y) v, \\
    \mathfrak{L}_2 v & := \left(  \mathfrak{d}^i(y) \partial_i v + \mathfrak{e}(y) v\right)|_{\Gamma_1^\prime},
  \end{aligned}
\end{equation}
where the coefficients satisfy 
\[ \mathfrak{a}^{ij} = \mathfrak{a}^{ji}, \quad \mathfrak{a}^{ij}, \mathfrak{b}^i, \mathfrak{c} \in C^{k+\alpha}(\overline{\Omega^\prime}), \quad \mathfrak{d}^i, \mathfrak{e} \in C^{k+1+\alpha}(\Gamma_1^\prime),\]
the interior operator is uniformly elliptic, and the boundary condition on $\Gamma_1^\prime$ is uniformly oblique.  Notice here that the summation is over $1 \leq i, j \leq n-1$ and $\partial_i = \partial_{y_i}$.   One should imagine $\mathfrak{L}$ as standing in for the limiting transversal operator $\limL_\pm^\prime$ elsewhere in the paper.

Call $v \in C^2(\Omega^\prime) \cap C^1(\overline{\Omega^\prime})$ a \emph{supersolution} of $\mathfrak{L}$ provided that 
\begin{equation}
  \label{defn supersoln} \left\{ \begin{aligned}
    \mathfrak{L}_1 v & \leq 0 & \qquad & \textrm{in } \Omega^\prime \\ 
    \mathfrak{L}_2 v & \geq 0 & \qquad & \textrm{on } \Gamma_1^\prime \\
  v & \geq 0 & \qquad & \textrm{on } \Gamma_0^\prime, \end{aligned} \right. 
\end{equation}
and a \emph{strict supersolution} if at least one of these inequalities is strict.  The following theorem ensures the existence of a principal eigenvalue of $\mathfrak{L}$, and states that it is negative precisely when $\mathfrak{L}$ obeys the maximum principle and Hopf boundary-point lemma.  

\begin{theorem} \label{oblique prob eigenvalue theorem} Consider a linear elliptic operator $\mathfrak{L}$ of the form \eqref{appendix eigenvalue L}.  
\begin{enumerate}[label=\rm(\alph*)]
\item \label{oblique existence part} There exists a unique $\sigma \in \mathbb{C}$ such that the spectral problem 
\begin{equation}
  \label{appendix eigenvalue problem}  \left\{ \begin{aligned} 
    \mathfrak{L}_1 \varphi & = \sigma \varphi & \qquad & \textrm{in } \Omega^\prime \\ 
    \mathfrak{L}_2 \varphi & = 0 & \qquad & \textrm{on } \Gamma_1^\prime \\
  \varphi & = 0  & \qquad & \textrm{on } \Gamma_0^\prime, \end{aligned} \right.  
\end{equation}
has a solution $\varphi_0 \in C^{k+2+\alpha}(\overline{\Omega^\prime})$ satisfying 
\[ \varphi_0 > 0 \textrm{ on }  \Omega^\prime \cup \Gamma_1^\prime, \qquad \nu \cdot \nabla \varphi_0 < 0 \textrm{ on } \Gamma_0^\prime.\]
We call $\sigma$ the \emph{principal eigenvalue} of $\mathfrak{L}$ and denote it $\prineigenvalue(\mathfrak{L})$.
\item \label{oblique dominant part} It always holds that $\prineigenvalue(\mathfrak{L})$ is simple and real.  Moreover, for any eigenvalue $\xi$ of $\mathfrak{L}$, we have that $\realpart{\xi} \leq \prineigenvalue(\mathfrak{L})$ with equality if and only if $\xi = \prineigenvalue(\mathfrak{L})$. 
\item \label{oblique classification part} The following are equivalent:
\begin{enumerate}[label=\rm(\roman*)]
\item $\sigma_0(\mathfrak{L}) < 0$;
\item $\mathfrak{L}$ possesses a positive strict supersolution; and 
\item If $v \in C^{2+\alpha}(\overline{\Omega^\prime})$ is a nonzero supersolution, then
\[ v > 0 \quad \textrm{on } \Omega^\prime \cup \Gamma_1^\prime, \qquad \nu \cdot \nabla v < 0 \quad \textrm{on } \Gamma_0^\prime \cap v^{-1}(0).\]  
\end{enumerate}
\end{enumerate}
\end{theorem}
\begin{proof}
Parts~\ref{oblique existence part} and \ref{oblique dominant part} are essentially given by Amann \cite[Theorem 12.1]{amann1983dual}.  We are assuming more smoothness on the domain and coefficients, which by a straightforward elliptic regularity argument allows us to upgrade $\varphi_0$ to $C^{k+2+\alpha}$.  Part~\ref{oblique classification part} is due to L\'opez-G\'omez \cite[Theorem 6.1]{lopezgomez2003classifying}.  
\end{proof}

The next lemma is an obvious consequence of the above theory, recorded for convenience of the reader.

\begin{lemma} \label{continuity sigma0 lemma}
For all $\varepsilon > 0$, there exists $\delta > 0$ such that if $\mathfrak{L}$ and $\tilde{\mathfrak{L}}$ are linear elliptic operators of the form \eqref{appendix eigenvalue L} whose coefficients satisfy 
\[ \| \mathfrak{a}^{ij} - \tilde{\mathfrak{a}}^{ij} \|_{C^{k+\alpha}},\, \| \mathfrak{b}^{i} - \tilde{\mathfrak{b}}^{i} \|_{C^{k+\alpha}}, \, \| \mathfrak{c} - \tilde{\mathfrak{c}} \|_{C^{k+\alpha}}, \, \| \mathfrak{d}^i - \tilde{\mathfrak{d}}^i \|_{C^{k+1+\alpha}}, \, \| \mathfrak{e} - \tilde{\mathfrak{e}} \|_{C^{k+1+\alpha}}  < \delta , \]
then
\[ | \prineigenvalue(\mathfrak{L}) - \prineigenvalue(\tilde{\mathfrak{L}})| < \varepsilon. \]
\end{lemma}
\begin{proof}
Recall from \cite[Section 6]{lopezgomez2003classifying}, say, that the principal eigenvalue  is obtained by applying the Krein--Rutman theorem to the resolvent operator associated to the eigenvalue problem \eqref{appendix eigenvalue problem}.  That is, set 
\[ \mathfrak{R} \colon  f \in C^{k+1+\alpha}(\overline{\Omega^\prime})  \longmapsto  v \in C^{k+2+\alpha}(\overline{\Omega^\prime}) \hookrightarrow C^{k+1+\alpha}(\overline{\Omega^\prime})\]
where $v$ is the unique solution to
\[ \left\{ \begin{aligned} 
\mathfrak{L}_1 v - \sigma v & = f & \qquad & \textrm{in } \Omega^\prime \\
\mathfrak{L}_2 v & = 0 & \qquad &\textrm{on } \Gamma_1^\prime \\
v & = 0 & \qquad & \textrm{on } \Gamma_0^\prime
\end{aligned} \right. \]
and $\sigma > 0$ is chosen large enough so that the above inverse exists.   Then $\prineigenvalue(\mathfrak{L})$ is determined by  $\sigma$ and the spectral radius of $\mathfrak{R}$.  Both of these can be controlled uniformly in terms of the coefficients of $\mathfrak{L}$ according to Schauder theory, and so the lemma follows.    
\end{proof}

\subsection{Spectral degeneracy and the essential spectrum} \label{essential spectrum appendix} 

Recall that if $\mathscr{W}$ and $\mathscr{Z}$ are Banach spaces, a mapping $\mathscr{G} \colon \mathscr{W} \to \mathscr{Z}$ is said to be \emph{locally proper} provided that $\mathscr{G}^{-1}(\mathcal{K}) \cap \mathcal{D}$ is compact for any compact subset $\mathcal{K} \subset \mathscr{Z}$ and closed and bounded set $\mathcal{D} \subset \mathscr{W}$.  In the case of a bounded linear operator, this is equivalent to being semi-Fredholm with a finite-dimensional kernel.  Likewise, for a bounded linear operator $\mathscr{A} \colon \Xspace_\bdd \to \Yspace_\bdd$, we define the \emph{essential spectrum} to be the set
\begin{equation}
  \essspec(\mathscr{A}) := \left\{ \xi \in \mathbb{C} : (\mathscr{A}_1 - \xi, \, \mathscr{A}_2 ) \textrm{ is not locally proper } \Xspace_\bdd \to \Yspace_\bdd  \right\}.  \label{def essential spectrum} 
\end{equation}
In particular, if $0 \not\in \essspec(\mathscr{A})$, then $\mathscr{A}$ is semi-Fredholm with index $\ind\mathscr{A} < +\infty$.

Consider now the linear elliptic operator $\mathscr{L} \colon \Xspace_\bdd \to \Yspace_\bdd$ given by
\begin{equation}
  \label{appendix L operator} \begin{aligned}
    \mathscr{L}_1 v & := a^{ij}(x,y) \partial_i \partial_j v + b^i(x,y) \partial_i v + c(x,y) v, \\
    \mathscr{L}_2 v & := \left(  \beta^i(x,y) \partial_i v + \gamma(x,y) v\right)|_{\Gamma_1},
  \end{aligned} 
\end{equation}
with 
\[ a^{ij} = a^{ji}, \quad a^{ij}, b^i, c  \in C_\bdd^{k+\alpha}(\overline{\Omega}), \quad  \beta_i, \gamma \in C_\bdd^{k+1+\alpha}(\Gamma_1),\]
such that $\mathscr{L}_1$ is uniformly elliptic and $\mathscr{L}_2$ is uniformly oblique.  We assume that these coefficients have well-defined limits
\begin{equation}
  \label{appendix limiting coeff} 
  \begin{aligned}
    &a^{ij}(x, \placeholder) \to a^{ij}_\pm,~b^i(x,\placeholder) \to b^i_\pm, ~ c(x,\placeholder) \to c_\pm \\
    &\beta^i(x,\placeholder) \to \beta_\pm^i,~ \gamma(x,\placeholder) \to \gamma_\pm
  \end{aligned}
  \qquad \textrm{as } x \to \pm\infty
\end{equation}
 where $a_\pm^{ij}, b_\pm^i, c_\pm \in C^{k-2+\alpha}(\overline{\Omega^\prime})$ and $\beta_\pm^i, \gamma_\pm \in C^{k-1+\alpha}(\Gamma_1^\prime)$.  Let $\limL_\pm$ denote the corresponding limiting operator and $\limL_\pm^\prime$ the transversal limiting operator at $x = \pm\infty$.  Note in particular that $\limL_\pm^\prime$ will be of the form \eqref{appendix eigenvalue L}, and hence by Theorem~\ref{oblique prob eigenvalue theorem} it has a principal eigenvalue.
 
The injectivity of these limiting operators is in fact equivalent to the local properness of $\limL$, as the following well-known result shows.  See, for example, \cite{volpert2011book1} or \cite[Lemma~A.7]{wheeler2013solitary}.  We note in passing that the proof of \cite[Lemma~A.5]{wheeler2013solitary} incorrectly cites \cite[Theorem~9.19]{gilbarg2001elliptic}. Thankfully, this can be resolved by substituting \cite[Theorem~5.54]{lieberman2013book}, as was pointed out to us by one of the authors of \cite{lisai2020sqg}.

\begin{lemma}[Local properness] \label{local proper lemma} The elliptic operator $\mathscr{L}$ is locally proper $\Xspace_\bdd \to \Yspace_\bdd$ if and only both $\limL_+$ and $\limL_-$ are injective $\Xspace_\bdd \to \Yspace_\bdd$.
\end{lemma}

It is then possible to characterize the spectral condition \eqref{spectral assumption} through the essential spectrum.

\begin{proposition}[Principal eigenvalues and essential spectrum] \label{essential spec proposition} Consider the linear elliptic operator $\mathscr{L}$ with corresponding limiting transversal operators $\limL_\pm^\prime$.  It holds that 
  \begin{equation}
    \essspec{(\mathscr{L})} \subset \mathbb{C}_-   \quad \textrm{if and only if} \quad \prineigenvalue(\limL_-^\prime),~\prineigenvalue(\limL_+^\prime) < 0,   \label{no ripples iff stable} 
  \end{equation}
  where here $\mathbb{C}_- = \{z \in \mathbb{C} : \operatorname{Re} z < 0 \}$ is the left half-plane.
\end{proposition}
\begin{proof}
For a fixed $m > 0$, consider the periodized cylindrical domain
\[ \Omega^{(m)} :=  \mathbb{T}_m \times \Omega^\prime, \qquad \textrm{where } \mathbb{T}_m := \mathbb{R} / \tfrac{2\pi}{m} \mathbb{Z}. \]
Denote by $\limL_\pm^{(m)}$ the restriction of $\limL_\pm$ to the subspace of $\Xspace_\bdd$ consisting functions which are $2\pi/m$-periodic in $x$.  By classical theory, the spectrum of $\limL_\pm^{(m)}$ consists of discrete eigenvalues with no finite accumulation points.  Moreover, each of these operators has a principal eigenvalue that, abusing notation slightly, we denote 
\[ \prineigenvalue(\limL_\pm^{(m)}) \in \mathbb{R}.\]
By definition, the principal eigenvalue is the unique eigenvalue for which the corresponding eigenfunction is positive on $\Omega^{(m)} \cup \Gamma_1^{(m)}$.     On the other hand, any eigenfunction of  $\limL_\pm^\prime$ is also an eigenfunction for $\limL_\pm^{(m)}$.  Therefore,
  \begin{equation}
    \prineigenvalue(\limL_\pm^{(m)}) = \prineigenvalue(\limL_\pm^\prime) \qquad \textrm{for all } m > 0.\label{periodic eigenvalue fact} 
  \end{equation}

In Lemma~\ref{local proper lemma}, we argued that $\xi \in \essspec(\limL)$ if and only if there exists a nontrivial solution $v_\pm \in \Xspace_\bdd$ (complexified in the usual way) to
\[ \left\{
\begin{aligned} 
\limL_{1\pm} v_\pm & = \xi v_\pm &\qquad &\textrm{in } \Omega \\
\limL_{2\pm} v_\pm & = 0 &\qquad& \textrm{on } \Gamma_1.
\end{aligned} \right.\] 
The coefficients above are independent of $x$, and so via Fourier analysis, it is clear that we can take $v_\pm$ to be either independent of $x$ or else $2\pi/m$-periodic in $x$ for some $m > 0$.  In the former case, $\xi$ will be an eigenvalue of $\limL_\pm^\prime$, whereas in the latter $\xi$ will be in the spectrum of $\limL_\pm^{(m)}$.  From this and \eqref{periodic eigenvalue fact} we conclude that 
\[ \max \left\{ \realpart{\xi} : \xi \in \essspec{(\limL)} \right\} = \max\left\{ \prineigenvalue(\limL_-^\prime), \, \prineigenvalue(\limL_+^\prime) \right\}.\]
The statement in \eqref{no ripples iff stable} follows. 
\end{proof}

\begin{lemma}[Asymptotic invertibility] \label{invertibility lemma}
If $\limL$ satisfies $\prineigenvalue(\limL_-^\prime), \prineigenvalue(\limL_+^\prime) < 0$, then $\limL_-$ and $\limL_+$ are invertible $\Xspace_\bdd \to \Yspace_\bdd$.
\end{lemma} 
\begin{proof}
Consider first the elliptic operator $(\limL_{1\pm}-\sigma, \limL_{2\pm})$ for $\sigma \gg 1$ to be determined.  Let $f \in \Yspace_\bdd$ be given and set $p := n/(1-\alpha)$.  Then clearly we have that 
\[  \sup_{m \in \mathbb{Z}} \left( \| f_1 \|_{W^{k,p}(\Omega_m)} + \| f_2 \|_{W^{k+1-\frac{1}{p}, p}(\Gamma_{1m})}  \right) \lesssim \| f_1 \|_{C^{k+\alpha}(\Omega)} + \| f_2 \|_{C^{k+1+\alpha}(\Gamma_1)} =   \| f \|_{\Yspace},\]
where $\Omega_m := (m, m+1) \times \Omega^\prime$ and $\Gamma_{1m} := (m,m+1) \times \Gamma_1^\prime$. The quantity on the far left-hand corresponds to the $F_\infty$ norm in the notation of \cite{volpert2011book1}.  We may then apply   \cite[Theorem 7.5]{volpert2011book1} to conclude that, for $\sigma$ sufficiently large, there exists a unique $u \in W_\loc^{k+2,p}(\Omega)$ such that 
\[ \left\{ \begin{aligned} 
\limL_{1\pm} u - \sigma u &= f_1 & \qquad & \textrm{in } \Omega \\
\limL_{2\pm} u & = f_2 & \qquad & \textrm{on } \Gamma_1 \\
u & = 0 & \qquad & \textrm{on } \Gamma_0, 
\end{aligned} \right.  \]
with the boundary conditions satisfied in the trace sense.  Moreover, that same result states that $u$ obeys a ``uniformly local'' Sobolev regularity estimate
\[ \sup_{m \in \mathbb{Z}} \| u \|_{W^{k+2,p}(\Omega_m)} \lesssim \sup_{m \in \mathbb{Z}} \left( \| f_1 \|_{W^{k,p}(\Omega_m)} + \| f_2 \|_{W^{k+1-\frac{1}{p}, p}(\Gamma_{1m})}  \right),  \]
which by Morrey's inequality leads to the H\"older norm bound
\[ \| u \|_{C^{k+1+\alpha}(\Omega)} \lesssim \| f \|_{\Yspace}. \]
A straightforward Schauder theory argument allows us to infer the improved local H\"older regularity $u \in \Xspace$.     Taken together with the uniform $C^0$ bounds on $u$ in terms of $\| f \|_{\Yspace}$, this implies that $u \in \Xspace_\bdd$; see, for instance, \cite[Lemma A.1]{wheeler2013solitary}.  

The above reasoning shows that $(\limL_{1\pm} - \sigma, \limL_{2\pm})$ is invertible $\Xspace_\bdd \to \Yspace_\bdd$ if $\sigma \gg 1$.  On the other hand, because $\prineigenvalue(\limL_-^\prime)$, $\prineigenvalue(\limL_+^\prime) < 0$, it also holds that $(\limL_{1\pm} - t\sigma, \limL_{2\pm})$ is semi-Fredholm as a mapping between these spaces for all $t \in [0,1]$.  Continuity of the index then implies $\ind\limL_\pm = 0$.  The proof is now complete as the hypothesis also ensures that $\limL_\pm$ is injective.  
\end{proof}

With these tools in hand, we can now provide the proof of Lemma~\ref{fredholm Xb lemma} on the Fredholm properties of the linearized operator $\F_u(u,\Lambda)$ at a front in $u \in \Xspace_\bdd$.

\begin{proof}[Proof of Lemma~\ref{fredholm Xb lemma}]
  Fix a solution $(u,\Lambda) \in \F^{-1}(0) \cap \genU_\infty$ and let $\limL := \F_u(u,\Lambda)$. Clearly, $\limL$ is of the form \eqref{appendix L operator}, and since $u \in \Xspace_\infty$, the coefficients have well-defined far-field  limits as in \eqref{appendix limiting coeff}, with limiting operators $\limL_\pm$. By hypothesis, the corresponding transversal limiting operators $\limL_\pm'$ satisfy \eqref{spectral assumption} and so Lemma~\ref{invertibility lemma} shows that $\limL_\pm$ is invertible $\Xspace_\bdd \to \Yspace_\bdd$.  In particular, this means that $\limL \colon \Xspace_\bdd \to \Yspace_\bdd$ is semi-Fredholm in light of Lemma~\ref{local proper lemma}. We will calulate its index $\ind\limL < +\infty$ using a pair of homotopies.

  Let $\sigma > 0$ be a constant to be determined, and consider the one-parameter family of operators
  \begin{align*}
    (\limL_1 - t \sigma, \, \limL_2) \colon \Xspace_\bdd \to \Yspace_\bdd \qquad \textrm{for } t \in [0, 1].
  \end{align*}
  Each of these operators is uniformly elliptic and uniformly oblique, and has well-defined limiting transversal operators at $x = \pm\infty$.  As $t\sigma \geq 0$, \eqref{spectral assumption} guarantees that the corresponding principal eigenvalues are all strictly negative. Thus, arguing as above, the entire family is semi-Fredholm. By the continuity of the index, the operators at the endpoints $t=0$ and $t=1$ therefore have the same index, that is $\ind\limL=\ind(\limL_1-\sigma,\limL_2)$.

  Next we consider the one-parameter family of operators
  \begin{align}
    \label{homotopy with neumann laplacian}
    \big(t\Delta + (1-t)\limL_1 - \sigma, \, t\partial_\nu + (1-t)\limL_2\big) \colon \Xspace_\bdd \to \Yspace_\bdd \qquad \textrm{for } t \in [0, 1]
  \end{align}
  where $\partial_\nu = \nu \cdot \nabla$ is the outward pointing normal derivative. We easily check that each of these operators is uniformly elliptic and uniformly oblique. By Lemma~\ref{continuity sigma0 lemma}, the principal eigenvalues of the limiting transveral operators
    \begin{align}
      \label{principal eigenvalues for laplacian homotopy}
      \sigma_0\big(t\Delta' + (1-t)\limL_1' - \sigma, \, t\partial_\nu + (1-t)\limL_2'\big) 
      = \sigma_0\big(t\Delta' + (1-t)\limL_1', \, t\partial_\nu + (1-t)\limL_2'\big) - \sigma
    \end{align}
  depend continuously on $t \in [0,1]$. Choosing $\sigma > 0$ sufficiently large, we can therefore guarantee that the principal eigenvalues \eqref{principal eigenvalues for laplacian homotopy} are strictly negative, and hence that the entire family of operators \eqref{homotopy with neumann laplacian} is semi-Fredholm. Appealing to the continuity of the index for a second time, we deduce that 
  \begin{align*}
    \ind \limL = \ind (\limL_1 - \sigma,\limL_2) = \ind (\Delta-\sigma,\partial_\nu) = 0,
  \end{align*}
  where the last equality follows, for instance, from Lemma~\ref{invertibility lemma}.
\end{proof}

\subsection{Principal eigenvalues and spectral degeneracy for transmission problems}\label{principal eigenvalue appendix transmission} 

For our application to internal waves, we require analogous results for transmission boundary conditions. We start with the problem posed on the transversal domain $\Omega^\prime = \Omega^\prime_1 \cup \Omega^\prime_2$ which satisfies $\Omega^\prime_1 \cap \Omega^\prime_2 = \emptyset$, $\partial\Omega^\prime_1 \cap \partial\Omega^\prime_2 = \Gamma^\prime_1$, and $\Gamma^\prime_0 = \partial \Omega^\prime \backslash \Gamma^\prime_1$. The operator in \eqref{appendix eigenvalue L} now becomes
\begin{equation}
 \label{appendix transmission L} 
  \begin{aligned}
    \mathfrak{L}_1 v & := \partial_i(\mathfrak{a}^{ij}(y) \partial_j v + \mathfrak{a}^i(y) v) + \mathfrak{b}^i(y) \partial_i v + \mathfrak{c}(y) v, \\
    \mathfrak{L}_2 v & := \jump{-\nu_i \mathfrak{a}^{ij}(y) \partial_j v} + \jump{-\nu_i \mathfrak{a}^{i}(y)} v|_{\Gamma_1^\prime},
  \end{aligned}
\end{equation}
where $\nu = (\nu_1, \ldots, \nu_{n-1})$ is the normal vector field on $\Gamma^\prime_1$ pointing outward from $\Omega^\prime_1$. The coefficients satisfy the following regularity conditions
\begin{equation*}
\mathfrak{a}^{ij} = \mathfrak{a}^{ji}, \quad \mathfrak{a}^{ij}, \mathfrak{a}^i \in C^{k+1+\alpha}(\overline{\Omega_1^\prime}) \cap C^{k+1+\alpha}(\overline{\Omega_2^\prime}), \quad  \mathfrak{b}^i, \mathfrak{c} \in C^{k+\alpha}(\overline{\Omega_1^\prime}) \cap C^{k+\alpha}(\overline{\Omega_2^\prime}).
\end{equation*}
Therefore, from the uniform ellipticity of $\mathfrak{L}$, we know that $\nu_j\mathfrak{a}^{ij}(y)$ is an outward-pointing vector field for $\Omega^\prime_1$ along $\Gamma_1^\prime$.

We define a supersolution of $\mathfrak{L}$ in \eqref{appendix transmission L} to be a function $v\in C(\overline{\Omega^\prime}) \cap C^1(\overline{\Omega^\prime_i}) \cap C^2(\Omega^\prime_i)$ for $i = 1, 2$ satisfying \eqref{defn supersoln}. As before, $v$ is called a strict supersolution if any of these inequalities are strict. 

Since the boundary components of $\Omega^\prime$ are $C^{k+2+\alpha}$, it is straightforward to verify that when $\mathfrak{c} \le 0$ then $\mathfrak{L}$ satisfies the maximum principle, and the following Hopf's lemma also holds true.
\begin{lemma}\label{lem hopf}
Let $\mathfrak{c}\le 0$ and $v \in C(\overline{\Omega^\prime}) \cap C^2(\Omega^\prime_1) \cap C^2(\Omega^\prime_2)$ be nonconstant and satisfy
\[
\mathfrak{L}_1 v \ge 0 \qquad \text{and} \qquad v \le m \qquad \text{in }\quad \Omega
\]
for some $m\ge 0$, where $\mathfrak{L}$ is given in \eqref{appendix transmission L}. Suppose that $u = m$ at some $x_0\in \partial\Omega^\prime$. Then
\begin{enumerate}[label=\rm(\alph*)]
\item if $x_0 \in \Gamma^\prime_0$ then for any outward-pointing vector $\xi$ we have $\xi \cdot \nabla v(x_0) > 0$;
\item if $x_0 \in \Gamma^\prime_1$ then for any outward-pointing vector $\xi$ of $\Omega^\prime_1$ we have $\jump{\xi \cdot \nabla v (x_0)} < 0$.
\end{enumerate}
\end{lemma}

With the help of the above discussion, one can adapt the argument in \cite{lopezgomez2003classifying} to reproduce Theorem \ref{oblique prob eigenvalue theorem} in the setting of transmission problems. This is the content of the next theorem which, though straightforward, appears to be new. 
\begin{corollary}[Principal eigenvalues for transmission problems]\label{cor pe transmission}
The conclusions of Theorem \ref{oblique prob eigenvalue theorem} holds for the elliptic transmission operator $\mathfrak{L}$ defined in \eqref{appendix transmission L}, with obvious modification on the regularity of functions. Moreover, the principal eigenvalue $\prineigenvalue(\mathfrak{L})$ depends continuously on the coefficients in the same form as Lemma \ref{continuity sigma0 lemma}.
\end{corollary}
\begin{proof}
The statements corresponding to Theorem~\ref{oblique prob eigenvalue theorem}\ref{oblique existence part} and \ref{oblique dominant part} still follow largely from Amann \cite{amann1983dual} with a few small adjustments.   Observe first that the  $C^{2+\alpha}$ regularity of the domain ensures the existence of a function $\psi \in C(\overline{\Omega^\prime}) \cap C^{2+\alpha}(\overline{\Omega^\prime_i})$ and a constant $\gamma > 0$ such that 
\[
\jump{-\nu_j \mathfrak{a}^{ij}(y) \partial_i \psi} \ge \gamma \qquad \textrm{on } \Gamma_1^\prime.
\]
Indeed, in a neighborhood $\Gamma_1^\prime$, it works to take $\psi := \text{dist}( \placeholder, \Gamma^\prime_1)$, then extend smoothly to $\overline{\Omega}$.  The unique solvability of the spectral problem \eqref{appendix eigenvalue problem} for the transmission operator can now be inferred from the existence of $\psi$, \cite[Theorem 1.2]{le2018transmission}, and an argument similar to \cite[Theorem 4.1]{lopezgomez2003classifying}.

Next consider part~\ref{oblique classification part}.  One can then easily check that 
\begin{equation}\label{dividing supersolution}
\begin{split}
h \text{ is a strictly positive supersolution of } \mathfrak{L} \quad \Rightarrow \quad &
 \begin{aligned}
 &\text{if } v \text{ is a supersolution of } \mathfrak{L}, \\ & \text{then } {v \over h} \text{ is a  supersolution of $\mathfrak{L}^h$}, \end{aligned}
\end{split}
\end{equation}
where $\mathfrak{L}^h$ is the operator defined by
\[
\begin{aligned} 
\mathfrak{L}_1^h w &:= \partial_i(\mathfrak{a}^{ij}(y) \partial_j w + \mathfrak{a}^i(y) w) + \left( \mathfrak{b}^i(y) + {2\over h} \mathfrak{a}^{ij}(y)\partial_j h \right) \partial_i w + {\mathfrak{L}_1 h \over h} w \\
\mathfrak{L}_2^h w & := \jump{-\nu_j \mathfrak{a}^{ij}(y) \partial_i w} + \left( {\mathfrak{L}_2 h \over h} w\right)|_{\Gamma_1^\prime}.
\end{aligned}
\]
Indeed, if we take $w := v/h$, then $w \in C(\overline{\Omega^\prime}) \cap C^1(\overline{\Omega^\prime_i}) \cap C^2(\Omega^\prime_i)$, and $\mathfrak{L} v = h \mathfrak{L}^h w$.
Since $h > 0$ and is a supersolution of $\mathfrak{L}$, we have
\begin{align*}
\mathfrak{L}_1^h w \le 0, \quad {\mathfrak{L}_1 h \over h} \le 0 \  \text{ in } \Omega^\prime; \qquad \qquad
\mathfrak{L}_2^h w \ge 0, \quad \frac{\mathfrak{L}_2 h}{h} \ge 0 \ \text{ on } \Gamma^\prime_1,
\end{align*}
implying that $w$ is a supersolution of $\mathfrak{L}^h$.

Now, it follows from \eqref{dividing supersolution} and Lemma \ref{lem hopf} that the classification result \cite[Theorem 3.1]{lopezgomez2003classifying} generalizes to the transmission case. Together with the regularity assumption on the domain, this further implies a result as in \cite[Theorem 5.2]{lopezgomez2003classifying}, and so ultimately we see that the conclusions of \cite[Theorem 6.2]{lopezgomez2003classifying} hold for the transmission problem.  This leads precisely to the statements in \ref{oblique classification part}. Finally, the continuity of $\sigma_0(\mathfrak{L})$ is a consequence of Schauder estimates for transmission problems (see, for example,  \cite[Theorem 1.2]{le2018transmission}) and the Krein--Rutman theorem.
\end{proof}

Having established the the theory for the principal eigenvalues for the transversal problem, next we consider the problem on the full domain $\Omega = \Omega_1 \cup \Omega_2$, where recall that $\Omega_1 \cap \Omega_2 = \emptyset$, $\partial\Omega_1 \cap \partial\Omega_2 = \Gamma_1$, and $\Gamma_0 = \partial \Omega \backslash \Gamma_1$. The linear operator $\limL \colon \Xspace_\bdd \to \Yspace_\bdd$ is given by
\begin{equation}
  \label{appendix L operator transmission} \begin{aligned}
    \limL_1 v & := \partial_i \left(a^{ij}(x,y) \partial_j v + a^i(x,y) v\right) + b^i(x,y) \partial_i v + c(x,y) v, \\
    \limL_2 v & := \jump{-\nu_i a^{ij} \partial_j v} + \jump{-\nu_i a^i} v,
  \end{aligned} 
\end{equation}
with 
\[ a^{ij} = a^{ji}, \quad a^{ij}, a^i \in C_\bdd^{k+1+\alpha}(\overline{\Omega_1}) \cap C^{k+1+\alpha}(\overline{\Omega_2}), \quad  b^i, c  \in C_\bdd^{k+\alpha}(\overline{\Omega_1}) \cap C_\bdd^{k+\alpha}(\overline{\Omega_2}), \]
such that $\mathscr{L}_1$ is uniformly elliptic. Here the spaces $\Xspace_\bdd$ and $\Yspace_\bdd$ are naturally defined as in Section~\ref{sec:trans}. The limits of the coefficients are
\begin{equation}
  \label{appendix limiting coeff transmission} 
  \begin{aligned}
    a^{ij}(x, \placeholder) \to a^{ij}_\pm,~a^{i}(x, \placeholder) \to a^{i}_\pm,~b^i(x,\placeholder) \to b^i_\pm, ~ c(x,\placeholder) \to c_\pm \end{aligned}
  \qquad \textrm{as } x \to \pm\infty
\end{equation}
 where $a_\pm^{ij}, a_\pm^{i} \in C^{k-1+\alpha}(\overline{\Omega^\prime_1}) \cap C^{k-1+\alpha}(\overline{\Omega^\prime_2})$, $b_\pm^i, c_\pm \in C^{k-2+\alpha}(\overline{\Omega^\prime_1}) \cap C^{k-2+\alpha}(\overline{\Omega^\prime}_2)$. Following Section~\ref{essential spectrum appendix}, we introduce the limiting operator $\limL_\pm$, and the transversal limiting operator $\limL^\prime_\pm$ which is of the form \eqref{appendix transmission L}. From Corollary \ref{cor pe transmission} we know that $\prineigenvalue(\mathscr{L}^\prime_\pm)$ exist. 
 
Note that Lemma \ref{local proper lemma} and Proposition \ref{essential spec proposition} still hold for $\mathscr{L}$ in \eqref{appendix L operator transmission}. Moreover, the same Fredholm property of $\mathscr{F}_u(u, \Lambda)$ for a transmission problem $\mathscr{F}$ can be proved provided the asymptotic invertibility of $\mathscr{L}$.
\begin{lemma}[Asymptotic invertibility for transmission operator] \label{invertibility lemma transmission}
Let $\limL$ be given as in \eqref{appendix L operator transmission}. If $\limL$ satisfies $\prineigenvalue(\limL_-^\prime), \prineigenvalue(\limL_+^\prime) < 0$, then $\limL_-$ and $\limL_+$ are invertible $\Xspace_\bdd \to \Yspace_\bdd$.
\end{lemma} 
\begin{proof}
Following the same idea as in Lemma \ref{invertibility lemma}, for $\sigma\gg 1$, we consider the solvability of 
\[ \left\{ \begin{aligned} 
\limL_{1\pm} u - \sigma u &= f_1 & \qquad & \textrm{in } \Omega \\
\limL_{2\pm} u & = f_2 & \qquad &  \textrm{on } \Gamma_1\\
u & = 0 & \qquad & \textrm{on } \Gamma_0. 
\end{aligned} \right.  \]
Note that one may find $g = (g^1, \ldots, g^n) \in \left( C_\bdd^{k+1+\alpha}(\overline{\Omega_1}) \cap C^{k+1+\alpha}(\overline{\Omega_2}) \right)^n$ with $\jump{-\nu_i g_i} = f_2$, and thus the above problem can be written as
\begin{equation}\label{new transmission} \left\{ \begin{aligned} 
\limL_{1\pm} u - \sigma u &= f + \partial_i g^i & \qquad & \textrm{in } \Omega \cup \Gamma_1 \\
\limL_{2\pm} u & = \jump{-\nu_i g^i} & \qquad &\textrm{on } \Gamma_1  \\
u & = 0 & \qquad & \textrm{on } \Gamma_0, 
\end{aligned} \right.  \end{equation}
where $f = f_1 - \partial_i g^i$. The advantage of \eqref{new transmission} is that it suggests the appropriate weak formulation of the problem. More precisely, let $p := n/(1-\alpha)$, and call $u\in W_0^{1,p}(\Omega\cup\Gamma_1)$  a weak solution to \eqref{new transmission} provided it satisfies
\[
\int_\Omega \left( -a^{ij}\partial_i u \partial_j \varphi - a^j u \partial_j \varphi + b^i \partial_x u \varphi + cu \varphi \right)\,dx \,dy = \int_\Omega \left( f \varphi - g^j \partial_j \varphi \right)\,dx \, dy
\]
for any $\varphi \in W_0^{1,p'}(\Omega\cup\Gamma_1)$. Defining $\Omega_m$ and $\Gamma_{1m}$ as in the proof of Lemma~\ref{invertibility lemma}, the local solvability of \eqref{new transmission} in $W^{1,p}(\Omega_m)$ for $\sigma$ sufficiently large can thus be obtained from \cite[Theorem 5]{dong2010partialBMO}. Following the strategy of \cite[Theorem 7.5]{volpert2011book1}, we conclude that there exists a unique solution $u \in W_\loc^{1,p}(\Omega)$ of \eqref{new transmission}. Using a standard partition of unity argument and local coordinate charts, it suffices to assume that $\Gamma_{1m}$ lies in a hyperplane. Letting $\tau$ be a tangent vector and differentiating the equation, one easily obtains  $W^{1,2}$ estimates for  $\tau \cdot \nabla u$.  Theorem \ref{thm meyers} then allows us to control $\tau \cdot \nabla u$ in $W^{1,p}$, and so we have the $W^{2,p}$ bound: 
\[ 
\| u \|_{W^{2,p}(\tilde\Omega^i_m)} \lesssim \| f_1 \|_{W^{1,p}(\Omega^i_m)} + \| f_2 \|_{W^{1-\frac{1}{p}, p}(\Gamma_{1m})} + \|u\|_{W^{1,2}(\Omega_m)},  \]
where $\Omega^i_m = (m, m+1) \times \Omega_i^\prime$ and $\tilde\Omega^i_m = (m+{1\over 4}, m+{3\over 4}) \times \Omega_i^\prime$, for $i = 1, 2$. The last term on the right-hand side can be dropped by the $W^{1,p}$ estimate of \cite[Theorem 5]{dong2010partialBMO}. Repeating this process yields
\[ 
\sup_{m \in \mathbb{Z}} \| u \|_{W^{k+2,p}(\tilde\Omega^i_m)} \lesssim \sup_{m \in \mathbb{Z}} \left( \| f_1 \|_{W^{k,p}(\Omega^i_m)} + \| f_2 \|_{W^{k+1-\frac{1}{p}, p}(\Gamma_{1m})} \right).  \]
The rest of the argument follows the same way as in the proof of Lemma \ref{invertibility lemma}.
\end{proof}

\section{Quoted results} \label{quoted results appendix}

Analytic global bifurcation theory was first introduced by Dancer \cite{dancer1973bifurcation,dancer1973globalstructure} in the late 1970s, and then refined and popularized by Buffoni and Toland \cite{buffoni2003analytic}.  As mentioned in the introduction, these results were developed with an eye towards problems on bounded domains, and thus took as given the fact that the nonlinear operator is Fredholm index $0$ and that the solution set is locally compact. In \cite{chen2018existence}, the authors offer a variant of the classical theory that removes those assumptions at the cost of additional alternatives along the global curve.   Though this was stated as a global continuation theorem similar to Theorem~\ref{general global bifurcation theorem}, it can be easily reconfigured as the following global implicit function theorem.

\begin{theorem} \label{homoclinic global ift} 
Let $\mathscr{W}$ and $\mathscr{Z}$ be Banach spaces, $\mathcal{W} \subset \mathscr{W} \times \mathbb{R}$ an open set containing a point $(w_0, \lambda_0)$. Suppose that  $\mathscr{G} \colon \mathcal{W} \to \mathscr{Z}$ is real analytic and satisfies 
  \begin{equation}
    \mathscr{G}(w_0, \lambda_0) = 0, \qquad \mathscr{G}_w(w_0, \lambda_0) \colon \mathscr{W} \to \mathscr{Z} \quad \textrm{is an isomorphism}.  \label{homoclinic global ift assumptions} 
  \end{equation}
 Then there exist a curve $\mathscr{K}$ that admits the global $C^0$ parameterization  
 \[ \mathscr{K} := \left\{ (w(s), \lambda(s)) : s \in \mathbb{R}  \right \} \subset \mathscr{G}^{-1}(0) \cap \mathcal{W},\]
  and satisfies the following.  
  \begin{enumerate}[label=\rm(\alph*)]
  \item \label{K well behaved} At each $s \in \mathbb{R}$, the linearized operator $\mathscr{G}_w(w(s), \lambda(s)) \colon \mathscr{W} \to \mathscr{Z}$ is Fredholm index $0$.
  \item \label{K alternatives} One of the following alternatives holds as $s \to \infty$ and $s \to -\infty$.
    \begin{enumerate}[label=\rm(A\arabic*$'$)]
    \item  \label{K blowup alternative}
      \textup{(Blowup)}  The quantity 
      \begin{align}
        \label{K blowup}
        N(s):= \n{w(s)}_{\mathscr{W}} + |\lambda(s)| +\frac 1{\dist((w(s),\lambda(s)), \, \dell \mathcal{W})}
      \end{align}
    \item \label{K loss of compactness alternative} \textup{(Loss of compactness)} There exists a sequence $s_n \to \pm\infty$ with $\sup_n N(s_n) < \infty$, but $( w(s_n), \lambda(s_n) )$ has no convergent subsequence in $\mathscr{W} \times \mathbb{R}$.  
      \item \label{K loss of fredholmness alternative} \textup{(Loss of Fredholmness)}    There exists a sequence $s_n \to \pm\infty$
        with $\sup_{n} N(s_n) < \infty$ and so that $(w(s_n), \lambda(s_n)) \to (w_*, \lambda_*) \in \mathcal{W}$ in $\mathscr{W} \times \mathbb{R}$, however $\mathscr{G}_w(w_*, \lambda_*)$ is not Fredholm index $0$.  
   \item \label{K loop} \textup{(Closed loop)} There exists $T > 0$ such that $(w(s+T), \lambda(s+T)) = (w(s), \lambda(s))$ for all $s \in (0,\infty)$. 
        \end{enumerate}
          \item \label{K reparam} Near each point $(w(s_0),\lambda(s_0)) \in \mathscr{K}$, we can locally reparametrize $\mathscr{K}$ so that $s\mapsto (w(s),\lambda(s))$ is real analytic.
 \item \label{K maximal part} The curve $\mathscr{K}$ is maximal in the sense that, if $\mathscr{J} \subset \mathscr{G}^{-1}(0) \cap \mathcal{W}$ is a locally real-analytic curve containing $(w_0,\lambda_0)$ and along which  $\mathscr G_w$ is Fredholm index 0, then $\mathscr{J} \subset \mathscr{K}$. 
  \end{enumerate}
\end{theorem}
\begin{proof}
This result follows from a straightforward adaptation of \cite[Theorem 6.1]{chen2018existence}, which is in turn based on \cite[Theorem 9.1.1]{buffoni2003analytic}, and so we only provide a sketch of the details that differ.   Following \cite{buffoni2003analytic}, we call a (maximal) connected component of the set
\[ \mathcal{A} := \left\{ (w, \lambda) \in \mathcal{W} : \mathscr{G}(w,\lambda) = 0,~ \mathscr{G}_w(w,\Lambda) \textrm{ is an isomorphism } \mathscr{W} \to \mathscr{Z}  \right\} \] a \emph{distinguished arc}.  By \eqref{homoclinic global ift assumptions} and the analytic implicit function theorem, there exists a distinguished arc $A_0$ containing $(w_0,\lambda_0)$.  Moreover, it admits the parameterization
\[ A_0 = \left\{ (w(s), \lambda(s)) : s \in (-1,1) \right\},\]
with $s \mapsto (w(s),\lambda(s))$ real analytic $(-1,1) \to \mathscr{W} \times \mathbb{R}$, and $(w(0), \lambda(0)) = (w_0, \lambda_0)$.   Consider now the limit $s \to 1$.  Through the same argument as in \cite[Theorem 6.1]{chen2018existence}, we conclude that one of the alternatives \ref{K blowup alternative}--\ref{K loop} occurs, or else $A_0$ connects to another distinguished arc $A_1$.  In the former case, taking $\mathscr{K} = A_0$ completes the proof.  In the latter, we can continue inductively by applying the same logic to (a suitably reparameterized) $A_0 \cup A_1$, and so on.  From here the proof is exactly as before, except there is the additional possibility of a closed loop since $(w_0, \lambda_0)$ is in the interior of $\mathcal{W}$.  This gives us the alternatives claimed in part~\ref{K alternatives}.  Because $\mathscr{K}$ is connected, the above construction also shows that part~\ref{K well behaved} holds (in fact, the linearized operator is invertible except at the endpoints of the distinguished arcs).  In the interior of a distinguished arc, the analytic reparameterization asked for in part~\ref{K reparam} is an immediate consequence of the analytic implicit function theorem.  At an endpoint where two arcs meet, it follows from a deep result characterizing the structure of analytic varieties; see \cite[Chapter 7]{buffoni2003analytic} and the proof of \cite[Theorem 9.1.1(d)]{buffoni2003analytic}.  Likewise, the maximality claimed in part~\ref{K maximal part} is a consequence of the construction and in particular the uniqueness of analytic continuation.  
\end{proof}

The next result we quote is the celebrated Hopf boundary-point lemma, which relates the geometric properties of the boundary to the non-degeneracy of solutions to elliptic PDE set there. In its classical form, the regularity requirement on the domain amounts to an interior ball condition \cite{hopf1952remark,oleinik1952properties}, but this can be relaxed to $C^{1+\alpha}$ \cite{giraud1933problemes} and below.  
\begin{lemma}[Hopf boundary-point lemma]\label{Hopf lemma}
  Let $\Omega \subset \mathbb{R}^n$ be a connected, open set (possibly unbounded) and consider the second-order operator $\mathscr{L}$ given by
  \begin{equation}
    \mathscr{L} := a^{ij}(x) \partial_i \partial_j +  b^i(x) \partial_i + c(x), \label{appendix: def L} 
  \end{equation}
  where we are using the summation convention with $\partial_i := \partial_{x_i}$,  the coefficients satisfy
  \[ a^{ij} = a^{ji},  \quad a^{ij}, \, b^i, \, c \in L^\infty(\Omega),\]
  and $\mathscr L$ is uniformly elliptic. 
  Let $u \in C^2(\Omega) \cap C^0(\overline{\Omega})$ be a classical solution of $\mathscr{L}u = 0$ in $\Omega$.
  \begin{enumerate}[label=\rm(\roman*)]
\item \label{hopf ball} {\rm (Interior ball)} Suppose that $u$ attains its maximum value on $\overline{\Omega}$ at a point $x_0 \in \partial \Omega$ for which there exists an open ball $B \subset \Omega$ with $\overline{B} \cap \partial\Omega = \{ x_0 \}$.  Assume that either $c \leq 0$ in $\Omega$, or else $u(x_0) = 0$.  Let $\nu$ be the outward unit normal to $\Omega$ at $x_0$. Then $u$ is a constant function or 
    \[ \liminf_{t\to 0^+} {u(x_0 + t \mu) - u(x_0) \over t} < 0,\]
    where $\mu$ is an arbitrary vector such that $\mu \cdot \nu < 0$.
\item \label{hopf holder} {\rm ($C^{1+\alpha}$ boundary)} Suppose that $\Omega$ is $C^{1+\alpha}$ for some $\alpha \in (0,1)$ and $u$ attains its maximum value on $\overline{\Omega}$ at a point $x_0 \in \partial \Omega$. Assume that either $c \leq 0$ in $\Omega$, or else $u(x_0) = 0$. Then the same result as in \ref{hopf ball} holds.
\end{enumerate}
\end{lemma}
\begin{remark}\label{remark hopf}
In particular, if $u \in C^1(\Omega \cup \{x_0\})$, then the above result states $\nu \cdot \nabla u(x_0) > 0$, for any outward pointing vector $\nu$.  
\end{remark}

Finally, for completeness, we recall the following standard fact about bordering of Fredholm operators; see, for example, \cite[Lemma 2.3]{beyn1990numerical}. 

\begin{lemma}[Fredholm bordering] \label{bordering lemma}  Let $\mathscr{W}$ and $\mathscr{Z}$ be Banach spaces and suppose that $\mathscr{A} \colon \mathscr{W} \to \mathscr{Z}$, $\mathscr{B} \colon \mathbb{R}^m \to \mathscr{Z}$, $\mathcal{C} \colon \mathscr{W} \to \mathbb{R}^n$, and $\mathcal{D} \colon \mathbb{R}^m \to \mathbb{R}^n$ are bounded linear mappings.  If, in addition, $\mathscr{A}$ is Fredholm index $k$, then the operator matrix 
\[ \begin{pmatrix} \mathscr{A} & \mathscr{B} \\ \mathcal{C} & \mathcal{D} \end{pmatrix} \colon \mathscr{W} \times \mathbb{R}^m \to \mathscr{Z} \times \mathbb{R}^n \]
is Fredholm with index $k + m - n$.
\end{lemma}

\bibliographystyle{siam}
\bibliography{projectdescription}

\end{document}